\documentclass{siamart220329}
\usepackage{graphicx}
\usepackage{amsmath,amssymb}
\usepackage{caption}
\usepackage{subcaption}
\usepackage{cite}
\usepackage{multicol}
\usepackage{multirow}
\usepackage{array}
\usepackage{cleveref}
\usepackage{enumerate}  
\usepackage{mathtools}
\usepackage{url}
\usepackage{bm}
\usepackage{xfrac}
\usepackage{cancel}
\allowdisplaybreaks
\def\W{\Omega}

\def\dt{\Delta t}
\def\dx{\mathrm{d}}

\def \mpM{\mathpzc{M}}

\def\dof{\textrm{DOF}}
\newcommand\vbr[1]{\langle{#1}\rangle_v}
\newcommand\Mip[2]{\langle{#1},{#2}\rangle_{\mpM}}
\DeclareMathAlphabet{\mathpzc}{OT1}{pzc}{m}{it}

\def \mA{\mathcal{A}}
\def \mB{\mathcal{B}}

\def \mE{\mathcal{E}}

\def \mL{\mathcal{L}}

\def \mN{\mathcal{N}}

\def \mR{\mathcal{R}}

\def \bbR{\mathbb{R}}

\def \bmrho{{\bm\rho}}
\def \bmq{{\bm q}}
\def \bmn{{\bm n}}
\def \bmeta{{\bm\eta}}

\def \bmmu{{\bm\mu}}
\def \bme{{\bm e}}
\def \bms{{\bm s}}
\def \bmF{{\bm F}}

\def \tSI{{\textrm{SI}}}
\def \tHL{{\textrm{HL}}}
\def \mtHL{{\mathrm{HL}}}
\def \tMM{{\textrm{MM-HL}}}
\def \tMMSI{{\textrm{MM-L}}}

\def \Vxhc{[V_{x,h}]^3}
\def \Vhat{\widehat{V}_h}

\def \EIx{\mathrm{E}_{x,h}^{\mathrm{I}}}

\def \lavg{\{\!\!\{}
\def \ravg{\}\!\!\}}
\def \ljmp{[\![}
\def \rjmp{]\!]}
\def \ledge{\big<\!\!\big<}
\def \redge{\big>\!\!\big>}

\newtheorem{prop}{Proposition}
\crefname{prop}{proposition}{propositions}
\newtheorem{remark}{Remark}

\numberwithin{figure}{subsection}
\numberwithin{remark}{section}
\numberwithin{table}{subsection}
\numberwithin{equation}{section}
\numberwithin{prop}{section}

\title{On high-order/low-order and micro-macro methods for implicit time-stepping of the BGK model\thanks{Notice: This manuscript has been authored by UT-Battelle, LLC under Contract No. DE-AC05-00OR22725 with the U.S. Department of Energy.  The publisher, by accepting the article for publication, acknowledges that the U.S. Government retains a non-exclusive, paid up, irrevocable, world-wide license to publish or reproduce the published form of the manuscript, or allow others to do so, for U.S. Government purposes. The DOE will provide public access to these results in accordance with the DOE Public Access Plan (\url{http://energy.gov/downloads/doe-public-access-plan}).}}
\author{Cory Hauck, Paul Laiu, and Stefan Schnake}

\author{Cory D.\ Hauck\thanks{Mathematics in Computation Section, Computer Science and Mathematics Division, Oak Ridge National Laboratory, Oak Ridge, TN 37831, USA and Mathematics Department, University of Tennessee, Knoxville, TN 37996, USA (\email{hauckc@ornl.gov}).} 
\and 
M.~Paul Laiu\thanks{Mathematics in Computation Section, Computer Science and Mathematics Division, Oak Ridge National Laboratory, Oak Ridge, TN 37831, USA (\email{laiump@ornl.gov}, \email{schnakesr@ornl.gov}).}
\and Stefan R.\ Schnake\footnotemark[3]}

\begin{document}

\maketitle

\begin{abstract}
    In this paper, a high-order/low-order (HOLO) method is combined with a micro-macro (MM) decomposition to accelerate iterative solvers in fully implicit time-stepping of the BGK equation for gas dynamics.
    The MM formulation represents a kinetic distribution as the sum of a local Maxwellian and a perturbation.
    In highly collisional regimes, the perturbation away from initial and boundary layers is small and can be compressed to reduce the overall storage cost of the distribution.
    The convergence behavior of the MM methods, the usual HOLO method, and the standard source iteration method is analyzed on a linear BGK model.
    Both the HOLO and MM methods are implemented using a discontinuous Galerkin (DG) discretization in phase space, which naturally preserves the consistency between high- and low-order models required by the HOLO approach.
    The accuracy and performance of these methods are compared on the Sod shock tube problem and a sudden wall heating boundary layer problem.
    Overall, the results demonstrate the robustness of the MM and HOLO approaches and illustrate the compression benefits enabled by the MM formulation when the kinetic distribution is near equilibrium.
\end{abstract}


\section{Introduction}\label{sect:intro}

The Bhatnagar-Gross-Krook (BGK) \cite{bhatnagar} model is a well-known kinetic equation for simulating rarefied gases via the evolution of a position-velocity phase-space distribution.
It is a simplification of the Boltzmann equation \cite{cercignani1988boltzmann} that relies on a nonlinear relaxation model to approximate the Boltzmann collision operator, the latter being a five-dimensional integral operator that is very expensive to compute.
The BGK collision operator recovers important properties of the Boltzmann operator; namely, it has the same collision invariants, satisfies an entropy dissipation law, and possesses the same local thermal equilibrium that enables it to recover the compressible Euler equations in the limit of infinite collisions \cite{cercignani1988boltzmann}.

Like other collisional kinetic equations, the BGK equation exhibits multiscale phenomena; in particular, it transitions between free streaming flows, when the collision frequency vanishes, to collision dominated fluid flow, when the collision frequency is large.
In fluid regimes, the BGK equation is amenable to a semi-implicit temporal discretization \cite{coron} in which the collision operator is treated implicitly and advection is treated explicitly.
However, under some circumstances, a fully implicit treatment may still be required.
Such situations arise (i) when the advection operator becomes stiff because the discrete maximum microscopic velocity is large, (ii) because of locally refined spatial meshes used to resolve boundary layers, or (iii) when a steady-state solution is desired.

The most straightforward approach to solve the BGK equation in a fully implicit manner is with source iteration (SI), a technique derived from the radiation transport community \cite{adams2002fast}.
The SI method separates the source and sink terms in the BGK collision operator and iterates to solution by lagging the source terms (and also the collision frequency if it depends on the moments of the distribution).
The remaining components of the equation form a linear transport operator that can be inverted by sweeping through the spatial mesh \cite{baker1998sn,pautz2002algorithm}.
These sweeps require that the underlying spatial discretization uses data that is upwind with respect to the microscopic velocity.
From the linear algebra point of view, the upwind formulation produces a lower or upper (block) triangular matrix system that can be solved by back substitution.

The SI procedure is simple and efficient, except when the collision frequency is large.  This regime is important, since it leads to a fluid limit. However, the number of sweep iterations needed to reach convergence becomes prohibitively large in this regime.

There are several approaches to accelerate the SI procedure when the collision frequency is large.
One approach is a high-order/low-order (HOLO) strategy \cite{taitano2014moment} that augments the kinetic equation with moment equations for mass, momentum, and energy and provides an improved estimate of the source term in the SI procedure.
The HOLO method has been shown to significantly reduce the number of iterations needed to converge to the SI solution when the collision frequency is large; however, a careful discretization is needed to ensure consistency between the coupled kinetic-moment system.
An extension of the HOLO method is the general synthetic iterative scheme (GSIS) which utilizes a Navier--Stokes--Fourier low-order solve to give improved contraction estimates for larger timesteps \cite{su2020can,su2020fast,zeng2023general}.

In the current work, we revisit the HOLO approach for computing implicit solutions of the BGK equation.
For simplicity, we restrict ourselves to a reduced phase space that includes one space and one velocity dimension (1D-1V), although the results presented readily generalize.  We combine the micro-macro (MM) and HOLO techniques to develop a method that obtains similar iteration costs as HOLO, but lends to a more memory-efficient discretization in the fluid limit.
The micro-macro decomposition \cite{liu2004boltzmann} is a well-known tool in the analysis and simulation of collisional kinetic equations, and has been used for semi-implicit time discretization of the BGK equations \cite{xiong2015high}.
Here we use it in the context of fully implicit methods.
We employ a discontinuous Galerkin (DG) discretization of the phase space that, unlike finite-difference and finite-volume discretizations, provides the required consistency between the high-order and low-order systems automatically.
We also present analysis on a linear BGK model to highlight the benefits and limitations of the HOLO and MM approaches.  In particular, the HOLO and MM approaches are not completely free of timestep restrictions.

The remainder of the paper is organized as follows.
In Section~\ref{sect:prelim}, we introduce the BGK model and its DG discretization.
In Section~\ref{sect:solvers}, we remind the reader of the SI  procedure and the HOLO strategy, discuss criteria for convergence to the SI solution, and introduce a new MM-HOLO formulation.
In Section~\ref{sec:analysis}, we formally analyze the convergence behavior of each iterative method on a linear BGK model.
In Section~\ref{sec:numerical}, we simulate the standard Sod shock tube problem and a boundary driven test problem.
These results demonstrate the analytical findings from earlier sections of the paper.
Section~\ref{sect:conclusions} contains conclusions and discussion for future work.


\section{Preliminaries, notation, and the model}\label{sect:prelim}

\subsection{The BGK model}\label{subsec:BGK_model}

The BGK model for a distribution $f = f(x,v,t)$, where $x\in\W_x:=(a,b) \subset \bbR$, $v\in\bbR$, and $t\geq 0$, is given by
\begin{equation}\label{eqn:BGK}
    \partial_t f + v\partial_x f = \nu(M(\bm{\rho}_f)-f), \quad (x,v,t) \in \W_x\times\bbR\times (0,\infty).
\end{equation}
In \eqref{eqn:BGK}, $\bm{\rho}_f = \bm{\rho}_f(x,t)$ is a vector-valued function containing the first three moments of $f$; that is, 
$
    \bm{\rho}_f = \left<\bm{e}f\right>_v$, where $\bme := (1,v,\tfrac{1}{2}v^2)^\top$ and 
$\langle\cdot\rangle_v = \int_\mathbb{R} (\cdot)\,\dx{v}$.
The constant $\nu>0$ is the collision frequency.
We use the notation $\bmrho_w $ to define the map that takes any function $w = w(v)$ to its moments $\bmrho_w= \left<{\bm e}w\right>_v$.
The moments of a distribution $w$ are related to the fluid variables of $w$ via a bijection; namely,
\begin{equation}\label{eqn:fluid_vars}
    \bmrho_w = (n_w,n_wu_w,\tfrac{1}{2}n_w(u_w^2+ \theta_w))^\top,
\end{equation}
where $n_w > 0$ is the number density, $u_w \in \bbR $ is the bulk velocity, and $\theta_w >0$ is the temperature associated to $w$.
It is natural to use the fluid variables to define the fluid equilibrium, $M(\bmrho_w)$, which is a local Maxwellian distribution specified by
\begin{equation}\label{eqn:maxwell_def}
    M(\bmrho_w) = \frac{n_w}{\sqrt{2\pi\theta_w}}\exp\left(\frac{-(v-u_w)^2}{2\theta_w}\right).
\end{equation}
In general, the notation $M(\bmeta)$ is used to specify a Maxwellian with moments given by $\bmeta$ using \eqref{eqn:fluid_vars} and \eqref{eqn:maxwell_def}.  \Cref{eqn:BGK} is equipped with inflow boundary data:
$f = f_-$ on the inflow boundary $\partial\W_-:= \{(a,v):v > 0\}\cup\{(b,v):v < 0\}$.  The outflow boundary is defined by $\partial\W_+ := \{(a,v):v < 0\}\cup\{(b,v):v > 0\}$.
In some cases $f_-$ is allowed to depend on the interior solution $f$, e.g., the far-field boundary condition that is self-consistent by setting
\begin{equation}\label{eqn:far-field-bc}
    f_- = M(\bmrho_f).
\end{equation} 
For brevity, we do not include the dependence on $f$ in the notation of $f_-$.

\subsection{The discontinuous Galerkin formulation}\label{subsec:discrete_notation}

In this subsection, we define the discontinuous Galerkin (DG) finite element method for \eqref{eqn:BGK}.

\subsubsection{Notation and discrete spaces}
 
Let $L^2(\W_x)$ be the standard Lebesgue space of square-integrable functions with canonical inner product $(\cdot\,,\cdot)_{\W_x}$ and norm $\|\cdot\|_{\W_x}$.
Let $\mathcal{T}_{x,h_x}$ be a partition of $\W_x$ with mesh parameter $h_x$ and interior skeleton $\EIx$.
We often use a uniform discretization into $N_x$ cells for $\mathcal{T}_{x,h_x}$.
Given an edge $e=\{x_e\}\in\EIx$, and a function with well-defined left and right traces at $x_e$, denoted by $q^\pm(x_e) = \lim_{x\to x_e^\pm}q(x)$, define the average and jump operators of $q$ respectively by
\begin{equation}\label{eqn:avg_and_jmp}
    \lavg{q}\ravg = \tfrac{1}{2}(q^++q^-),\qquad \ljmp{q}\rjmp = q^- - q^+.  
\end{equation}
We denote by $\ledge\cdot\redge_e$ the point-wise evaluation at $x_e$ where $e\in\EIx$, and let $\ledge\cdot\redge_{\EIx} := \sum_{e\in\EIx}\ledge\cdot\redge_e$. 
Let $V_{x,h}$ be the DG finite element space on $\mathcal{T}_{x,h_x}$ that is given by
\begin{equation}\label{eqn:discrete_space_x}
    V_{x,h} = V_{x,h}^\kappa := \{ q\in L^2(\W_x): q\big|_K\in \mathbb{P}^\kappa(K)~\forall K\in\mathcal{T}_{x,h_x} \},
\end{equation}
where $\mathbb{P}^\kappa(K)$ is the space of polynomials on $K$ with degree less than or equal to $\kappa$.
Define $\Vxhc$ to be the vector-valued DG space where $\bmeta\in\Vxhc$ if and only if each component of $\bmeta$ is in $V_{x,h}$.

To maintain conservation properties at the discrete level, we formulate a method with different trial and test spaces.
For the trial space, we restrict the velocity domain to $v\in\W_v:=[-v_{\max},v_{\max}]$ for some appropriate choice of $v_{\max}>0$ and define $L^2(\W_v)$ and its associated inner product in the same way as $L^2(\W_x)$.
Given an even integer $N_v>0$, we partition $\W_v$ into $N_v$ uniform intervals, where each interval $I_j = (v_{j-1},v_j)$ for $j=1,\ldots,N_v$ is given by 
\begin{equation}\label{eqn:velocity_partition}
  v_j = -v_{\max} + \tfrac{2jv_{\max}}{N_v}\quad\forall j=0,\ldots,N_v.
\end{equation}
Forcing $N_v$ to be even permits sweeping methods to solve the transport operator since the sign of $v$ is constant on each $I_j$.
The trial space $V_{v,h}$ is defined as
\begin{equation}\label{eqn:discrete_space_v_trial}
    V_{v,h} = \{ q\in L^2(\W_v): q\big|_K\in \mathbb{P}^2(I_j)~\forall j=1,\ldots,N_v \}.
\end{equation}
For the test space, we extend the partition of $\W_v$ to $\mathbb{R}$ via a collection of intervals $\hat{I}_j$ where $\hat{I}_1 := (-\infty,v_1)$, $\hat{I}_{N_v} := (v_{N_v-1},\infty)$, and $\hat{I}_j:=I_j$ for $j=2,\ldots,N_v-1$. 
We define $L_{\text{loc}}^2(\mathbb{R})$ to be the space of locally square-integrable functions on $\mathbb{R}$.
The test space $\widehat{V}_{v,h}\subset L_{\text{loc}}^2(\mathbb{R})$ is given by 
\begin{equation}\label{eqn:discrete_space_v_test}
    \widehat{V}_{v,h} = \{ q\in L_{\text{loc}}^2(\bbR): q\big|_K\in \mathbb{P}^2(\hat{I}_j)~\forall j=1,\ldots,N_v \}.
\end{equation}
The test space $\widehat{V}_{v,h}$ is needed so that $\{1,v,v^2\} \subset \widehat{V}_{v,h}$; such containment does not hold if $V_{v,h}$ is the test space.  Additionally, $\widehat{V}_{v,h}$ does not introduce issues with integrability since it is used to test against functions in $V_{v,h}$, which have compact support, or against a Maxwellian. 
There are two natural embeddings for $V_{v,h}$.  One is the embedding $V_{v,h}\hookrightarrow L^2(\bbR)$ by the trivial (zero) extension.
The other is $V_{v,h}\hookrightarrow \widehat{V}_{v,h}$ by extending the polynomials on $I_1$ and $I_{N_v}$ to  $\hat{I}_1$ and $\hat{I}_{N_v}$ respectively.

The DG trial and test spaces on the $(x,v)$ phase space are given by $V_h = V_{x,h}\otimes V_{v,h}$ and $\Vhat = V_{x,h}\otimes \widehat{V}_{v,h}$ respectively.
We define the inner product $(\cdot\,,\cdot)$ and norm $\|\cdot\| = \sqrt{(\cdot,\cdot)}$ for $L^2(\W_x\times\bbR)$ by
\begin{equation}\label{eqn:inner_product_xv}\textstyle
    (w,z) = \int_{\W_x\times\mathbb{R}} w(x,v)z(x,v)\,\dx{x}\,\dx{v}    
\end{equation}
respectively. 
We use the same notation for $L^2(\W)$, where $\W := \W_x\times\W_v$.
\Cref{eqn:inner_product_xv} defines also defines an inner product for $V_{h}$ using the trivial extension.
Integration $\ledge\cdot\redge$ of edges in $\W$ are decomposed into pointwise evaluations in $x$ and one-dimensional integration in $v$.

\subsubsection{Discretization of the transport operator and BGK model}

We discretize the transport operator $w \mapsto v\partial_x w$ with outflow boundary data using
\begin{equation} \label{eqn:transport_disc_def}
    \mA(w,z) := -(vw,\partial_{x,h} z) + \ledge\widehat{vw},\ljmp z\rjmp\redge_{\EIx\times\bbR} + \ledge vw,z\bm{n}\redge_{\partial\W_+},
\end{equation}    
where the numerical flux $\widehat{vw}$ is the standard upwind flux given by
\begin{equation}\label{eqn:numerical_flux}
    \widehat{vw} = v\lavg w\ravg + \tfrac{|v|}{2}\ljmp w\rjmp,
\end{equation}
and $\partial_{x,h}$ denotes the piecewise gradient on $\mathcal{T}_{x,h}$.
Here $\bm{n}=-1$ or $\bm{n}=+1$ depending of whether the $x$-coordinate in $\partial\W_+$ is $a$ or $b$ respectively.
The inflow data $f_-$ is discretized by
\begin{equation} \label{eqn:transport_disc_def_bc}
    \mB(f_-,z) = \ledge vf_-,z\bm{n}\redge_{\partial\W_-}
\end{equation}
with the same definition of $\bm{n}$ on $\partial\W_-$, and we assume $f_-$ is continuous in $V_h$.
We then build the semi-discrete DG scheme for \eqref{eqn:BGK}: \textit{Find $f_h(t)\in V_h$ such that}
\begin{equation}\label{eqn:discrete_BGK}
    (\partial_t f_h,z_h) + \mL(f_h,z_h) = \nu(M(\bmrho_{f_h}),z_h) - \mathcal{B}(f_-,z_h) \quad\forall z_h\in \Vhat
\end{equation}
where
\begin{equation}
    \mL(w_h,z_h) := \mathcal{A}(w_h,z_h) + \nu(w_h,z_h)  \quad\forall w_h\in V_h, z_h\in \Vhat.
\end{equation}

We discretize \eqref{eqn:discrete_BGK} in time via the backward Euler method; extensions to higher-order Runge-Kutta methods and the justification for a fully implicit treatment of \eqref{eqn:discrete_BGK} are discussed in \Cref{subsec:time-stepping,subsec:suddenheat:implicit_need} respectively.
Let $\dt>0$ and $t^{\{k\}}=k\dt$, for $k=\{0,1,2,\dots\}$.
The weak formulation of the backward Euler discretization is: \textit{Given $f^{\{k\}}\in V_h$, find $f^{\{k+1\}}\in V_h$ such that}\footnote{To reduce notation, we suppress the $h$ dependence on the fully-discrete approximation $f^{\{k\}}$.} 
\begin{equation}\label{eqn:BE_BGK}
    (f^{\{k+1\}},z_h) + \dt\mL(f^{\{k+1\}},z_h) = (f^{\{k\}},z_h) + \dt\nu(M(\bmrho_{f^{\{k+1\}}}),z_h) - \dt\mathcal{B}(f_-^{\{k+1\}},z_h)
\end{equation}
\textit{for any $z_h\in \Vhat$}.
Here $f^{\{k\}}\approx f_h(\cdot\,,\cdot,t^{\{k\}})$ and $f_-^{\{k+1\}}$ is the inflow data at $t^{\{k+1\}}$.
Associated to \eqref{eqn:BE_BGK} is residual $\mathcal{R}:V_h\to V_h$ defined for all $z_h\in \Vhat$ by
\begin{equation}\label{eqn:BE_residual}
    (\mathcal{R}w,z_h) = (w,z_h) + \dt\mL(w,z_h)  - (f^{\{k\}},z_h) - \dt\nu(M(\bmrho_{w}),z_h) + \dt\mathcal{B}(f_-,z_h).
\end{equation}


\section{Iterative solvers}\label{sect:solvers}

We now present several iterative methods to solve \eqref{eqn:BE_BGK}.
We first review the source iteration \cite{adams2002fast} and high-order/low-order \cite{taitano2014moment} methods, then introduce two micro-macro approaches.

\subsection{Source iteration}\label{subsec:SI}

The simplest iterative method, called \textit{source iteration} (SI) \cite{adams2002fast}, lags the right-hand side of \eqref{eqn:BE_BGK}: \textit{Given $f^\ell \in V_h$, find $f^{\ell+1}\in V_h$ such that}
\begin{equation}\label{eqn:BE_SI}
    (f^{\ell+1},z_h) + \dt\mL(f^{\ell+1},z_h) = (f^{\{k\}},z_h) + \dt\nu(M(\bmrho_{f^{\ell}}),z_h) - \dt\mathcal{B}(f^{\ell}_-,z_h)
\end{equation}
\textit{for every $z_h\in\Vhat$}.  Here $f^{\ell}_-$  denotes the possible dependence of $f_-$ on $f^\ell$.

For each fixed $\ell$, the operator on the left-hand side of \eqref{eqn:BE_SI} is linear and can be inverted by sweeping \cite{baker1998sn}.  
However, because the Maxwellian on the right-hand side of \eqref{eqn:BE_SI} is lagged, the contraction constant for at least the linear model (see \Cref{sec:analysis}) is bounded by $\frac{\dt\nu}{1+\dt\nu}$.
In highly collisional regimes ($\nu\gg 1$), this contraction constant is close to 1, and thus many iterations are required in order to converge \eqref{eqn:BE_SI}.
Moreover, in higher physical dimensions with unstructured meshes on $\W_x$, the sweeps used to invert \eqref{eqn:BE_SI} are more complicated and expensive \cite{pautz2002algorithm}.
These facts motivate strategies to accelerate SI.  

\subsection{The HOLO method}

One approach to accelerate SI is the \textit{high-order/low-order} (HOLO) method; see \cite{chacon2017multiscale} for a review and \cite{taitano2014moment} for a specific application to the BGK equation. 
The idea of the HOLO method is to decrease the number of transport sweeps in \eqref{eqn:BE_SI} by constructing a better approximation to $\bmrho_{f^{\{k+1\}}}$ than $\bmrho_{f^{\ell}}$.
Because $\bmrho_{f^{\{k+1\}}}\in\Vxhc$ is only a function of $x$ and $t$, a moment-based solve to improve $\bmrho_{f^{\ell}}$ should be cheaper than a sweep in phase space, especially on unstructured meshes in higher dimensions.

We now motivate the HOLO method.
Choosing $z_h = \bm{e}\cdot\bm{q}_h$ where $\bmq_h\in \Vxhc$ in \eqref{eqn:BE_BGK} yields the following moment system for 
$\bmrho_{f^{\{k+1\}}}$:
\begin{align}  \label{eqn:BE_rho_update}  
\begin{split}
    (\bmrho_{f^{\{k+1\}}},\bmq_h)_{\W_x} 
    &+ \dt \mathcal{A}(f^{\{k+1\}},\bm{e}\cdot\bm{q}_h) 
    = (\bmrho_{f^{\{k\}}},\bmq_h)_{\W_x}  
    - \dt \mB(f_-^{\{k+1\}},\bme\cdot\bmq_h).
\end{split} 
\end{align}
Upon writing $f^{\{k+1\}} = M(\bmrho_{f^{\{k+1\}}}) + (f^{\{k+1\}} - M(\bmrho_{f^{\{k+1\}}}))$ and rearranging terms, \eqref{eqn:BE_rho_update} becomes
\begin{align}  \label{eqn:BE_rho_update:2}  
\begin{split}
    (\bmrho_{f^{\{k+1\}}},\bmq_h)_{\W_x} 
    &+ \dt \mE(\bmrho_{f^{\{k+1\}}},\bmq_h)
    = (\bmrho_{f^{\{k\}}},\bmq_h)_{\W_x} - \dt [ \mathcal{A}(f^{\{k+1\}},\bme\cdot\bmq_h)  \\
    &\qquad - \mE(\bmrho_{f^{\{k+1\}}},\bmq_h) ]
    - \dt \mB(f_-^{\{k+1\}},\bme\cdot\bmq_h),
\end{split} 
\end{align}
where
\begin{align}\label{eqn:euler_disc_def}
\begin{split}
    \mE(\bmeta,\bmq_h) &:= \mA(M(\bmeta),\bme\cdot\bmq_h) \\
    &\phantom{:}= 
    -(\bmF(\bmeta),\partial_{x,h} \bmq_h)_{\W_x} + \ledge\widehat{\bmF(\bmeta)},\ljmp\bmq_h\rjmp\redge_{\EIx} + \ledge\bme vM(\bmeta) ,\bmq_h\bmn\redge_{\partial\W_+},
\end{split}
\end{align}
$\bm{F}$ is the flux associated with the Euler equations, given by
\begin{equation}\label{eqn:euler_flux}
    \bmF(\bmeta)
        = \left<\bme vM(\bmeta)\right>_v = (nu,nu^2+n\theta,\tfrac{1}{2}nu(u^2+3\theta))^\top,
\end{equation}
and $\widehat{\bmF(\bmeta)}$ is the numerical flux associated to $\bmF$, given by
\begin{equation}\label{eqn:euler_kinetic_flux}
        \widehat{\bmF(\bmeta)} = \lavg\bmF(\bmeta)\ravg + \ljmp\left<|v|\bme M(\bmeta)\right>_v\rjmp.
\end{equation}

The HOLO method replaces $\bmrho_{f^{\{k+1\}}}$ in \eqref{eqn:BE_rho_update:2} by a new update $\bmrho^{\ell+1}$ that is computed using $f^\ell$ to construct the right-hand side of \eqref{eqn:BE_rho_update:2}.
The method is as follows: \textit{Given $f^\ell\in V_h$, find $f^{\ell+1}\in V_h$ such that}
\begin{subequations}\label{eqn:BE_HOLO}
\begin{align} \label{eqn:BE_HOLO:HI} 
\begin{split}  
    (f^{\ell+1},z_h) + \dt\mL(f^{\ell+1},z_h) &= (f^{\{k\}},z_h) + \dt\nu(M(\bmrho^{\ell+1}),z_h) - \dt\mathcal{B}(f^{\ell}_-,z_h)
\end{split}
\end{align}
\textit{for all $z_h\in \Vhat$, where $\bmrho^{\ell+1}\in \Vxhc$ is given for any $\bmq_h\in\Vxhc$ by}
\begin{align}  \label{eqn:BE_HOLO:LO}  
\begin{split}
    (\bmrho^{\ell+1},\bmq_h)_{\W_x} 
    &+ \dt \mE(\bmrho^{\ell+1},\bmq_h) 
    = (\bmrho_{f^{\{k\}}},\bmq_h)_{\W_x}  
     \\
    &\quad - \dt \big[\mathcal{A}(f^{\ell},\bm{e}\cdot\bm{q}_h) - \mE(\bmrho_{f^\ell},\bmq_h)\big] - \dt \mB(f^\ell_-,\bme\cdot\bmq_h).
\end{split} 
\end{align}
\end{subequations}

\begin{remark}\label{rmk:boundary_condtions}
    For calculations with far-field boundary conditions \eqref{eqn:far-field-bc}, in which $f_-$ depends only on moment data, we modify the method slightly by replacing $\mathcal{B}(f^{\ell}_-,z_h)$ in \eqref{eqn:BE_HOLO} with $\mathcal{B}(f^{\ell+1}_-,z_h)$ that is built from $\bmrho^{\ell+1}$.  
    In \eqref{eqn:BE_HOLO:LO}, $\mathcal{B}(f_-^{\ell+1},z_h)$ is moved to the left-hand side and treated as part of the solve.
    This small modification is a choice that has little bearing on the numerical results, as long as $\mathcal{B}$ is treated consistently in \eqref{eqn:BE_HOLO:HI} and \eqref{eqn:BE_HOLO:LO}.
\end{remark}

The perturbative flux $\mathcal{A}(f^{\ell},\bm{e}\cdot\bm{q}_h)
    - \mE(\bmrho_{f^\ell},\bmq_h)$ on the right-hand side of \eqref{eqn:BE_HOLO:LO} is a DG discretization of the moments
\begin{equation}
\label{eqn:heat_flux}
((0,0,\big<\tfrac{1}{2}(v-u_{f^\ell})^3f^\ell\big>_v)^\top,\bmq_h)_{\W_x}.
\end{equation}
Since the last component of \eqref{eqn:heat_flux} is the heat flux associated to $f^\ell$,
\eqref{eqn:BE_HOLO:LO} can be viewed as an approximation to the moment system for $\bmrho_{f^{\{k+1\}}}$, see \eqref{eqn:BE_rho_update:2}, in which the heat flux is calculated from the previous iterate.  

The following proposition shows that, in the limit $\ell \to \infty$, the moments generated by \eqref{eqn:BE_HOLO:LO} match the moments of the kinetic update in \eqref{eqn:BE_HOLO:HI}, provided a structural condition on the Euler flux holds.

\begin{prop}\label{prop:HOLO_moment_convergence}
    Let $\mE^*:\Vxhc\to\Vxhc$ be defined by 
    \begin{equation}\label{eqn:euler_timestep_map}
        (\mE^*\bmeta,\bmq_h) = (1+{\dt\nu})(\bmeta,\bmq_h) + \Delta t\mE(\bmeta,\bmq_h) \quad\forall\bmq_h\in\Vxhc.
    \end{equation}
    Assume that $\mE^*$ is injective on $\Vxhc$.
    Let $\{f^\ell,\bmrho^\ell\}_\ell$ be defined from \eqref{eqn:BE_HOLO}.
    Suppose that $f^\ell \to f^*\in V_h$ and $\bmrho^\ell\to\bmrho^*\in\Vxhc$ as $\ell\to\infty$.
    Then $\bmrho_{f^*} = \bmrho^*$.
\end{prop}

\begin{proof}
    Let $f_-^* = \lim_{\ell\to\infty} f_-^\ell$.
    Taking the limit of \eqref{eqn:BE_HOLO} as $\ell\to\infty$ yields
    \begin{align} 
    \label{eqn:HOLO_moment_convergence:1} 
        (f^*,z_h) + \dt\mL(f^*,z_h) &= (f^{\{k\}},z_h) + {\dt\nu}(M(\bmrho^*),z_h) - \dt\mathcal{B}(f_-^*,z_h),\\
    \label{eqn:HOLO_moment_convergence:2}
    \begin{split}
        (\bmrho^*,\bmq_h)_{\W_x} + \dt \mE(\bmrho^*,\bmq_h) &= (\bmrho_{f^{\{k\}}},\bmq_h)_{\W_x}
        - \dt[\mA(f^*,\bme\cdot\bmq_h)
        -\mE(\bmrho_{f^*},\bmq_h)] \\
        &\quad -\dt \mB(f_-^*,\bme\cdot\bmq_h),
    \end{split} 
    \end{align}
    for any $z_h\in\Vhat$ and $\bmq_h\in\Vxhc$.
    Choosing $z_h=\bme\cdot\bmq_h\in\Vhat$ in \eqref{eqn:HOLO_moment_convergence:1} and then subtracting \eqref{eqn:HOLO_moment_convergence:2} gives, after some cancellations,  $\mE^*\bmrho_{f^*} = \mE^*\bmrho^*$.
    Therefore the assumption that $\mE^*$ is injective yields the intended result.
    The proof is complete.
\end{proof}

\begin{remark}~
    \begin{enumerate}
        \item The numerical experiments in \Cref{sec:numerical}, specifically \Cref{tab:holo_vs_SI_break}, suggest that the condition on $\mE^*$ is necessary as well as sufficient.
        \item The proof of \Cref{prop:HOLO_moment_convergence} allows $\mE$ to be inconsistent with $\mA$, that is, $\mE(\bmeta,\bmq_h)\neq \mA(M(\bmeta),\bme\cdot\bmq_h)$. This gives the opportunity for choosing different discretizations for $\mE$ and $\mA$, but may degrade the performance of HOLO. 
    \end{enumerate}
\end{remark}

When applying accelerators, it is important that the accelerated iterations converge to a solution of \eqref{eqn:BE_BGK}. 
In \cite{taitano2014moment}, where the HOLO formulation was discretized using finite differences, several consistency terms were added in order to preserve this property. 
An advantage of the DG discretization is that this desired consistency is automatically satisfied.
This result, which follows from \Cref{prop:HOLO_moment_convergence}, is shown below.

\begin{prop}\label{prop:HOLO_vs_SI_equiv}
    Let $\{f^\ell,\bmrho^\ell\}_\ell$ be defined from \eqref{eqn:BE_HOLO}.
    Suppose that $f^\ell \to f^*\in V_h$ and $\bmrho^\ell\to\bmrho^*\in\Vxhc$ as $\ell\to\infty$, and $\bmrho^*=\bmrho_{f^*}$.
    Then $f^*$ is a solution to \eqref{eqn:BE_BGK}.
\end{prop}

\begin{proof}
    Substituting $\bmrho_{f^*}$ for $\bmrho^*$ in \eqref{eqn:HOLO_moment_convergence:1} immediately implies the result.
\end{proof}

\subsection{The micro-macro method}\label{subsec:MM_method}

In the micro-macro (MM) formulation, $f$ is decomposed as $f=M(\bmrho)+g$ where $\bmrho=\bmrho_f$.  
The function $g = g(x,v,t)$ is called the \textit{micro distribution} and satisfies $\langle\bme g\rangle_v=0$.
The BGK model \eqref{eqn:BGK} with the MM ansatz can be split into a coupled system:
\begin{subequations}\label{eqn:BGK_MM}
\begin{align}
    \partial_t\bmrho + \partial_x\bm{F}(\bmrho) &= -\partial_x\big<\bme vg\big>_v, \label{eqn:BGK_MM:rho}  \\
    \partial_t g + v\partial_x g + \nu g &= -\partial_t M(\bmrho)  -v\partial_x M(\bmrho).\label{eqn:BGK_MM:g}
\end{align}
\end{subequations}
When $\nu \gg 1$, the magnitude of $g$ away from initial and boundary layers is $\mathcal{O}(\nu^{-1})$.  In such areas, the MM decomposition is often preferred since the discretization of $g$ can be compressed to reduce the overall degrees of freedom required for an accurate numerical solution of the MM system \eqref{eqn:BGK_MM}.
This fact was demonstrated numerically in \cite{endeve2022conservative} for the Vlasov-Poisson system with a Lenard-Bernstein collision operator.   

The condition $\langle\bme g\rangle_v=0$ is automatically satisfied in \eqref{eqn:BGK_MM}, but can be lost if care is not taken when discretizing in both phase space and time.
In \cite{endeve2022conservative}, the authors develop spatial and temporal (implicit-explicit) methods that maintain this condition discretely in time.

Using the backward Euler discretization of \eqref{eqn:BGK_MM}, we propose the following discrete MM analog to \eqref{eqn:BE_BGK}: \textit{Find $\bmrho^{\{k+1\}}\in \Vxhc$ and $g^{\{k+1\}}\in V_h$ such that}
\begin{subequations}\label{eqn:BGK_MM_BE}
\begin{align}\label{eqn:BGK_MM_BE:rho}
\begin{split}
    (\bmrho^{\{k+1\}},\bmq_h)_{\W_x} &+ \dt\mE(\bmrho^{\{k+1\}},\bmq_h) = (\bmrho^{\{k\}},\bmq_h)_{\W_x} -\dt\mA(g^{\{k+1\}},\bme\cdot\bmq_h) \\
    &\!\!\!\!\!\!-\dt\mB(f_-^{\{k+1\}},\bme\cdot\bmq_h),
\end{split} \\
\begin{split}
    (g^{\{k+1\}},z_h) &+ \dt\mL(g^{\{k+1\}},z_h)  = (g^{\{k\}},z_h)
    +(M(\bmrho^{\{k\}}),z_h) \\
    &\!\!\!\!\!\!-(M(\bmrho^{\{k+1\}}),z_h)
    -\dt\mA(M(\bmrho^{\{k+1\}}),z_h)
    -\dt\mB(f_-^{\{k+1\}},z_h),
\end{split}
\label{eqn:BGK_MM_BE:g}
\end{align}
\end{subequations}
\textit{for all $\bmq_h\in \Vxhc$ and $z_h\in \Vhat$.}
In \eqref{eqn:BGK_MM_BE}, $f_-^{\{k+1\}}$ is built with the data from $M(\bmrho^{\{k+1\}}) + g^{\{k+1\}}$.
We show that $g^{\{k+1\}}$ satisfies the zero-moment condition.

\begin{prop}\label{prop:BGK_MM_BE_cons}
    Suppose $\langle\bme g^{\{k\}}\rangle_v=0$. If $\bmrho^{\{k+1\}}\in \Vxhc$ and $g^{\{k+1\}}\in V_h$ satisfy \eqref{eqn:BGK_MM_BE}, then $\langle\bme g^{\{k+1\}}\rangle_v=0$.
\end{prop}

\begin{proof}
    For brevity, denote $\widehat{\bmrho}:=\bmrho^{\{k+1\}}$, $\widehat{g}:=g^{\{k+1\}}$.
    Choosing $z_h=\bme\cdot\bmq_h\in \Vhat$ in \eqref{eqn:BGK_MM_BE:g}, recalling \eqref{eqn:euler_disc_def}, and rearranging terms gives
    \begin{align}
    \begin{split}
     (1+{\dt\nu})(\bmrho_{\widehat{g}},\bmq_h)_{\W_x} + (\widehat{\bmrho},\bmq_h)_{\W_x} &+  \dt\mE(\widehat{\bmrho},\bmq_h) = (\bmrho^{\{k\}},\bmq_h)_{\W_x} \\
     &\quad - \dt\mA(\widehat{g},\bme\cdot\bmq_h) -\dt \mB(f_-^{\{k+1\}},\bme\cdot\bmq_h).
    \end{split}
    \label{eqn:BGK_MM_BE_cons:2}    
    \end{align}
    Subtracting \eqref{eqn:BGK_MM_BE:rho} from \eqref{eqn:BGK_MM_BE_cons:2} yields
    $\big(1+{\dt\nu}\big)(\bmrho_{\widehat{g}},\bmq_h)_{\W_x} = 0$ for all $\bmq_h\in\Vxhc$, which immediately implies that $\langle\bme \widehat{g}\,\rangle_v=\bmrho_{\widehat{g}}=0$.  The proof is complete.
\end{proof}

We now consider two iterative methods to solve \eqref{eqn:BGK_MM_BE}.
For the first method, we lag right-hand side of \eqref{eqn:BGK_MM_BE:rho}. 
We denote this scheme by \tMMSI: \textit{Given $\bmrho^\ell\in\Vxhc$ and $g^{\ell}\in V_h$, find $\bmrho^{\ell+1}\in\Vxhc$ and $g^{\ell+1}\in V_h$ such that for any $\bmq_h\in \Vxhc$ and $z_h\in \Vhat$, there holds}
\begin{subequations}\label{eqn:BE_MM_SI}
\begin{align}
    \begin{split}
    (\bmrho^{\ell+1},\bmq_h)_{\W_x} &+ \dt\mE(\bmrho^{\ell+1},\bmq_h) = (\bmrho^{\{k\}},\bmq_h)_{\W_x}  -\dt\mA(g^{\ell},\bme\cdot\bmq_h) \\
    &\quad-\dt \mB(f_-^\ell,\bme\cdot\bmq_h),
    \label{eqn:BE_MM_SI:rho}
    \end{split} \\
    \begin{split}
    (g^{\ell+1},z_h) & + \dt\mL(g^{\ell+1},z_h)  = (g^{\{k\}},z_h) + (M(\bmrho^{\{k\}}),z_h) \\
    &\quad - (M(\bmrho^{\ell+1}),z_h)
    -\dt \mA(M(\bmrho^{\ell+1}),z_h) -\dt \mB(f_-^\ell,z_h).
    \end{split}
    \label{eqn:BE_MM_SI:g}
\end{align}
\end{subequations} 

The MM-L iteration requires the same operations as the HOLO method \eqref{eqn:BE_HOLO}: a nonlinear fluid solve for $\bmrho^{\ell+1}$ and a linear transport sweep for $g^{\ell+1}$.
However, we show in \Cref{subsec:riemann} that this approach has the opposite problem of the source iteration method \eqref{eqn:BE_SI}. 
When $\nu\gg 1$, MM-L performs well, since $g$ is small and \eqref{eqn:BE_MM_SI:rho} is a good approximation to the fluid limit.  However, for moderately sized $\nu$, the number of iterations quickly explodes. 
We attribute to the poor performance of MM-L to the following two facts: (i) $\langle\bme g^{\ell-1}\rangle_v=0$ does not guarantee that $\langle\bme g^{\ell}\rangle_v=0$, and (ii) if $\langle\bme g^{\ell}\rangle_v\neq 0$, then an improper heat flux, i.e., $\mA(g^\ell,\bme\cdot\bmq_h)$, is being used in \eqref{eqn:BE_MM_SI:rho}.
We propose a MM-HOLO method for \eqref{eqn:BGK_MM_BE} by employing a HOLO-like strategy and applying the heat flux source correction in \eqref{eqn:BE_HOLO:LO} using the current approximation of the kinetic distribution.  Using the MM ansatz, $f^\ell=M(\bmrho^\ell) + g^\ell$, the correction reads 
\begin{equation}
\mathcal{A}(f^{\ell}-M(\bmrho_{f^\ell}),\bm{e}\cdot\bm{q}_h) = \mE(\bmrho^\ell,\bmq_h) + \mathcal{A}(g^{\ell},\bm{e}\cdot\bm{q}_h) -\mE(\bmrho^\ell+\bmrho_{g^\ell},\bmq_h).
\end{equation}
The MM-HOLO iteration is:  \textit{Given $\bmrho^\ell\in\Vxhc$ and $g^{\ell}\in V_h$, find  $\bmrho^{\ell+1}\in\Vxhc$ and $g^{\ell+1}\in V_h$ such that for any $\bmq_h\in \Vxhc$ and $z_h\in \Vhat$, there holds}
\begin{subequations}\label{eqn:BE_MM}
\begin{align}
\begin{split}
    (\bmrho^{\ell+1},\bmq_h)_{\W_x} 
    &+ \dt\mE(\bmrho^{\ell+1},\bmq_h) 
    = (\bmrho^{\{k\}},\bmq_h)_{\W_x}  
    -\dt\big[\mE(\bmrho^\ell,\bmq_h) \\
    &\quad +\mA(g^{\ell},\bme\cdot\bmq_h) - \mE(\bmrho^\ell+\bmrho_{g^\ell},\bmq_h)\big] - \dt\mB(f_-^\ell,\bme\cdot\bmq_h),
\end{split}\label{eqn:BE_MM:rho}\\
\begin{split}
    (g^{\ell+1},z_h)&+\dt\mL(g^{\ell+1},z_h)  = (g^{\{k\}},z_h) + (M(\bmrho^{\{k\}}),z_h) \\
    &\quad - (M(\bmrho^{\ell+1}),z_h)
    -\dt \mA(M(\bmrho^{\ell+1}),z_h) -\dt \mB(f_-^\ell,z_h).
\end{split}
\label{eqn:BE_MM:g}
\end{align}
\end{subequations}
In the event of boundary conditions that depend only on moment data of $f$, we use the modification given in \Cref{rmk:boundary_condtions} for both MM methods.

If $\bmrho_{g^\ell}=\langle\bme g^{\ell}\rangle_v=0$, then \eqref{eqn:BE_MM} is equivalent to \eqref{eqn:BE_MM_SI}; however, this condition often does not hold.
We will show analytically and numerically that \eqref{eqn:BE_MM} is superior to \eqref{eqn:BE_MM_SI}, as \eqref{eqn:BE_MM:rho} provides a much better approximation of the heat flux. 

MM-HOLO \eqref{eqn:BE_MM} inherits the same properties as HOLO in \Cref{prop:HOLO_moment_convergence,prop:HOLO_vs_SI_equiv}.
We state these properties below but omit the proofs due to the similarity with \Cref{prop:HOLO_moment_convergence,prop:HOLO_vs_SI_equiv}.

\begin{prop}\label{prop:MM_moment_convergence}
    Suppose $\langle\bme g^{\{k\}}\rangle_v=0$.
    Assume that $\mE^*$ in \eqref{eqn:euler_timestep_map} is injective on $\Vxhc$.
    Let $\{\bmrho^\ell,g^\ell\}_\ell$ be defined from the MM-HOLO method \eqref{eqn:BE_MM}.
    Suppose that $\bmrho^\ell \to \bmrho^*\in \Vxhc$ and $g^\ell\to g^*\in V_h$ as $\ell\to\infty$.
    Then $\bmrho_{g^*}=0$.
\end{prop}

\begin{prop}\label{prop:MM_vs_MM_SI_equiv}
    Assume $\langle\bme g^{\{k\}}\rangle_v=0$.
    Suppose $(\bmrho^\ell,g^\ell)$ in the MM-L method \eqref{eqn:BE_MM_SI} converges to $(\bmrho^\tMMSI,g^\tMMSI)\in \Vxhc\times V_h$ as $\ell\to\infty$.  Further suppose $(\bmrho^\ell,g^\ell)$ in the MM-HOLO method \eqref{eqn:BE_MM} converges to $(\bmrho^\tMM,g^\tMM)\in \Vxhc\times V_h$ as $\ell\to\infty$ and that $\bmrho_{g^\tMM} = 0$.
    Then $(\bmrho^\tMMSI,g^\tMMSI)$ and $(\bmrho^\tMM,g^\tMM)$ both solve \eqref{eqn:BGK_MM}.
\end{prop}


\section{Convergence analysis in a linear BGK setting}\label{sec:analysis}

In this section, we provide some formal analysis that shows the advantages and limitations of the HOLO and MM methods when compared to source iteration \eqref{eqn:BE_SI}.
We pose several simplifying assumptions which highlight the dependence of the convergence rate in $\ell$ with respect to problem and discretization parameters; namely, $\nu$, $\dt$, and $h_x$.
We focus on a simplified linear BGK model given by
\begin{equation}\label{eqn:lin_bgk:system}
    \partial_t f + v\partial_x f = \nu({n}_f\mpM-f),
\end{equation}
where ${n}_f=\vbr{f}$.
The static Maxwellian $\mpM=\mpM(v)$ is defined as
\begin{equation}
    \mpM(v) = \frac{1}{\sqrt{2\pi{\theta_0}}}\exp\bigg(\frac{-(v-{u_0})^2}{2{\theta_0}}\bigg),
\end{equation}
where ${u_0} = \vbr{v\mpM}\in\bbR$ and ${\theta_0} = \vbr{v^2\mpM}-u_0^2>0$ are constant. 
Note that $\vbr{\mpM}=1$.
A more complicated but more physically relevant model that preserves all three conservation invariants is analyzed in \Cref{sec:appendix_bgk}.

\begin{remark}\label{rmk:linear_bgk_disc}
    We equip the linear BGK model \eqref{eqn:lin_bgk:system} with periodic boundary conditions in $x$ and give no discretization in velocity space.
    We discretize $\partial_x$ on $V_{x,h}$ via central differences; this grants a strong form operator $A$ that is skew-symmetric with respect to $L^2(\W_x)$ and contains only imaginary eigenvalues $i\lambda$ with $\lambda\in\bbR$ and $|\lambda|\leq 1/h_x$.  
\end{remark}

Applying the backward Euler scheme with this discretization leads to the following problem: \textit{Given $f^{\{k\}}$, find $f^{\{k+1\}}$ such that}
\begin{equation}\label{eqn:lin_bgk:BE}
    f^{\{k+1\}} + \dt vA f^{\{k+1\}} + \nu\dt f^{\{k+1\}} = f^{\{k\}} + \nu\dt \mpM{n}_{f^{\{k+1\}}}.
\end{equation}

Let $P_\mpM:L^2(\W_v)\to \textrm{span}\{\mpM\}$ be given by $P_\mpM w = \vbr{w}\mpM=\mpM{n}_w$.
$P_\mpM$ is an orthogonal projection with respect to the $\mpM^{-1}$ inner product $\Mip{w}{z} := \vbr{wz\mpM^{-1}}$ and can be extended to an orthogonal projection in $L^2(\W)$ with respect to the inner product $(w,z)_{\mpM}:=(w,z\mpM^{-1})$.
Define $P_\mpM^\perp = I-P_\mpM$.
In this decomposition,
\begin{align}
f = P_\mpM f + P_\mpM^\perp f &= \mpM{n}_f + (f - \mpM{n}_f),\label{eqn:lin_bgk:decomp} \\
\|f\|_{\mpM}^2 = \|P_\mpM f\|_\mpM^2 + \|P_\mpM^\perp f\|_\mpM^2 &= \|\mpM{n}_f\|_{\mpM}^2 + \|P_\mpM^\perp f\|_\mpM^2 = \|{n}_f\|_{\W_x}^2 + \|P_\mathpzc{M}^\perp f\|_\mathpzc{M}^2.   \label{eqn:lin_bgk:orth_decomp}
\end{align}

\subsection{Source iteration}\label{subsec:lin_bgk:SI_theory}

The source iteration method to solve \eqref{eqn:lin_bgk:BE} is as follows: \textit{Given $f^\ell$, find $f^{\ell+1}$ such that}
\begin{equation}\label{eqn:lin_bgk:SI}
    f^{\ell+1} + \dt vA f^{\ell+1} + \nu\dt f^{\ell+1} = f^{\{k\}} + \nu\dt \mpM{n}_{f^\ell}.
\end{equation}

We now list the convergence result for SI.
\begin{prop}\label{prop:lin_bgk:SI_error}
    Define $e^{\ell} = f^{\ell} - f^{\{k+1\}}$ where $f^\ell$ is given in \eqref{eqn:lin_bgk:SI}.
    Then
    \begin{equation}\label{eqn:lin_bgk:SI_error:0}
        (1+\nu\dt)\|P_\mpM e^{\ell+1}\|_\mpM^2 + (1+\nu\dt)\|P_\mpM^\perp e^{\ell+1}\|_\mpM^2 \leq \nu\dt\|P_\mpM e^{\ell}\|_\mpM \|P_\mpM e^{\ell+1}\|_\mpM.
    \end{equation}
    Moreover,
    \begin{align}\label{eqn:lin_bgk:SI_error:1}
        \|P_\mpM e^{\ell+1}\|_\mpM \leq \frac{\nu\dt}{1+\nu\dt} \|P_\mpM e^{\ell}\|_\mpM.
    \end{align}
\end{prop}

\begin{proof}
    Subtracting \eqref{eqn:lin_bgk:BE} from \eqref{eqn:lin_bgk:SI} yields
    \begin{align}\label{eqn:lin_bgk:SI_error:2}
        (1+\nu\dt)e^{\ell+1} + \dt vA e^{\ell+1} = \nu\dt (\mpM{n}_{f^\ell}-\mpM{n}_{f^{\{k+1\}}}) = \nu\dt P_\mpM e^{\ell}.
    \end{align}
    Testing \eqref{eqn:lin_bgk:SI_error:2} by $e^{\ell+1}\mpM^{-1}$, and then applying \eqref{eqn:lin_bgk:orth_decomp}, the skew-symmetry of $A$, the orthogonality of $P_\mpM$, and H\"older's inequality yields
    \begin{align}\label{eqn:lin_bgk:SI_error:3}
    \begin{split}
        (1+\nu\dt)\big( \|P_\mpM e^{\ell+1}\|_\mpM^2 + \|P_\mpM^\perp e^{\ell+1}\|_\mpM^2 \big) &= \nu\dt (P_\mpM e^\ell,e^{\ell+1})_\mpM \\
        &= \nu\dt(P_\mpM e^\ell,P_\mpM e^{\ell+1})_\mpM
    \end{split} \\
        &\leq \nu\dt\|P_\mpM e^{\ell}\|_\mpM \|P_\mpM e^{\ell+1}\|_\mpM.
    \label{eqn:lin_bgk:SI_error:4}
    \end{align}
    \Cref{eqn:lin_bgk:SI_error:4} is exactly \eqref{eqn:lin_bgk:SI_error:0} and leads to \eqref{eqn:lin_bgk:SI_error:1}.
    The proof is complete.
\end{proof}

From \eqref{eqn:lin_bgk:SI_error:0}, one can show $\|P_\mpM^\perp e^{\ell+1}\|$ is also bounded by $\|P_\mpM e^{\ell}\|$; therefore, SI is unconditionally stable in $\dt$.
However, as $\nu\dt\to\infty$, the contraction constant on the error approaches one and the convergence rate degrades.

\subsection{HOLO method}

The HOLO method for \eqref{eqn:lin_bgk:BE} is: \textit{Given $f^\ell$, find $f^{\ell+1}$ such that}
\begin{subequations}\label{eqn:lin_bgk:HOLO}
\begin{align}
    f^{\ell+1} + \dt vA f^{\ell+1} + \nu\dt f^{\ell+1} &= f^{\{k\}} + \nu\dt \mpM{n}^{\ell+1}, 
    \label{eqn:lin_bgk:HOLO:hi}\\
    {n}^{\ell+1} + {u_0}\dt A {n}^{\ell+1} &= {n}_{f^{\{k\}}} - \dt A\vbr{ v(f^\ell-\mpM{n}_{f^\ell}) }. \label{eqn:lin_bgk:HOLO:lo}
\end{align}
\end{subequations}

To motivate \eqref{eqn:lin_bgk:HOLO}, we integrate \eqref{eqn:lin_bgk:BE} in $v$ and
then add $\dt A\vbr{v\mpM{n}_{f^{\{k+1\}}}}={u_0}\dt  A{n}_{f^{\{k+1\}}}$ to both sides.
These actions yield
\begin{equation} \label{eqn:lin_bgk:HOLO_deriv:1}
    {n}_{f^{\{k+1\}}} +{u_0}\dt A{n}_{f^{\{k+1\}}} = {n}_{f^{\{k\}}} - \dt A\vbr{v(f^{\{k+1\}}-\mpM{n}_{f^{\{k+1\}}})}.
\end{equation}
Lagging the right-hand side of \eqref{eqn:lin_bgk:HOLO_deriv:1} leads to \eqref{eqn:lin_bgk:HOLO:lo}.
The rightmost term of \eqref{eqn:lin_bgk:HOLO:lo} can be written as $-\dt A\vbr{v P_\mpM^\perp f^\ell}$.  
This fact is key to the convergence behavior of HOLO, and its role is show below.
\begin{prop}\label{prop:lin_bgk:HOLO_error}
    Define $e^{\ell} = f^{\ell} - f^{\{k+1\}}$ where $f^\ell$ is given in \eqref{eqn:lin_bgk:HOLO:hi}.
    Then for any $1\leq\delta\leq 2$,
    \begin{equation}\label{eqn:lin_bgk:HOLO_error:0}
        (1+\tfrac{2-\delta}{2}\nu\dt)\|P_\mpM e^{\ell+1}\|_\mpM^2 + (1+\nu\dt)\|P_\mpM^\perp e^{\ell+1}\|_\mpM^2 \leq \tfrac{\nu\dt}{2\delta} C_{\tHL} \|P_\mpM^\perp e^{\ell}\|_\mpM^2,
    \end{equation}
    where $C_{\mtHL} := \frac{(\theta_0)\dt^2/h_x^2}{1+u_0^2\dt^2/h_x^2}$.
    Moreover,
    \begin{equation}\label{eqn:lin_bgk:HOLO_error:1}\textstyle
        \|P_\mpM^\perp e^{\ell+1}\|_\mpM^2 \leq \frac{1}{4}C_{\tHL}\frac{\dt\nu}{1+\dt\nu}\|P_\mpM^\perp e^{\ell}\|_\mpM^2,\,\,\,\,\,
        \|P_\mpM e^{\ell+1}\|_\mpM^2 \leq C_{\tHL}\frac{\dt\nu}{2+\dt\nu}\|P_\mpM^\perp e^{\ell}\|_\mpM^2.
    \end{equation}
\end{prop}

\begin{proof}
    Let ${n}^{\{k+1\}} := {n}_{f^{\{k+1\}}}$.
    Similar to the proof of \Cref{prop:lin_bgk:SI_error}, we subtract \eqref{eqn:lin_bgk:BE} from \eqref{eqn:lin_bgk:HOLO:hi} and test by $e^{\ell+1}\mpM^{-1}$ which yields (c.f.~\eqref{eqn:lin_bgk:SI_error:3})
    \begin{equation}\label{eqn:lin_bgk:HOLO_error:3a}
        (1+\nu\dt)\big( \|P_\mpM e^{\ell+1}\|_\mpM^2 + \|P_\mpM^\perp e^{\ell+1}\|_\mpM^2 \big) = \nu\dt(\mpM({n}^{\ell+1}-{n}^{\{k+1\}}),P_\mpM e^{\ell+1})_\mpM.
    \end{equation}
    An application of H\"older's and Young's inequality with weight $1\leq\delta\leq 2$ yields
    \begin{align}\label{eqn:lin_bgk:HOLO_error:3b}
    \begin{split}
        (\mpM({n}^{\ell+1}-{n}^{\{k+1\}}),P_\mpM e^{\ell+1})_\mpM
        &\leq\|\mpM{n}^{\ell+1}-\mpM{n}^{\{k+1\}}\|_\mpM\|P_\mpM e^{\ell+1}\|_\mpM \\
        &\leq \tfrac{1}{2\delta}\|\mpM({n}^{\ell+1}-{n}^{\{k+1\}})\|_\mpM^2 + \tfrac{\delta}{2}\|P_\mpM e^{\ell+1}\|_\mpM^2 \\
        &= \tfrac{1}{2\delta}\|{n}^{\ell+1}-{n}^{\{k+1\}}\|_{\W_x}^2 + \tfrac{\delta}{2}\|P_\mpM e^{\ell+1}\|_\mpM^2. \\
    \end{split}
    \end{align}
    Applying \eqref{eqn:lin_bgk:HOLO_error:3b} to \eqref{eqn:lin_bgk:HOLO_error:3a} and rearranging yields
    \begin{equation}\label{eqn:lin_bgk:HOLO_error:3}
        (1+\tfrac{2-\delta}{2}\nu\dt)\|P_\mpM e^{\ell+1}\|_\mpM^2 + (1+\nu\dt)\|P_\mpM^\perp e^{\ell+1}\|_\mpM^2 \leq \tfrac{\nu\dt}{2\delta}\|{n}^{\ell+1}-{n}^{\{k+1\}}\|_{\W_x}^2.
    \end{equation}
    We will now bound the right-hand side of \eqref{eqn:lin_bgk:HOLO_error:3}. 
    Subtracting \eqref{eqn:lin_bgk:HOLO_deriv:1} from \eqref{eqn:lin_bgk:HOLO:lo} and noting $\vbr{P_\mpM^\perp e^{\ell}}=0$ we obtain the error equation for the low-order solve, i.e.,
    \begin{align}\label{eqn:lin_bgk:HOLO_error:4}
        (I+{u_0}\dt A)({n}^{\ell+1}-{n}^{\{k+1\}}) = -\dt A\vbr{vP_\mpM^\perp e^{\ell}} =-\dt A\vbr{(v-u_0)P_\mpM^\perp e^{\ell}}.
    \end{align}
    Let $A^\dagger = -\dt(I+{u_0}\dt A)^{-1}A$.
    Since $A$ is normal, $A^\dagger$ is normal and thus
    \begin{equation} \label{eqn:hyp_heat:HOLO_error:6}
        \|A^\dagger\|^2 = \max_{\lambda\in\sigma(A^\dagger)}|\lambda|^2 
        = \max_{i\lambda\in\sigma(A)} \big|\tfrac{i\dt\lambda}{1+i{u_0}\dt \lambda}\big|^2 \leq \max_{\lambda^2 \leq h_x^{-2}} \tfrac{\dt^2\lambda^2}{1+u_0^2\dt^2\lambda^2} \leq \tfrac{\dt^2/h_x^2}{1+u_0^2\dt^2/h_x^2},
    \end{equation}
    where $\sigma(B)$ denotes the spectrum of a matrix $B$.
    Therefore, 
    \begin{equation}\label{eqn:lin_bgk:HOLO_error:5}
        \|{n}^{\ell+1}-{n}^{\{k+1\}}\|_{\W_x}^2 \leq \frac{\dt^2/h_x^2}{1+u_0^2\dt^2/h_x^2}\|\vbr{(v-u_0)P_\mpM^\perp e^{\ell}}\|_{\W_x}^2.
    \end{equation}
    Since $
        \vbr{(v-u_0)P_\mpM^\perp e^{\ell}}^2 
        \leq \vbr{(v-u_0)^2\mpM}\Mip{P_\mpM^\perp e^\ell}{P_\mpM^\perp e^\ell} = \theta_0\Mip{P_\mpM^\perp e^\ell}{P_\mpM^\perp e^\ell},
    $
    \eqref{eqn:lin_bgk:HOLO_error:5} becomes
    \begin{equation}\label{eqn:lin_bgk:HOLO_error:7}
        \|{n}^{\ell+1}-{n}^{\{k+1\}}\|_{\W_x}^2 \leq \frac{\theta_0\dt^2/h_x^2}{1+u_0^2\dt^2/h_x^2}\|P_\mpM^\perp e^\ell\|_\mpM^2= C_\tHL\|P_\mpM^\perp e^\ell\|_\mpM^2.
    \end{equation}
    Combining \eqref{eqn:lin_bgk:HOLO_error:3} and \eqref{eqn:lin_bgk:HOLO_error:7} yield \eqref{eqn:lin_bgk:HOLO_error:0}.
    The bounds in \eqref{eqn:lin_bgk:HOLO_error:1} follow from \eqref{eqn:lin_bgk:HOLO_error:0} and setting $\delta = 2$ and $\delta=1$. 
    The proof is complete.
\end{proof}

\begin{remark}
We consider the case when $\nu\dt\gg 1$.
Then \Cref{prop:lin_bgk:SI_error,prop:lin_bgk:HOLO_error} show the contraction constants of SI and HOLO are close to 1 and $\tfrac{1}{2}\sqrt{C_\mtHL}$ respectively.
If $C_\mtHL$ is well controlled, then the contraction constant of HOLO is bounded away from 1, and thus we expect HOLO to perform better than SI.
However, there are choices of ${u_0}$ and ${\theta_0}$ such that $\sqrt{C_\mtHL}$ is directly proportional to $\dt/h_x$.  Hence HOLO, unlike SI, is only conditionally stable.
\end{remark}

\subsection{Micro-macro methods}

We now analyze the MM methods.
Applying the MM-L approach to \eqref{eqn:lin_bgk:BE} yields the following method: \textit{Given $\{{n}^\ell,$ $\!g^\ell\}$, find $\{{n}^{\ell+1},$ $\!g^{\ell+1}\}$ such that}
\begin{subequations}\label{eqn:lin_bgk:MM_SI}
\begin{align} \label{eqn:lin_bgk:MM_SI:rho}
    {n}^{\ell+1} + {u_0}\dt A {n}^{\ell+1} &= {n}_{f^{\{k\}}} - \dt A\vbr{ vg^\ell }, \\
    g^{\ell+1} + \dt vAg^{\ell+1} + \nu\dt g^{\ell+1} &= \mpM{n}^{\{k\}}+g^{\{k\}} - \mpM{n}^{\ell+1} - \dt vA(\mpM{n}^{\ell+1}). \label{eqn:lin_bgk:MM_SI:g}
\end{align}
\end{subequations}
If ${n}^\ell$ and $g^\ell$ from \eqref{eqn:lin_bgk:MM_SI} converge to ${n}^*$ and $g^*$ respectively, then ${n}_{g^*}=0$; moreover, $f^*=\mpM{n}^*+g^*$ solves \eqref{eqn:lin_bgk:BE}.

We now explain the poor performance of MM-L, which is  demonstrated in \Cref{subsec:riemann} and primarily caused by the low-order solve \eqref{eqn:lin_bgk:MM_SI:rho}.
The high-order solve presents no issue since, by letting $f^\ell = \mpM{n}^\ell + g^\ell$, \eqref{eqn:lin_bgk:MM_SI:g} reduces to the high-order solve of HOLO; namely, \eqref{eqn:lin_bgk:HOLO:hi}.
In the HOLO low-order solve, \eqref{eqn:lin_bgk:HOLO:lo}, the density ${n}^{\ell+1}$ is a function of $P_\mpM^\perp f^\ell$; therefore, the error in ${n}^{\ell+1}$ is bounded by $P_M^\perp e^\ell$, see \eqref{eqn:lin_bgk:HOLO_error:7}.
However, for MM-L, one has $g^\ell = P_\mpM^\perp f^\ell + \mpM{n}_{g^\ell}$, and if $n_{g^\ell}\neq 0$, which is often the case, then the latter term does not vanish.
Following a similar strategy as the proof of \Cref{prop:lin_bgk:HOLO_error}, an analog of \eqref{eqn:lin_bgk:HOLO_error:7} for MM-L can be derived; namely,
\begin{equation}\label{eqn:lin_bgk:MM_SI:err_est}
    \|{n}^{\ell+1}-{n}^{\{k+1\}}\|_{\W_x}^2 \leq C_\tHL(\|P_\mpM^\perp e^\ell\|_\mpM^2 + \|n_{g^\ell}\|_{\W_x}^2).
\end{equation}
When $\nu\gg 1$, we conjecture that ${n}_{g^\ell}$ is sufficiently small such that the convergence rate of the MM-L method is not harmed. 
However, as $\nu$ becomes smaller, $\|n_{g^\ell}\|_{\W_x}$ becomes the dominant term in \eqref{eqn:lin_bgk:MM_SI:err_est} and convergence will most likely stagnate.

The MM-HOLO method is similar to the MM-L method, but the low-order solve \eqref{eqn:lin_bgk:MM_SI:rho} is instead given by 
\begin{equation}\label{eqn:lin_bgk:MM_HOLO}
\begin{split}
    {n}^{\ell+1} + {u_0}\dt A {n}^{\ell+1} &= {n}_{f^{\{k\}}} - \dt A\vbr{ v(g^\ell-\mpM{n}_{g^\ell}) }.
\end{split}
\end{equation}
Since $g^\ell-\mpM{n}_{g^\ell} = P_\mpM^\perp f^\ell$, the error of ${n}^{\ell+1}$ from \eqref{eqn:lin_bgk:MM_HOLO} can be closed in terms of $P_\mpM^\perp e^\ell$.
In fact, the MM-HOLO method is equivalent to the HOLO method \eqref{eqn:lin_bgk:HOLO} due to the linearity of \eqref{eqn:lin_bgk:system} and of the lack of a velocity discretization.


\section{Numerical results}\label{sec:numerical}

In this section we numerically verify the claims and analysis given in the preceding sections. 
We test the above methods on two example problems: the Sod shock tube problem \cite{sod1978survey} and a 1D-1V
variation of the sudden wall heating boundary layer problem in \cite{aoki1991numerical}. 
For all tests in this section, we set $\kappa=2$ for $V_{x,h}$ in \eqref{eqn:discrete_space_x}.

The transport solves --- \eqref{eqn:BE_SI}, \eqref{eqn:BE_HOLO:HI}, \eqref{eqn:BE_MM_SI:g}, and \eqref{eqn:BE_MM:g} --- are all linear problems that are inverted using sweeping methods \cite{adams2002fast}.
The nonlinear fluid equations --- \eqref{eqn:BE_HOLO:LO}, \eqref{eqn:BE_MM_SI:rho}, and \eqref{eqn:BE_MM:rho} --- are  
computed using a Jacobian-free Newton-Krylov (JFNK) solver.  
Unless otherwise stated, the JFNK solver exits when the residual is below a specified threshold which we set as $10^{-2}$ times the stopping criterion for the iterative methods (see  \eqref{eqn:stopping_criterion_SI}).

\subsection{Time stepping methods}\label{subsec:time-stepping}

For higher-order time integration, we use the diagonally implicit Runge-Kutta (DIRK) method of third order that is L-stable; see, for example, \cite{kennedy2016diagonally,alexander1977diagonally}. 
An $s$-stage RK method is expressed by the Butcher tableau
\begin{align}\label{eqn:butcher_tableau}
    \begin{array}{c | c }
        b & A \\ \hline
          & c
    \end{array}
    \qquad\qquad\qquad\qquad
    \begin{array}{c|c c c c}
    \alpha_3 & \alpha_3 & 0 & 0 \\
    \tfrac{1+\alpha_3}{2} & \tfrac{1-\alpha_3}{2} & \alpha_3 & 0 \\ 
    1 & \gamma_1 & \gamma_2 & \alpha_3 \\ \hline
    & \gamma_1 & \gamma_2 & \alpha_3
    \end{array}
\end{align}
where $A=[a_{ij}]\in\mathbb{R}^{s\times s}$, $b=[b_i]\in\mathbb{R}^s$, and $c=[c_i]\in\mathbb{R}^s$. 
The generic tableau on the left of \eqref{eqn:butcher_tableau} corresponds to an RK method on the ODE $y'(t)=F(t,y)$ given by
\begin{subequations}\label{eqn:RK_method}
\begin{align}
    y_i^{\{k\}} &= y^{\{k\}} +  \dt\textstyle{\sum_{j=1}^s} a_{ij} F( t^{\{k\}} + c_i\dt,y_j^{\{k\}}),\quad i=1,\ldots,s \label{eqn:RK_stages}\\
    y^{\{k+1\}} &= y^{\{k\}} + \dt\textstyle{\sum_{i=1}^s} \,\,\,b_iF( t^{\{k\}} + c_i\dt,y_i^{\{k\}}).\label{eqn:RK_update}
\end{align}
\end{subequations}
For DIRK methods, the matrix $A$ is upper triangular so that each solve in \eqref{eqn:RK_stages} is sequential and the only timestep treated implicitly per stage is $y_i^{\{k\}}$.
The other terms in \eqref{eqn:RK_stages} for $j<i$ are treated as a source.
Therefore, the SI, HOLO, and MM iterative techniques derived for the backward Euler method \eqref{eqn:BE_BGK}
are sufficient for each solve in \eqref{eqn:RK_stages} via a rescale of the timestep $\dt$ to $a_{ii}\dt$ and the addition of an external source. 
Additionally, each solve in \eqref{eqn:RK_stages}  is initialized with $y_{i-1}^{\{k\}}$ where $y_0^{\{k\}}:=y^{\{k\}}$.

For the MM methods \eqref{eqn:BE_MM_SI} and \eqref{eqn:BE_MM}, the external source is built from \eqref{eqn:BGK_MM_BE:g}, and then its moments are taken as the source in \eqref{eqn:BGK_MM_BE:rho}.
This treatment avoids the propagation of errors from $\langle \bme g\rangle_v\approx 0$ within a timestep.

On the right-hand side of \eqref{eqn:butcher_tableau} is the tableau for the DIRK3 method used in this paper, where $\gamma_1 = -\tfrac{1}{4}(6\alpha_3^2-16\alpha_3+1)$, $\gamma_2 = \tfrac{1}{4}(6\alpha_3^2-20\alpha_3+5)$, and $\alpha_3\approx 0.4358665$ is the root of $\alpha^3-3\alpha^2+\tfrac{3}{2}\alpha-\tfrac{1}{6} = 0$ lying in $(\tfrac16,\tfrac12)$.
For consistency in this section, we refer to the backward Euler method as DIRK1.

\subsection{Sod shock tube problem}\label{subsec:riemann}

The Sod shock tube problem is a standard test for the Euler equations and collisional kinetic models \cite{sod1978survey}.  In the kinetic setting, this test poses a Maxwellian initial condition with a discontinuity in the fluid variables, given by
\begin{align}\label{eqn:riemann:initial_fluid_vars}
    (n, u, \theta)^\top
    = (1, 0, 1)^\top
    ~~\text{if}~x \leq 0;
    \qquad
    (n, u, \theta)^\top
    =(0.125, 0, 0.8)^\top
    ~~\text{if}~x > 0,
\end{align}  
where $x\in\W_x=(-1,1)$.  
We set $N_x=256$ and use far-field boundary conditions \eqref{eqn:far-field-bc} on both left and right boundaries. 
We set the truncated velocity domain as $\Omega_v=(-6,6)$.
Unless otherwise stated, we set $N_v=32$ and $\dt = 3.125\times 10^{-3}$ and use a backward Euler (DIRK1) method.
For the DG method with $\kappa$-degree polynomials, 
\begin{equation}\label{eqn:dt_expl}\textstyle
    \dt_{\textrm{expl}} := \frac{1}{2\kappa+1}\frac{1}{v_{\max}}h_x 
\end{equation}
is the usual maximum timestep for an explicit method to remain stable.  In this case $\dt_{\textrm{expl}} = \tfrac{1}{5}\tfrac{1}{6}h_x \approx 2.60\times 10^{-4}$, which is 12 times smaller than $\dt$.

\subsubsection{Consistency of HOLO method}
We first test that the discretization of the HOLO method \eqref{eqn:BE_HOLO} is consistent with the SI method \eqref{eqn:BE_SI} in the sense that the limit of the HOLO method satisfies \eqref{eqn:BE_BGK}.
We set $\nu=\frac{1}{2\dt}$ and perform exactly one timestep for SI and for HOLO.
We iterate SI to $\ell=26$ which produces moments $\bmrho^\tSI$ with a relative residual $\|\mathcal{R}f^{27}\|/\|f^{27}\|=1.29\times10^{-13}$, where $\mathcal{R}$ is defined in \eqref{eqn:BE_residual}.
We then run HOLO acceleration with the same parameters until stagnation is reached at 16 iterations. 
We set the exit threshold for the JFNK solver used determine $\bmrho^{\ell+1}$ in \eqref{eqn:BE_HOLO:LO} to $10^{-14}$.

In \Cref{tab:holo_per_l} we list several quantities of interest.
The first two columns compare two possible termination criteria for HOLO. 
The first column reports the relative $L^2$ difference in the moments of $f^\ell$ between iterations, that is,
\begin{equation}\label{eqn:stopping_criterion_SI}
    \frac{\|\bmrho_{f^{\ell+1}}-\bmrho_{f^{\ell}}\|_{\W_x}}{\|\bmrho_{f^{\ell+1}}\|_{\W_x}} < \text{tol}.
\end{equation}
Condition \eqref{eqn:stopping_criterion_SI} is a standard termination criterion for SI.
The second column uses the relative $L^2$ difference between $\bmrho_{f^{\ell}}$ and the accelerated moments $\bmrho^{\ell+1}$, that is,
\begin{equation}\label{eqn:stopping_criterion_HOLO}
    \frac{\|\bmrho^{\ell+1}-\bmrho_{f^{\ell}}\|_{\W_x}}{\|\bmrho_{f^{\ell+1}}\|_{\W_x}} < \text{tol}.
\end{equation}
The authors in \cite{taitano2014moment} used a version of \eqref{eqn:stopping_criterion_HOLO} to terminate the HOLO method. 
The latter three columns compare the moments of the fluid solve in HOLO \eqref{eqn:BE_HOLO:LO} to the converged SI moments.  
The results in \Cref{tab:holo_per_l} demonstrate that (i) the HOLO and SI approximations agree and (ii) the DG method naturally provides the consistency that, in a finite difference setting, requires additional consistency terms \cite{taitano2014moment}.
While \eqref{eqn:stopping_criterion_HOLO} could be used in lieu of \eqref{eqn:stopping_criterion_SI} for the HOLO method, we will continue to use \eqref{eqn:stopping_criterion_SI} as the termination criterion for all methods for the rest of the paper.

\begin{table}[ht]
    \centering
    \begin{tabular}{ c || c c | c c c}
         & \multicolumn{2}{c|}{Criterion}
         & \multicolumn{3}{c}{$\dfrac{\|\rho^{d,\ell+1}-\rho^{d,\textrm{SI}}\|_{\W_x}}{\|\rho^{d,\textrm{SI}}\|_{\W_x}}$} \\
         $\ell$ & 
         \eqref{eqn:stopping_criterion_SI}  &
         \eqref{eqn:stopping_criterion_HOLO} &
         $d=1$ &
         $d=2$ &
         $d=3$ \\ \hline
         0 &   2.86e-02 &   2.91e-02 &   2.06e-03 &   6.67e-02 &   1.70e-03 \\
         4 &   4.16e-07 &   4.92e-07 &  1.40e-07 &   3.50e-06 &   6.88e-08 \\
         8 &   1.01e-10 &  1.32e-10 &  5.00e-11 &   9.47e-10 &   1.24e-11 \\
         12 &   5.06e-14  & 6.74e-14 &  2.71e-14 &   5.00e-13 &   8.36e-15 \\
         16 &   1.02e-15 &   1.59e-15 &  2.49e-15 &   1.38e-13 &   5.42e-15
    \end{tabular}
    \caption{
    Sod shock tube (\Cref{subsec:riemann}): Consistency of the HOLO method when compared to the SI method. The first two columns report the relative difference of moments of HOLO under two different metrics at iteration $\ell$. The last three columns report the relative difference of the moments of HOLO versus SI.
    Here $\rho^{d,\tSI}$ and $\rho^{d,\ell+1}$ correspond to the $d$-th components of moments of SI \eqref{eqn:BE_SI} and the low-order moments $\bmrho^{\ell+1}$ from HOLO \eqref{eqn:BE_HOLO}, respectively.
    The moments of SI are converged to a relative residual of $1.29\times10^{-13}$. 
    At convergence, the HOLO and SI iterations agree up to the SI residual.
    }
    \label{tab:holo_per_l}
\end{table}

We now provide a case where HOLO is inconsistent in the $\ell$-limit. 
\Cref{prop:HOLO_moment_convergence,prop:HOLO_vs_SI_equiv} prove the consistency of HOLO to SI if we assume $\mE^*$ from \eqref{eqn:euler_timestep_map} is injective.
However, it is well-known that the Euler flux $\bmF$ in \eqref{eqn:euler_flux} is indefinite; that is, $\partial_\bmeta\bmF(\bmeta)$ can have both positive and negative eigenvalues.
Hence, as $\dt\nu$ remains constant and $\dt$ increases, we expect $\mE^*$ to eventually be non-injective, in which case the conclusions of \Cref{prop:HOLO_moment_convergence,prop:HOLO_vs_SI_equiv} may not hold.
To test the hypothesis above, we run a single timestep for HOLO and SI for increasing $\dt$ and fixing $\nu = \frac{1}{2\dt}$ so that $\dt\nu = 1/2$ is constant across runs.
For SI, we iterate the method until $\|\mR f^{\ell}\|/\|f^{\ell}\|\leq 10^{-9}$.
For HOLO, we iterate until \eqref{eqn:stopping_criterion_SI} is satisfied with a tolerance of $10^{-8}$.
In \Cref{tab:holo_vs_SI_break}, we list several metrics for the HOLO iterates, including the stopping criteria \eqref{eqn:stopping_criterion_SI} and \eqref{eqn:stopping_criterion_HOLO}, the relative residual, and the relative low-order and distribution moment errors against the converged SI moments. 
For $\dt\leq 2\times 10^{-2}$, HOLO is consistent to SI up to the tolerance of $10^{-8}$.
However, as $\dt$ increases, the HOLO method continues to converge in terms of \eqref{eqn:stopping_criterion_SI}, but the consistency error increases.
Based on the analysis in \Cref{prop:HOLO_moment_convergence,prop:HOLO_vs_SI_equiv}, we conjecture that this lack of consistency is because $\dt$ is large enough so that $\mE^*$ is no longer injective.
Fortunately, the stopping criterion for HOLO \eqref{eqn:stopping_criterion_HOLO} exactly measures this inconsistency.

\begin{table}[ht]
    \centering
    {\setlength{\tabcolsep}{4pt}
    \begin{tabular}{ c || c | c c | c@{\hspace{-2pt}}c@{\hspace{-2pt}}c}
         & & \multicolumn{2}{c|}{Criterion}
         &
         \multirow{2}{*}{ $\dfrac{\|\mathcal{R}f^{\ell+1}\|}{\|f^{\ell+1}\|}$ }
          &
         \multirow{2}{*}{ $\dfrac{\|\bmrho_{f^{\ell+1}}-\bmrho^\tSI\|_{\W_x}}{\|\bmrho^\tSI \|_{\W_x}}$ }
         &
         \multirow{2}{*}{ $\dfrac{\|\bmrho^{\ell+1}-\bmrho^\tSI\|_{\W_x}}{\|\bmrho^\tSI\|_{\W_x}}$ }
         \\ 
         $\dt$ & $\ell$ & \eqref{eqn:stopping_criterion_SI} & \eqref{eqn:stopping_criterion_HOLO}
         &
         &
         \\[0.2em] \hline
         5.00e-03 & 7 & 5.57e-09 & 8.63e-09 & 2.56e-09 & 9.31e-10 & 4.33e-09 \\
         1.00e-02 & 8 & 4.77e-09 & 1.07e-08 & 4.70e-09 & 1.07e-09 & 8.37e-09 \\
         2.00e-02 & 7 & 6.35e-09 & 1.64e-08 & 8.45e-09 & 1.34e-09 & 1.56e-08 \\
         4.00e-02 & 7 & 7.30e-09 & 6.81e-08 & 4.60e-07 & 5.00e-07 & 4.76e-07 \\
         8.00e-02 & 11 & 7.26e-09 & 1.47e-05 & 1.59e-04 & 1.95e-04 & 1.90e-04
    \end{tabular}
    }
    \caption{
    Sod shock tube (\Cref{subsec:riemann}): Consistency of the HOLO method as $\dt$ increases while $\dt\nu$ remains constant.
    Here SI \eqref{eqn:BE_SI} with moments denoted by $\bmrho^\tSI$ is iterated until the relative residual is below $10^{-9}$.
    The HOLO method \eqref{eqn:BE_HOLO} terminates at iteration $\ell$ when \eqref{eqn:stopping_criterion_SI} is below $10^{-8}$.  
    As $\dt$ increases with $\dt\nu$ fixed, the consistency of HOLO to SI is lost.
    }
    \label{tab:holo_vs_SI_break}
\end{table}

\subsubsection{Comparison of iteration counts}

We next compare the number of iterations for the four methods listed in \Cref{sect:solvers}: SI \eqref{eqn:BE_SI}, HOLO \eqref{eqn:BE_HOLO}, MM-L \eqref{eqn:BE_MM_SI}, and MM-HOLO \eqref{eqn:BE_MM}.  
We run ten timesteps for $\nu$ such that $\dt\nu$ ranges from $10^{-1}$ to $10^{5}$.
In each timestep, we run until the stopping criterion specified in \eqref{eqn:stopping_criterion_SI} is less than $10^{-8}$.
For MM-L and MM-HOLO, $\bmrho_{f^\ell} := \bmrho^\ell + \bmrho_{g^\ell}$ is used in \eqref{eqn:stopping_criterion_SI}.

\Cref{tab:iterates_vs_dt_tau} shows the average number of iterations per timestep for each method.
The SI method performs as expected: the number of iterations to achieve convergence worsens with larger $\dt\nu$, which suggests the contraction constant in the nonlinear case takes the same form as the one in \eqref{eqn:lin_bgk:SI_error:1}.
In highly-collisional regimes, this constant is close to 1, leading to prohibitive iteration counts.
Average iterations for the HOLO method are lower than SI for each collision frequency listed.
Moreover, HOLO is far superior to SI in the moderate to high collisional regimes which agrees with the formal estimates in for linear case (see \Cref{prop:lin_bgk:HOLO_error}).
MM-L carries the opposite problem as SI -- the iteration count is only viable in high to moderate collisional regimes and the performance falls off as the collision frequency is lowered.  
Once $\nu\dt < 1$, the MM-L method does not even converge, most likely because in this regime $g^\ell$ is not sufficiently small and thus the inconsistency $\langle\bme g^{\ell}\rangle_v \neq 0$ is not negligible.
This shows that MM-L is impractical when compared to HOLO or SI, and we do not consider the MM-L method for any further numerical results.
Finally, adding a proper heat flux correction term in the MM-HOLO method \eqref{eqn:BE_MM:rho} fixes the issues with MM-L in moderate to low collisional regimes.
We find that MM-HOLO and HOLO perform similarly when $\nu\dt\approx 1$.
If $\nu\dt \gg 1$, then MM-HOLO is slightly better.
When $\nu\dt \ll 1$, MM-HOLO convergence is slightly worse.

\begin{table}[ht]
    \centering
    {\setlength\tabcolsep{5 pt}
    \begin{tabular}{ l | c c c c c c c c c c}
         $\dt\nu$ & $10^{-4}$ & $10^{-3}$ & $10^{-2}$ & $10^{-1}$ & $10^{0}$ & $10^{1}$ & $10^{2}$ & $10^3$ & $10^4$ \\ 
         $\nu = 3.2\times 10^{n}, n = $ & $-2$ & $-1$ & $0$ & $1$ & $2$ & $3$ & $4$ & $5$ & $6$\\ \hline
         SI\hfill\eqref{eqn:BE_SI} & 3 & 3.6 & 4.4 & 7 & 20.2 & 123.8 & $>900$ & -- & -- \\ 
         HOLO\hfill\eqref{eqn:BE_HOLO} & 3 & 3 & 3.7 & 4.8 & 7.1 & 8.3 & 6.5 & 6.5 & 6.5 \\ 
         MM-L\hfill\eqref{eqn:BE_MM_SI} & -- & -- & -- & DNC & 43.3 & 11.9 & 5.1 & 4.5 & 3 \\
         MM-HOLO\hfill\eqref{eqn:BE_MM} & 5.1 & 5.1 & 5.1 & 5.4 & 7.2 & 8.1 & 4.7 & 3.4 & 3
    \end{tabular}
    }
    \caption{
    Sod shock tube (\Cref{subsec:riemann}): The average number of iterations per timestep over 10 timesteps for each method applied to the Sod shock tube with tolerance $10^{-8}$. 
    Here, DNC stands for ``did not converge'', and listings of ``~--~'' denote that the run was not attempted.  
    Unlike SI and MM-L, the HOLO and MM-HOLO methods are feasible over all collision scales.
    }
    \label{tab:iterates_vs_dt_tau}
\end{table}

\subsubsection{Compression benefits of the MM-HOLO method}
\label{sec:LR_Riemann}

We now test the compression benefits of the MM-HOLO method \eqref{eqn:BE_MM} versus HOLO acceleration \eqref{eqn:BE_HOLO} over multiple collision scales.
We first show that in areas of high collisionality, the micro perturbation $g$ is small.
\Cref{fig:riemann:g_surf} plots the micro distribution $g$ for $N_v=64$ at $t=0.1$ for both $\nu=10^{2}$ and $10^{4}$ and verifies that $g=\mathcal{O}(\nu^{-1})$.
\begin{figure}
    \centering
    \begin{subfigure}{0.45\textwidth}
        \includegraphics[width=\textwidth]{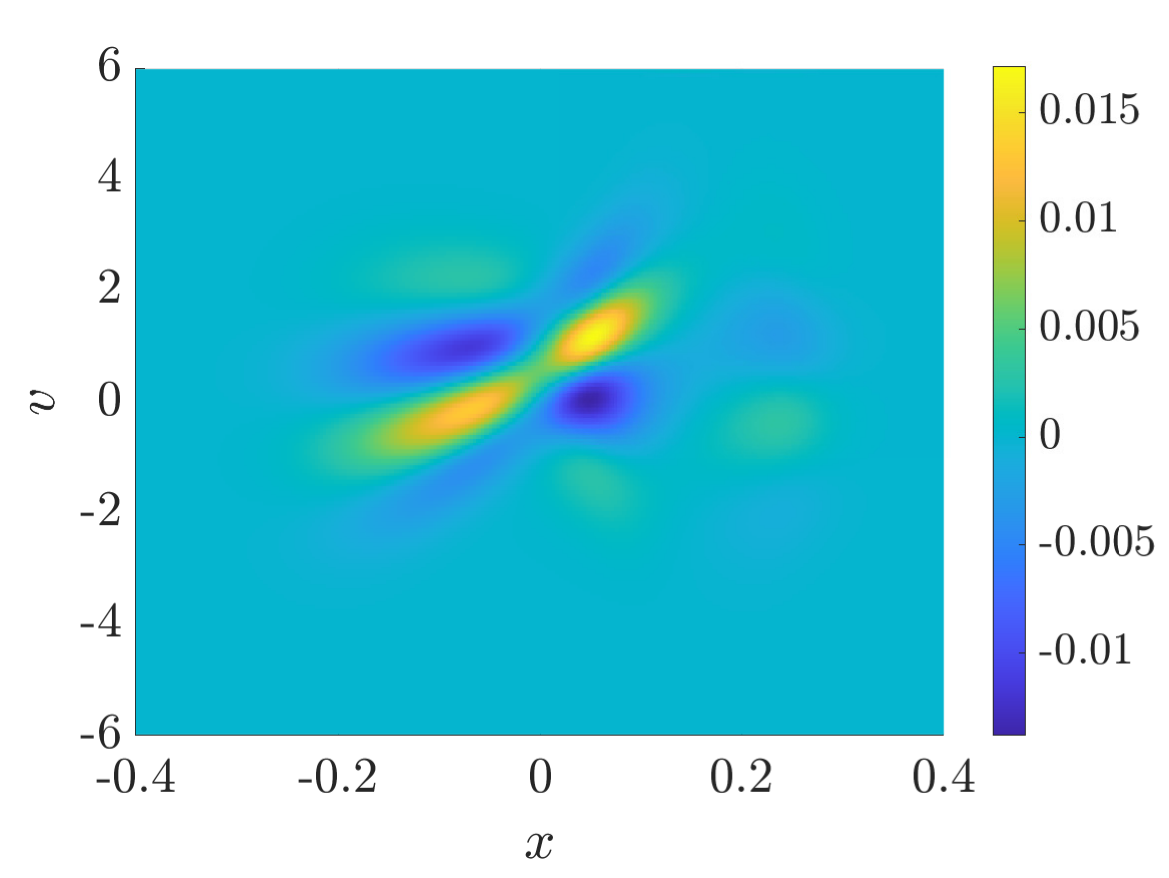}
        \caption{$\nu = 10^{2}$}
        \label{fig:riemann:g_surf_tau_1e-2_DIRK1}
    \end{subfigure}
    \begin{subfigure}{0.45\textwidth}
        \includegraphics[width=\textwidth]{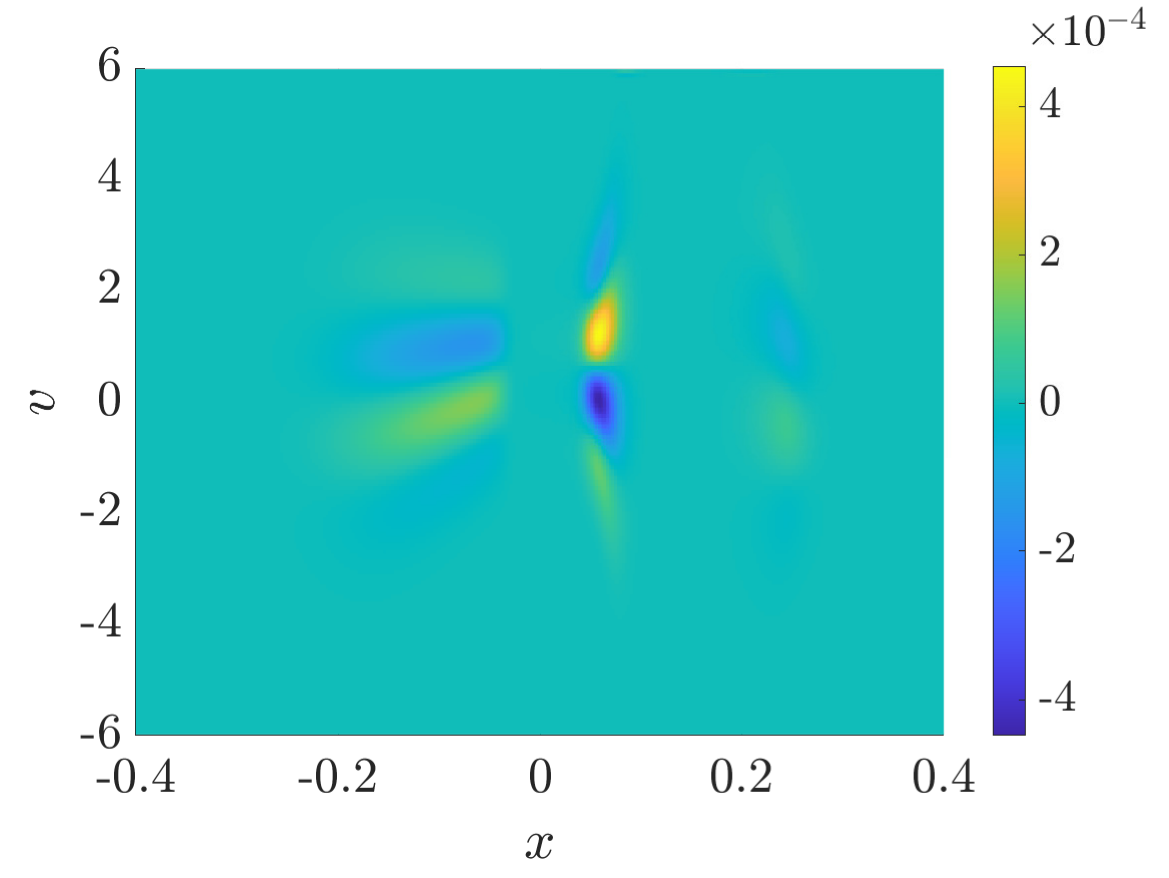}
        \caption{$\nu = 10^{4}$}
        \label{fig:riemann:g_surf_tau_1e-4_DIRK1}
    \end{subfigure}
    \caption{
    Sod shock tube (\Cref{subsec:riemann}): Plots of the converged micro distribution $g$ from the MM-HOLO method. 
    Here $N_v=64$ and $t=0.1$.  As $\nu$ increases, $g$ becomes smaller.
    } 
    \label{fig:riemann:g_surf}
\end{figure}
Since only the perturbation $g$ of the MM-HOLO ansatz $M(\bmrho)+g$ is discretized in phase-space,
when $\nu$ is large, we expect the MM-HOLO method to be more compressible than the HOLO approximation of $f$.
We verify this claim using two compression methods: coarse velocity discretization (following the approach of \cite[\S 6.2]{endeve2022conservative}) and low-rank approximation. 

\paragraph{Coarse velocity discretization}
Both the HOLO and MM-HOLO methods are run for 32 timesteps to a final time $t=0.1$ for $\nu\in\{10^{2},10^{3},10^{4}\}$ and $N_v\in\{4,6,8,10,12\}$ with DIRK1 and DIRK3 schemes.
Within a stage in a timestep, the HOLO and MM-HOLO methods terminate when \eqref{eqn:stopping_criterion_SI} is satisfied with a tolerance of $10^{-8}$. 
To build a reference solution, we use the average of the HOLO and MM-HOLO solutions at $t=0.1$ with $N_v=64$.
Each plot in \Cref{fig:riemann:HOLO_vs_MM_compression} reports the number of velocity degrees of freedom (DOF) per physical DOF versus the relative $L^2$ error against the reference fluid variables.
For HOLO, the discrete distribution $f\in V_h$ has a velocity DOF of $3N_v$ per physical DOF.  
Since MM-HOLO requires storage of both the moments $\bmrho\in\Vxhc$ and the micro distribution $g\in V_h$, its velocity DOF per physical DOF is $3N_v+3$.

When $\nu=10^{2}$, \Cref{fig:riemann:tau_1e-2_DIRK1,fig:riemann:tau_1e-2_DIRK3} show that the HOLO and MM-HOLO methods are largely comparable.
In this case, $f$ is still sufficiently far away from the Maxwellian such that the micro perturbation $g$ is sufficiently large and contains finer-level detail in $v$ that is necessary for accuracy.
When $\nu=10^{3}$, the results in \Cref{fig:riemann:tau_1e-3_DIRK1,fig:riemann:tau_1e-3_DIRK3} start to show the compression benefits of MM-HOLO over HOLO; this is especially evident in the DIRK3 method.
Finally, with $\nu=10^{4}$, \Cref{fig:riemann:tau_1e-4_DIRK1,fig:riemann:tau_1e-4_DIRK3} show the largest improvement in MM-HOLO over HOLO for lower $N_v$.
In particular, the MM-HOLO method saturates at $N_v=6$ for DIRK3 while HOLO requires a velocity resolution of $N_v=10$ to reach the same error.

\begin{figure}
    \captionsetup[subfigure]{aboveskip=5pt,belowskip=5pt}
    \centering
    \begin{subfigure}{0.32\textwidth}
        \includegraphics[width=\textwidth]{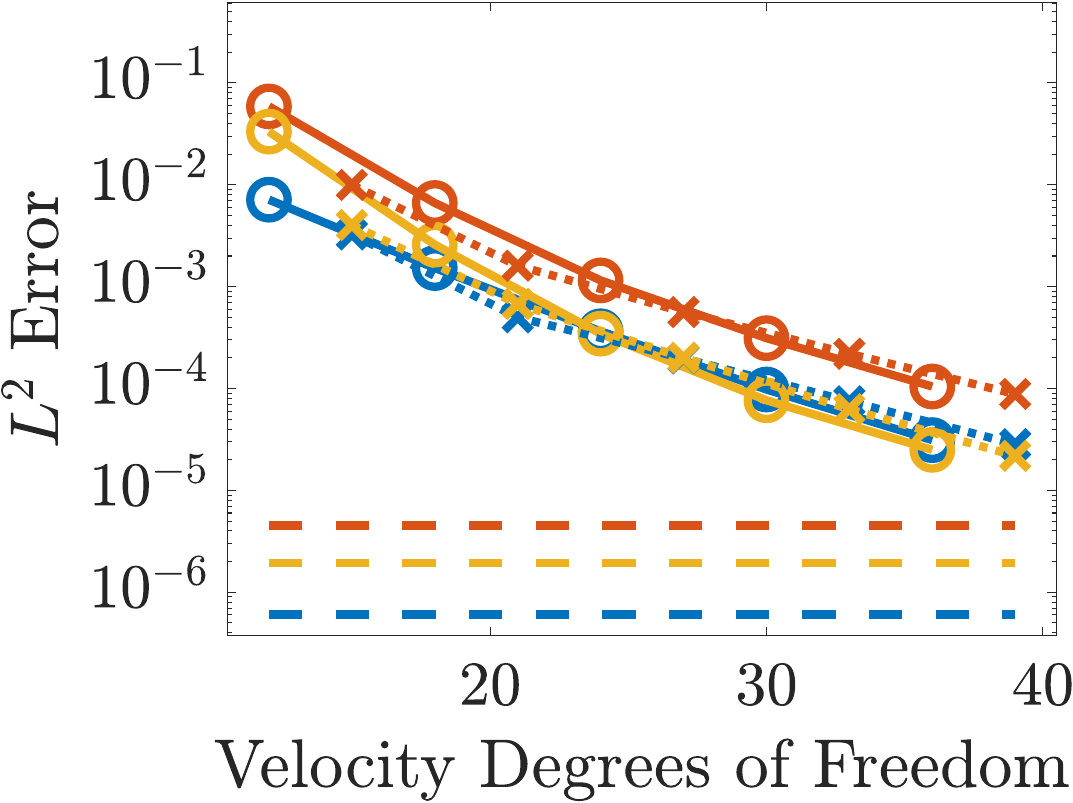}
        \caption{DIRK1 -- $\nu = 10^{2}$}
        \label{fig:riemann:tau_1e-2_DIRK1}
    \end{subfigure}
    \begin{subfigure}{0.32\textwidth}
        \includegraphics[width=\textwidth]{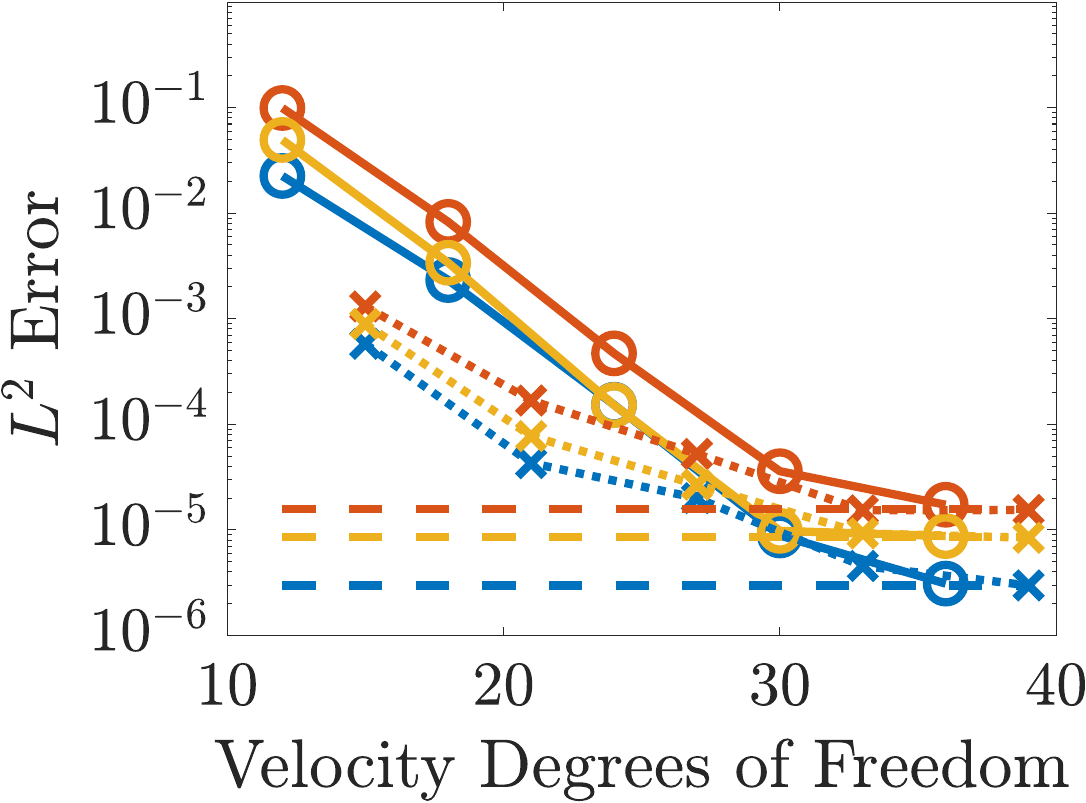}
        \caption{DIRK1 -- $\nu = 10^{3}$}
        \label{fig:riemann:tau_1e-3_DIRK1}
    \end{subfigure}
    \begin{subfigure}{0.32\textwidth}
        \includegraphics[width=\textwidth]{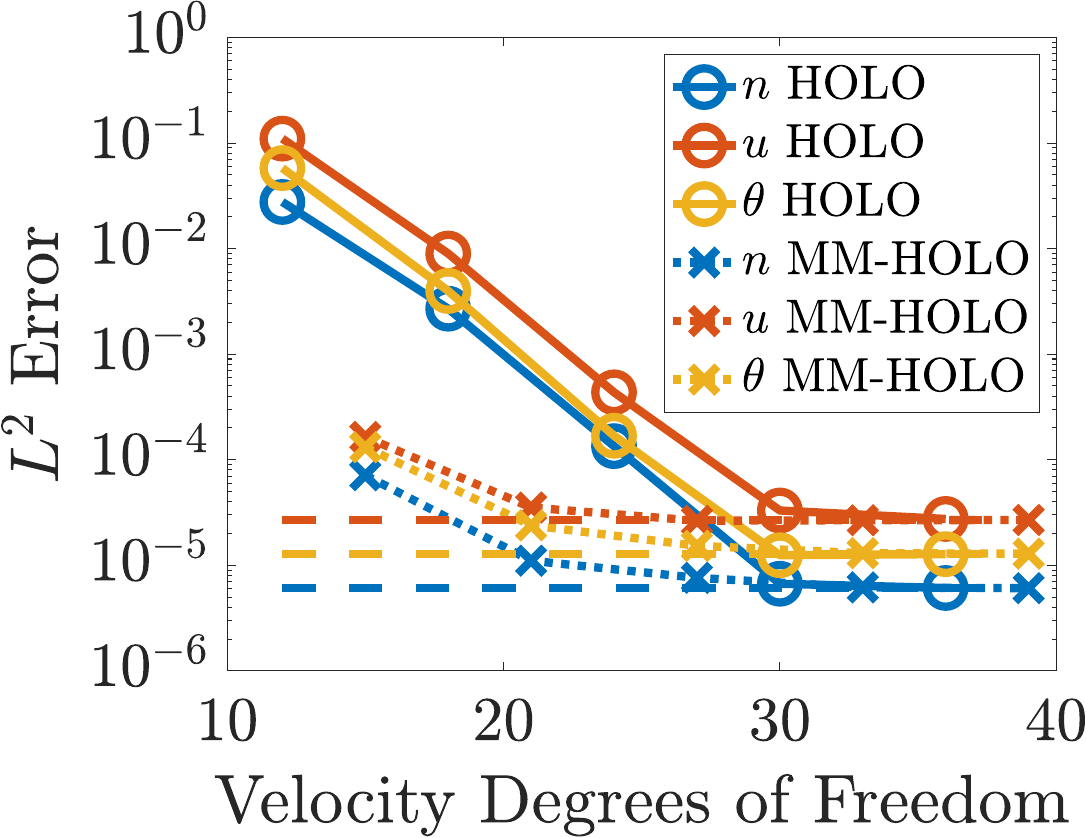}
        \caption{DIRK1 -- $\nu = 10^{4}$}
        \label{fig:riemann:tau_1e-4_DIRK1}
    \end{subfigure}

    \begin{subfigure}{0.32\textwidth}
        \includegraphics[width=\textwidth]{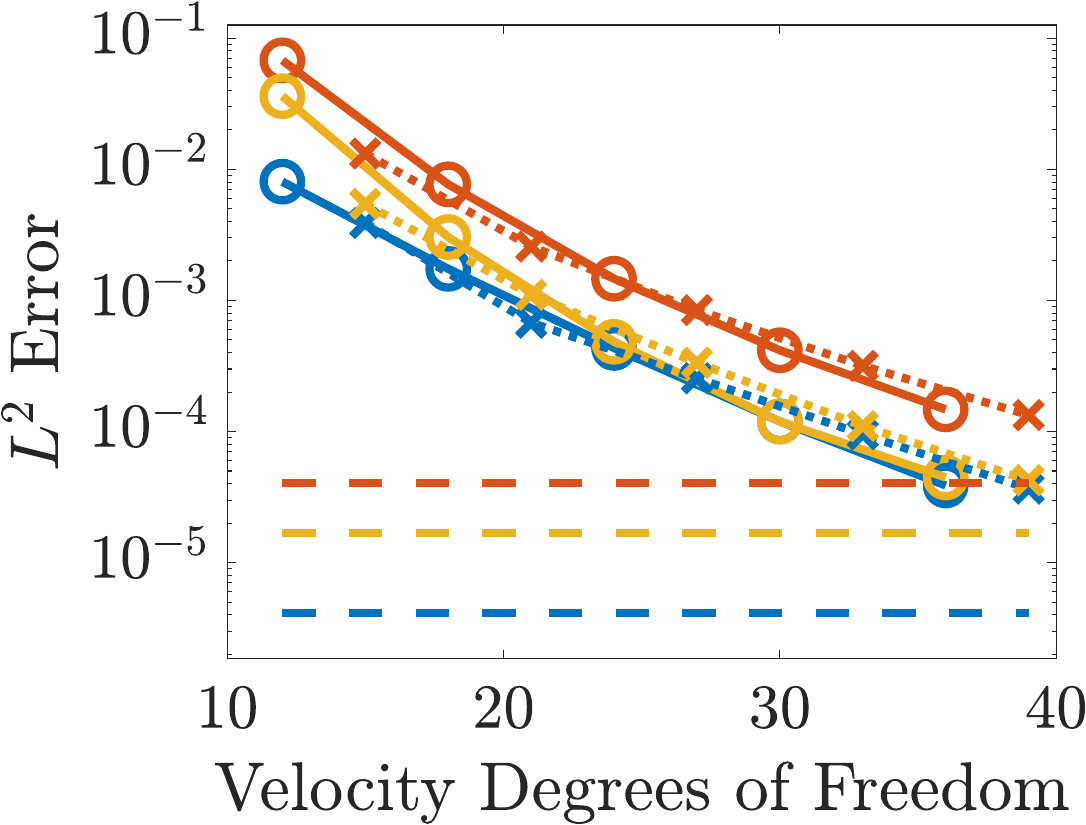}
        \caption{DIRK3 -- $\nu = 10^{2}$}
        \label{fig:riemann:tau_1e-2_DIRK3}
    \end{subfigure}
    \begin{subfigure}{0.32\textwidth}
        \includegraphics[width=\textwidth]{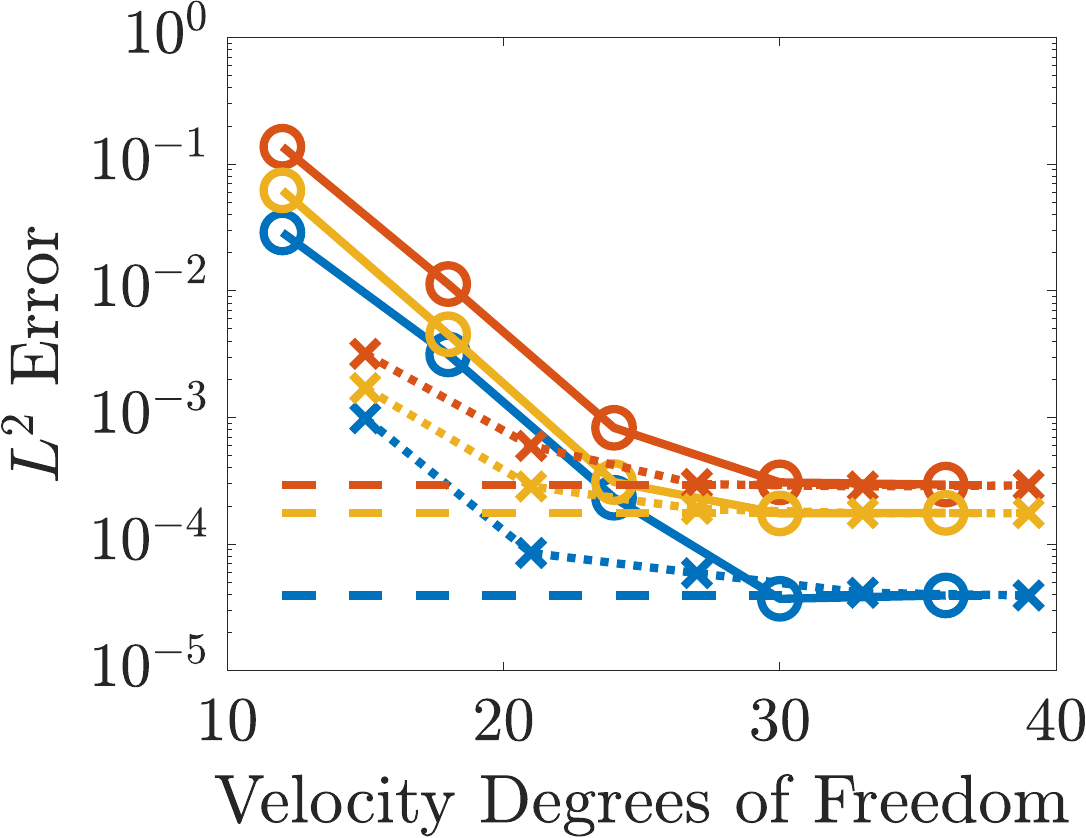}
        \caption{DIRK3 -- $\nu = 10^{3}$}
        \label{fig:riemann:tau_1e-3_DIRK3}
    \end{subfigure}
    \begin{subfigure}{0.32\textwidth}
        \includegraphics[width=\textwidth]{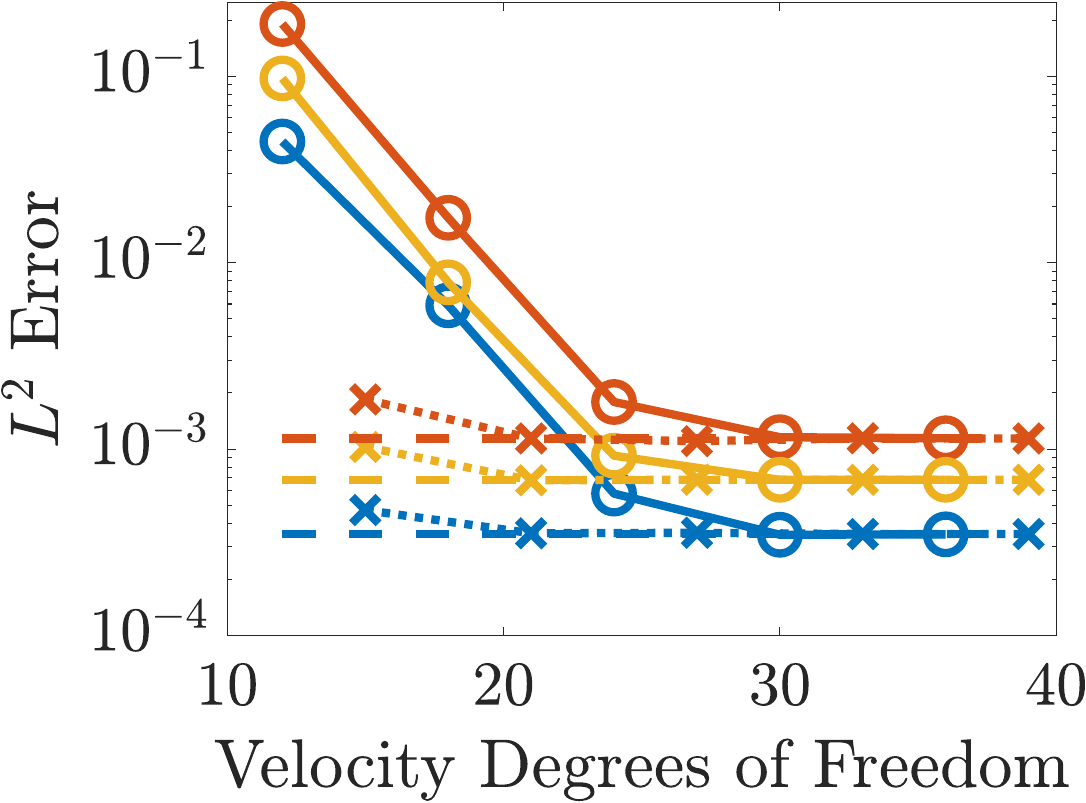}
        \caption{DIRK3 -- $\nu = 10^{4}$}
        \label{fig:riemann:tau_1e-4_DIRK3}
    \end{subfigure}
    \caption{
    Sod shock tube (\Cref{subsec:riemann}): Relative $L^2$ error of the fluid variables of the HOLO and MM-HOLO methods plotted against the velocity degrees of freedom at $t=0.1$.
    We compare the methods with three different collision frequencies and two DIRK methods.
    We set $N_x=256$ and consider $N_v\in\{4,6,8,10,12\}$.
    The reference solution is defined be the average of the MM-HOLO and HOLO solutions computed with $N_x=256$ and $N_v=64$.
    The dashed lines represent the saturation point of the methods which is defined to be half of the relative difference between the two reference solutions.
    The legend in \Cref{fig:riemann:tau_1e-4_DIRK1} is consistent across the other figures.
    When $\nu\gg 1$, the MM-HOLO method is more accurate than HOLO on coarse velocity meshes. 
    }
    \label{fig:riemann:HOLO_vs_MM_compression}
\end{figure}

\paragraph{Low-rank approximation} 
We now see how both reference solutions compare when compressed using low-rank techniques.  
Given a kinetic distribution $f\in V_h$, its coefficient representation $F$ in a basis can be viewed as a $\dof_x\times \dof_v$ matrix where $\dof$ represents the degrees of freedom in each dimension and, in this case, is given by $\dof_x = 3N_x$ and $\dof_v = 3N_v$.
To construct $F$, we employ a nodal DG representation where the nodes are given by a rescaling of the tensored 3-point Gauss-Legendre rule on each local element in $x$ and $v$.
We run the HOLO and MM-HOLO methods for $N_v=64$ to $t=0.1$ with backward Euler time stepping.
For the HOLO method, we use a singular value decomposition (SVD) of the coefficient matrix: $F=XSV^\top$, where $X\in\mathbb{R}^{ \dof_x\times m}$ and $V\in\mathbb{R}^{\dof_v\times m}$ are orthogonal, and $S=\text{diag}(\sigma_1,...,\sigma_{m})\in\mathbb{R}^{m\times m}$ is diagonal and $m=\min\{\dof_x,\dof_v\}$.
Given $r\geq 1$, let $F_r = X_rS_rV_r^\top$ where $S_r = \text{diag}(\sigma_1,\ldots,\sigma_r)\in \mathbb{R}^{r\times r}$, and $X_r$ and $V_r$ are the first $r$ columns of $X$ and $V$ respectively. 
The low-rank matrix $F_r$ corresponds to a function $f_r\in V_h$ that we compare against the reference solution.
We define the compression factor as the ratio of the storage cost of the low-rank $F_r$ versus the storage cost of $F$, i.e.,
\begin{equation}\label{eqn:riemann:compression_factor:HOLO}
    \textrm{Compression Factor (\%)} = 100\frac{r(\dof_x+\dof_v+1)}{\dof_x\dof_v}.
\end{equation}
For the MM-HOLO method, we perform the same low-rank operations as above on the micro distribution $g$ to produce a low-rank approximation $g_r\in V_h$, resulting in an approximation $f_r = M(\bmrho) + g_r$ to $f$. 
Because we have to keep track of the moments $\bmrho\in\Vxhc$ separately, this representation has a compression factor
\begin{equation}\label{eqn:riemann:compression_factor:MM}
    \textrm{Compression Factor (\%)} = 100\frac{3\dof_x + r(\dof_x+\dof_v+1)}{\dof_x\dof_v}.
\end{equation}
If \Cref{fig:riemann:LR:HOLO_vs_MM_compression}, we plot the compression factor, a function of the rank $r$, versus the relative $L^2$ error for the HOLO and MM-HOLO methods and $\nu\in\{10^{2},10^{3},10^{4}\}$.
For $\nu=10^{2}$, there is little difference in the compression of MM-HOLO vs HOLO.  
However, as $\nu$ increases, the compression benefit of MM-HOLO begins.  
For $\nu=10^{3}$, the MM-HOLO method is more accurate for a given compression factor, but both methods saturate at similar compression factors.
At $\nu=10^{4}$, the MM-HOLO method is significantly more accurate for a given compression factor and saturates sooner.

\begin{figure}
    \centering
    \begin{subfigure}{0.32\textwidth}
        \includegraphics[width=\textwidth]{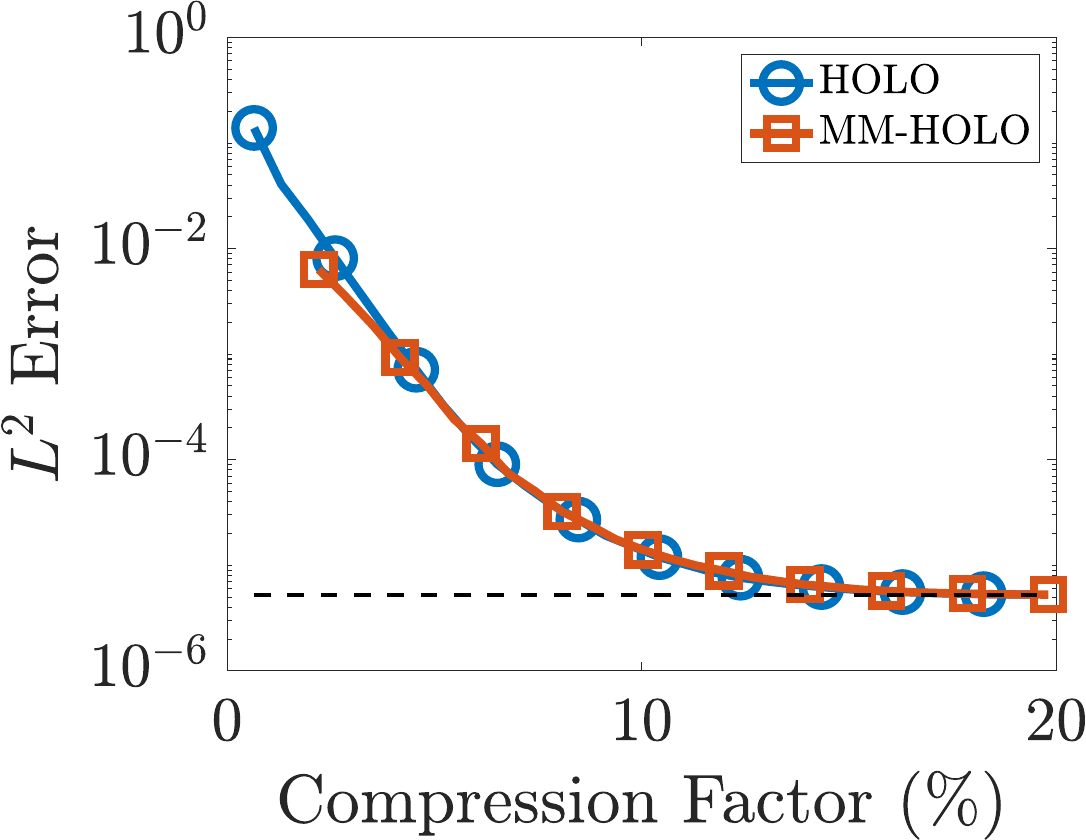}
        \caption{$\nu = 10^{2}$}
        \label{fig:riemann:LR:tau_1e-2_DIRK1}
    \end{subfigure}
    \begin{subfigure}{0.32\textwidth}
        \includegraphics[width=\textwidth]{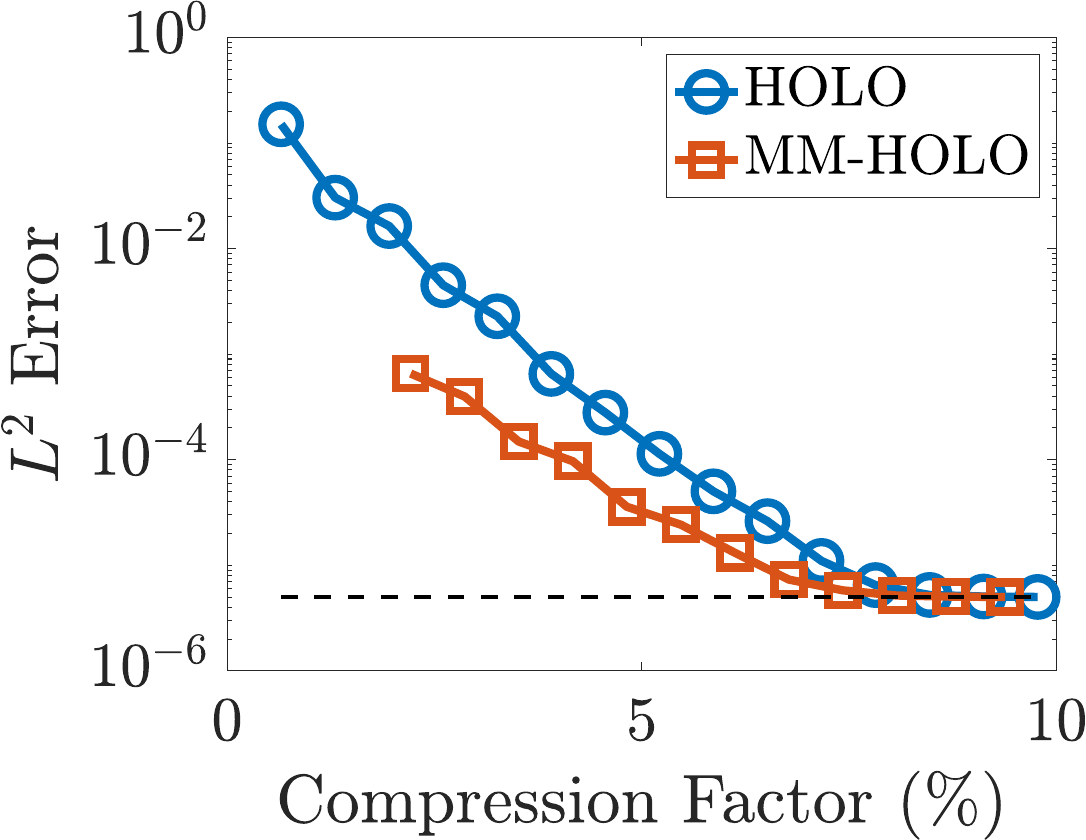}
        \caption{$\nu = 10^{3}$}
        \label{fig:riemann:LR:tau_1e-3_DIRK1}
    \end{subfigure}
    \begin{subfigure}{0.32\textwidth}
        \includegraphics[width=\textwidth]{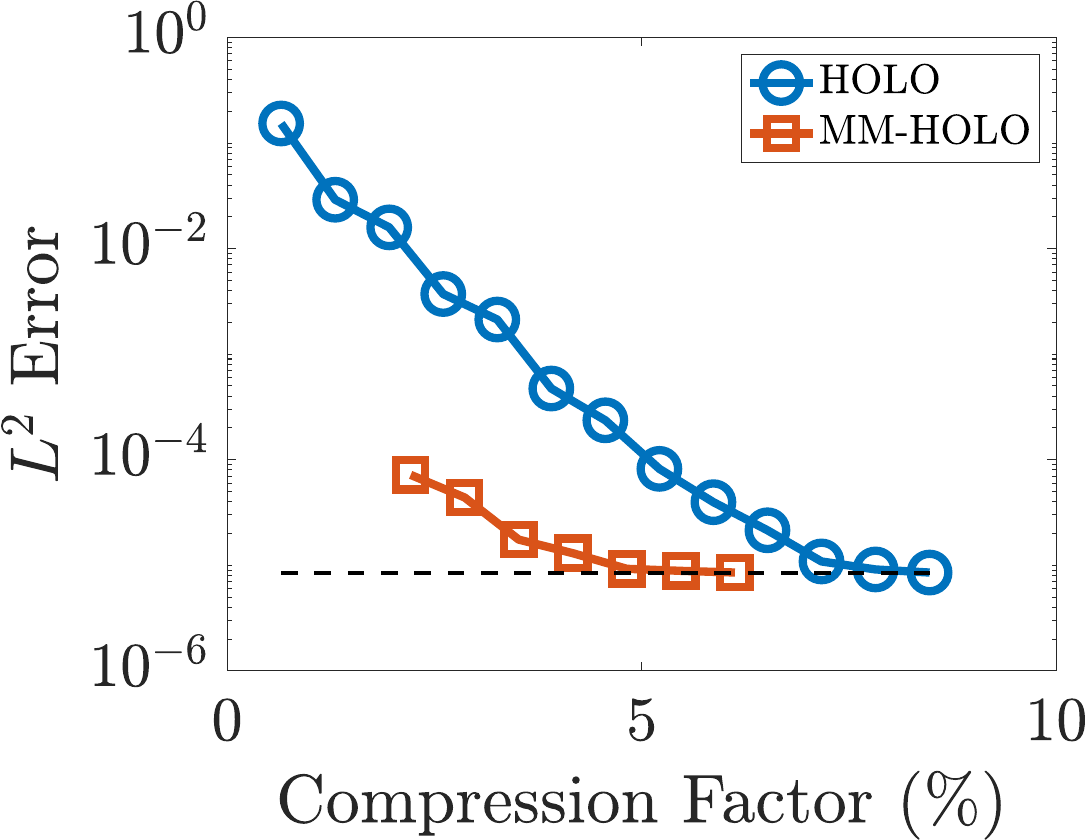}
        \caption{$\nu = 10^{4}$}
        \label{fig:riemann:LR:tau_1e-4_DIRK1}
    \end{subfigure}
    \caption{Sod shock tube (\Cref{subsec:riemann}): Relative $L^2$ error of the low-rank compressed distribution against the reference.  The reference is defined as the average of the MM-HOLO and HOLO solutions computed with $N_x = 256$ and $N_v=64$.  The compression factors are given for the HOLO and MM-HOLO methods in \eqref{eqn:riemann:compression_factor:HOLO} and \eqref{eqn:riemann:compression_factor:MM} respectively.
    Compression of the micro distribution $g$ is more efficient than compression of the kinetic distribution $f$.}
    \label{fig:riemann:LR:HOLO_vs_MM_compression}
\end{figure}

\subsection{Sudden wall heating}\label{subsec:suddenheat}

We next test a sudden wall heating boundary layer problem that is a 1D-1V analog of the example given in \cite{aoki1991numerical}. 
In this problem, the temperature at the left boundary differs from the temperature of the initial condition; this leads to a boundary layer formed near the wall and a shock that travels across the domain.

We let $\W_x = (0,6)$ and $\W_v = (-8,8)$, and set $f(t^{\{0\}}) = M(\bmrho^{\{0\}})$, where $\bmrho^{\{0\}} = [1,0,\tfrac12]^\top$.
We use the far-field boundary condition \eqref{eqn:far-field-bc} at $x=6$.
At the wall $x=0$, we use the sudden wall heating boundary condition $f^- = \sigma_f M(\bmrho_{-})$, where $\bmrho_{-} = [1,0,1]^\top$, and   
\begin{equation}\label{eqn:sigma_definition}\textstyle
    \sigma_f = -\left(\frac{2\pi}{2}\right)^{1/2}\langle vf(0,v,t) \rangle_{\{v < 0\}}
\end{equation}
is a reflection parameter that enforces mass conservation.
We set $\nu=128$.
From \cite[Equation 10]{aoki1991numerical}, this sets the mean-free-path and mean-free-time respectively as
\begin{equation}\label{eqn:mean-free-path-time}\textstyle
\ell_0=\frac{\sqrt{8}}{\sqrt{\pi}\nu} \approx 1.25\times 10^{-2} \qquad\text{and}\qquad t_0 = \frac{2}{\sqrt{\pi}\nu} \approx 8.82\times 10^{-3}.   
\end{equation}
We use a non-uniform mesh on $\W_x$ that is comprised of two uniform meshes with $N_{x,1}$ cells from $(0,0.25)$ and $N_{x,2}$ cells from $(0.25,6)$.
We justify this meshing strategy in \Cref{subsec:suddenheat:implicit_need}.
Note that $0.25\approx 20\ell_0$.
For all tests, we use the DIRK3 time-stepping scheme, set the tolerance for each method at $10^{-6}$, and set the JFNK solver to terminate at $10^{-9}$ unless otherwise specified.

\subsubsection{Need for implicit methods}\label{subsec:suddenheat:implicit_need}

We first demonstrate the need for fully implicit methods for this problem.
It has been shown (see \cite{aoki1991numerical}) that sufficient resolution in $x$ near the wall is needed to properly resolve the boundary layer.
To demonstrate this fact, we solve the problem using the SI method \eqref{eqn:BE_SI} with $N_v = 32$, $N_{x,2} = 58$, and $\dt = 0.025$, and consider three resolutions at the wall: $N_{x,1} \in\{3,25,250\}$, i.e.,  
$h_x \approx \{7\ell_0,0.8\ell_0,0.08\ell_0\}$ respectively.
In \cite{haack2012high} the authors choose 6-8 cells per mean-free path while the authors in \cite{aoki1991numerical} set $h_x\in (0.0025\ell_0,0.1\ell_0)$ depending on the distance from the wall.
We note that \cite{aoki1991numerical,haack2012high} consider a problem posed in three velocity dimensions instead of one; therefore, reference to their results should only be taken qualitatively.  

In \Cref{fig:suddenheat:boundary_layer_compression:dist:t=0.025,fig:suddenheat:boundary_layer_compression:dist:t=1}, we plot a slice of the distribution for each prescribed spatial resolution along $x\approx 0.01\ell_0$.
For $t=\dt$, \Cref{fig:suddenheat:boundary_layer_compression:dist:t=0.025} shows $h_x\approx 7\ell_0$ is not sufficient to capture the discontinuity in velocity between the inflow and outflow boundary of the distribution.
The discontinuity is observable when $h_x\approx0.8\ell_0$ and is fully resolved when $h_x$ is further refined.
As $t$ increases, the discontinuity decreases, see  \Cref{fig:suddenheat:boundary_layer_compression:dist:t=1}, which is consistent with the results in \cite{aoki1991numerical,haack2012high}.
In this case the resolution near the boundary is less important.
In \Cref{fig:suddenheat:boundary_layer_compression:fluid:t=0.025,fig:suddenheat:boundary_layer_compression:fluid:t=1} we plot the bulk velocity $u$, which confirms that the discontinuity is not well captured by the coarse $h_x$ resolution at $t=\dt$ while the results between all three resolutions are similar for longer times. 

\begin{figure}[ht]
    \captionsetup[subfigure]{aboveskip=5pt,belowskip=5pt}
    \centering
    \begin{subfigure}{0.40\textwidth}
        \includegraphics[width=\textwidth, height=0.7\textwidth]{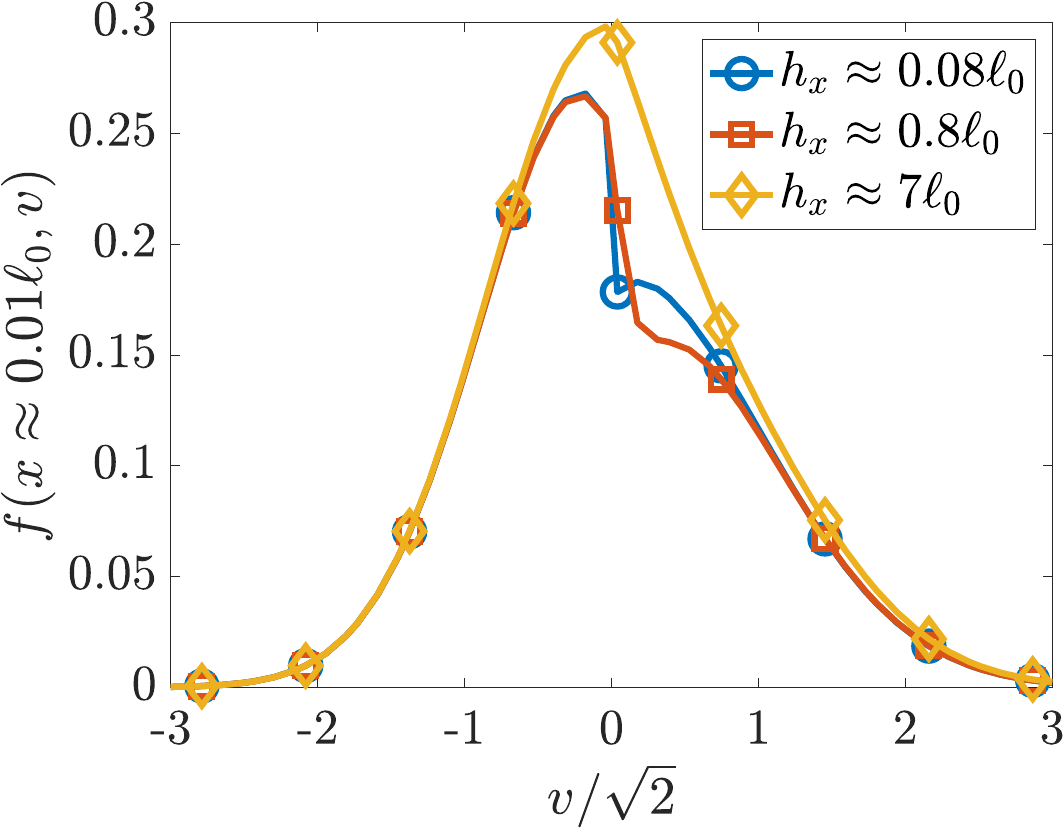}
        \caption{Distribution; $t=0.025\approx 2.83t_0$}
        \label{fig:suddenheat:boundary_layer_compression:dist:t=0.025}
    \end{subfigure}
    \hspace{0.1\textwidth}
    \begin{subfigure}{0.4\textwidth}
        \includegraphics[width=\textwidth, height=0.7\textwidth]{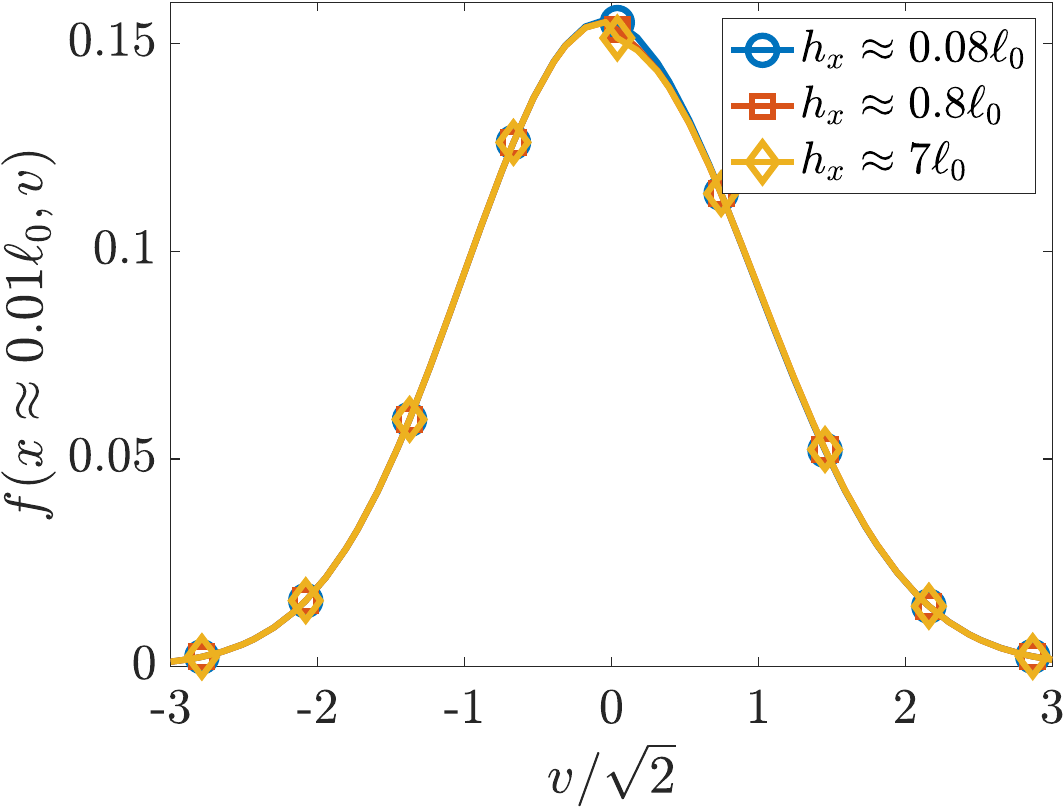}
        \caption{Distribution; $t=1\approx 113t_0$}
        \label{fig:suddenheat:boundary_layer_compression:dist:t=1}
    \end{subfigure}

    \begin{subfigure}{0.40\textwidth}
         \includegraphics[width=\textwidth, height=0.7\textwidth]{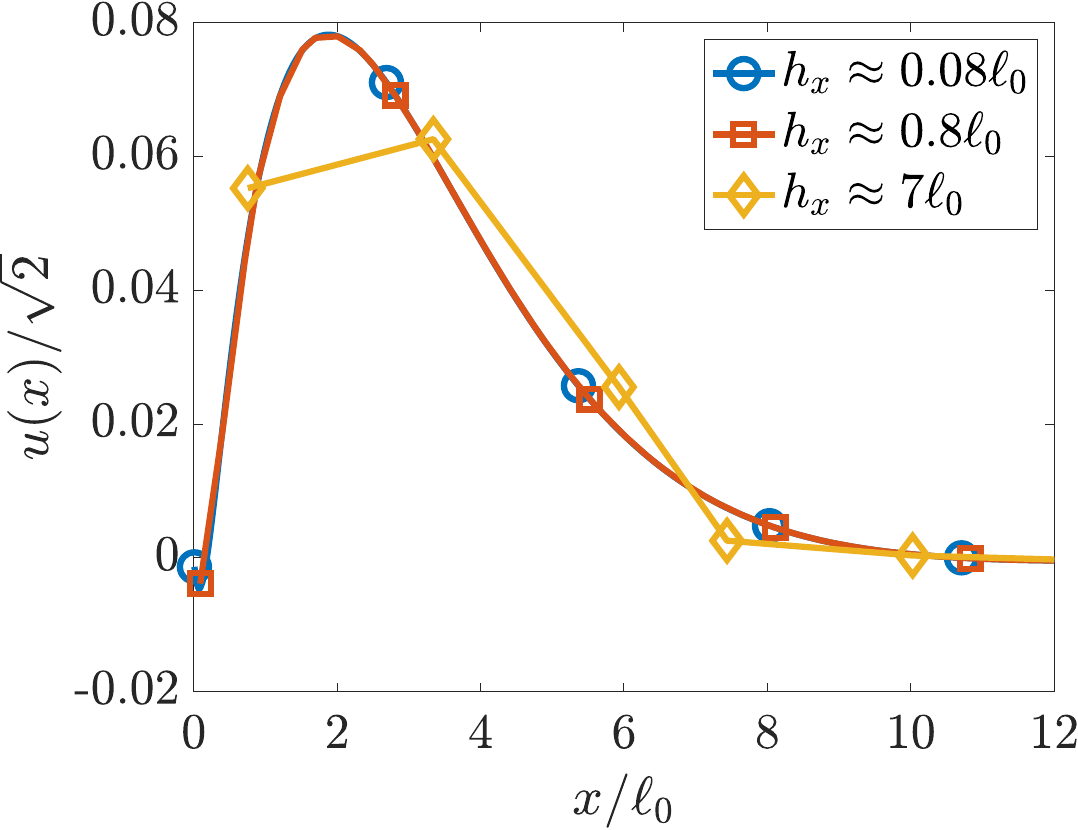}
        \caption{Bulk Velocity;  $t=0.025\approx 2.83t_0$}
        \label{fig:suddenheat:boundary_layer_compression:fluid:t=0.025}
    \end{subfigure}
    \hspace{0.1\textwidth}
    \begin{subfigure}{0.40\textwidth}
         \includegraphics[width=\textwidth, height=0.7\textwidth]{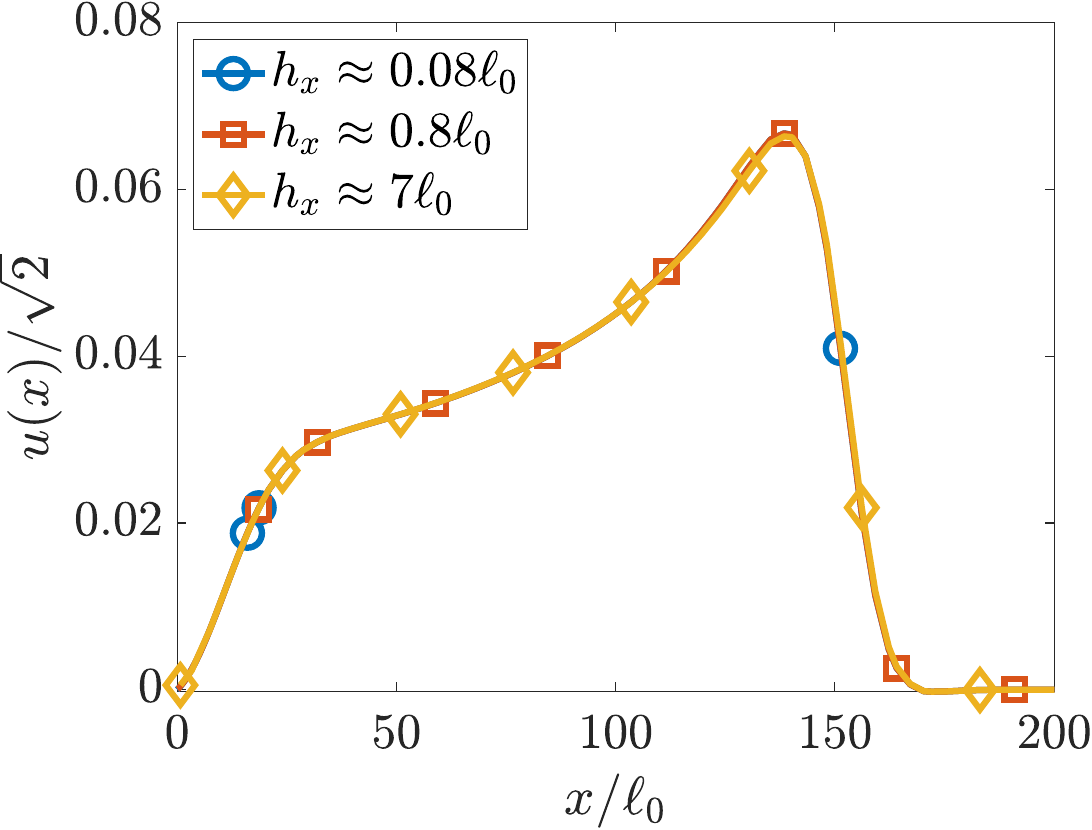}
        \caption{Bulk Velocity; $t=1\approx 113t_0$}
        \label{fig:suddenheat:boundary_layer_compression:fluid:t=1}
    \end{subfigure}
    \caption{
    Sudden wall heating (\Cref{subsec:suddenheat}): Plots to compare effects of the boundary layer and moments for varying resolution at the wall.
    Left: Slice of the distribution in velocity at $x\approx 0.01\ell_0$.
    Right: Bulk Velocity.
    }
    \label{fig:suddenheat:boundary_layer_compression}
\end{figure}

The results of \Cref{fig:suddenheat:boundary_layer_compression} suggest that for a short time a fine resolution must be taken at the left wall in order to capture the transition layer from kinetic to fluid regimes.  
For explicit and implicit-explicit (IMEX) integrators, the fine spatial resolution leads to a restrictive timestep (see \eqref{eqn:dt_expl}) that might not be needed for accuracy.  
To illustrate this fact, we apply the SI method with DIRK3 time-stepping, $N_{x,1} = 250$, and three different timesteps $\dt\in\{2.5\times 10^{-2},5.0\times 10^{-3},1.0\times 10^{-3}\}$.
We compare these three runs to a finite volume code \cite{habbershaw2022progress} that is second-order in space with a uniform discretization of $h_x = 10^{-3}$ from $(0,3)$ and second-order in velocity with $N_v=96$.
This finite volume method uses a third-order IMEX method where the collision operator and transport operator are respectively treated implicitly in four stages and explicitly in three stages. 
A timestep of $\dt = 5.625\times 10^{-5}$ is selected which is $0.9$ times the maximum explicit timestep for the second-order finite volume method with $h_x=10^{-3}$.

In \Cref{fig:suddenheat:FV_comp} we plot the velocity and temperature profiles for these four runs at $t=0.025$.
As we decrease $\dt$ for the SI runs, we approach the IMEX solution.
At $\dt = 10^{-3}$, the termination criterion \eqref{eqn:stopping_criterion_SI} is reached at 4 iterations per stage.
Therefore, while SI requires in every timestep four transport solves as opposed to one in IMEX\footnote{We assume the transport solve of SI and the transport evaluation of IMEX are comparable operations.}, SI gives a comparable solution with a 16 times larger timestep.

\begin{figure}
    \centering
    \begin{subfigure}{0.45\textwidth}
       \includegraphics[width=\textwidth]{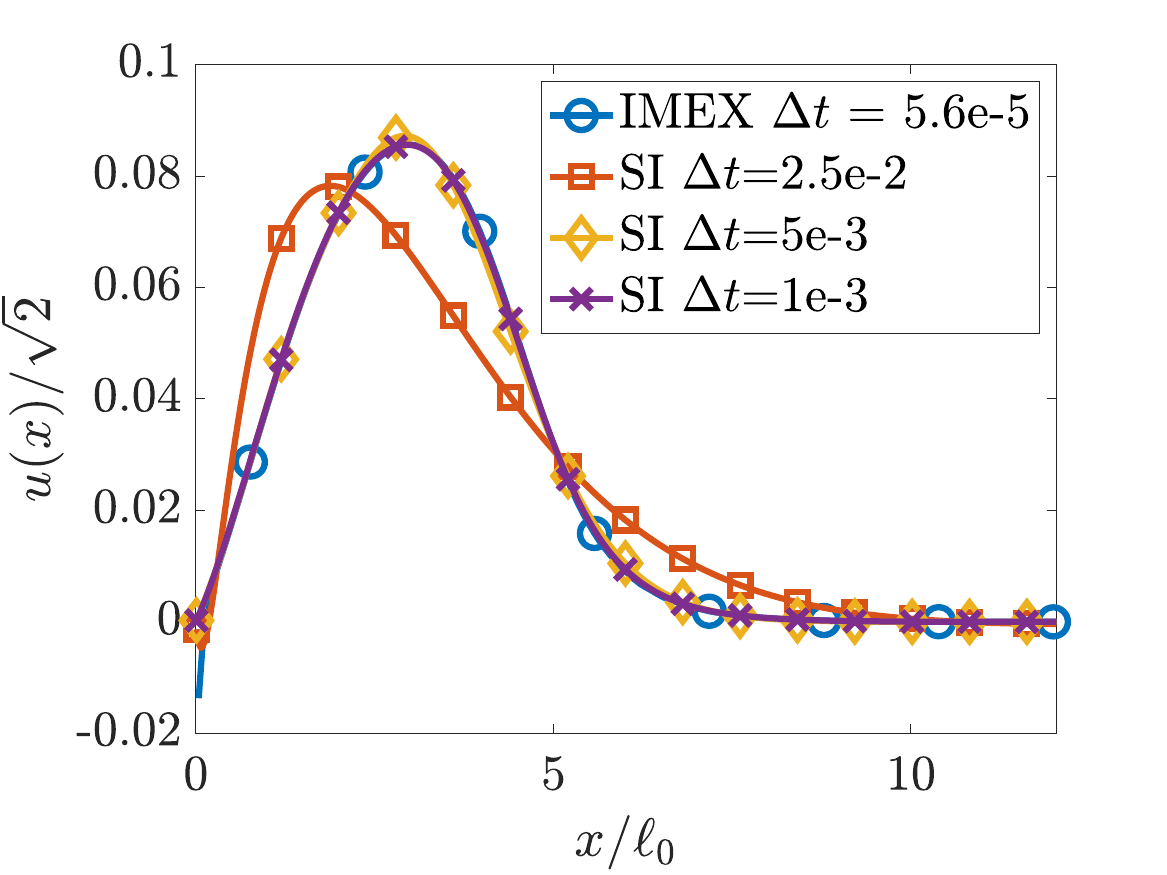}
       \caption{Velocity}
       \label{fig:suddenheat:FV_comp:velocity}
    \end{subfigure}
    \begin{subfigure}{0.45\textwidth}
       \includegraphics[width=\textwidth]{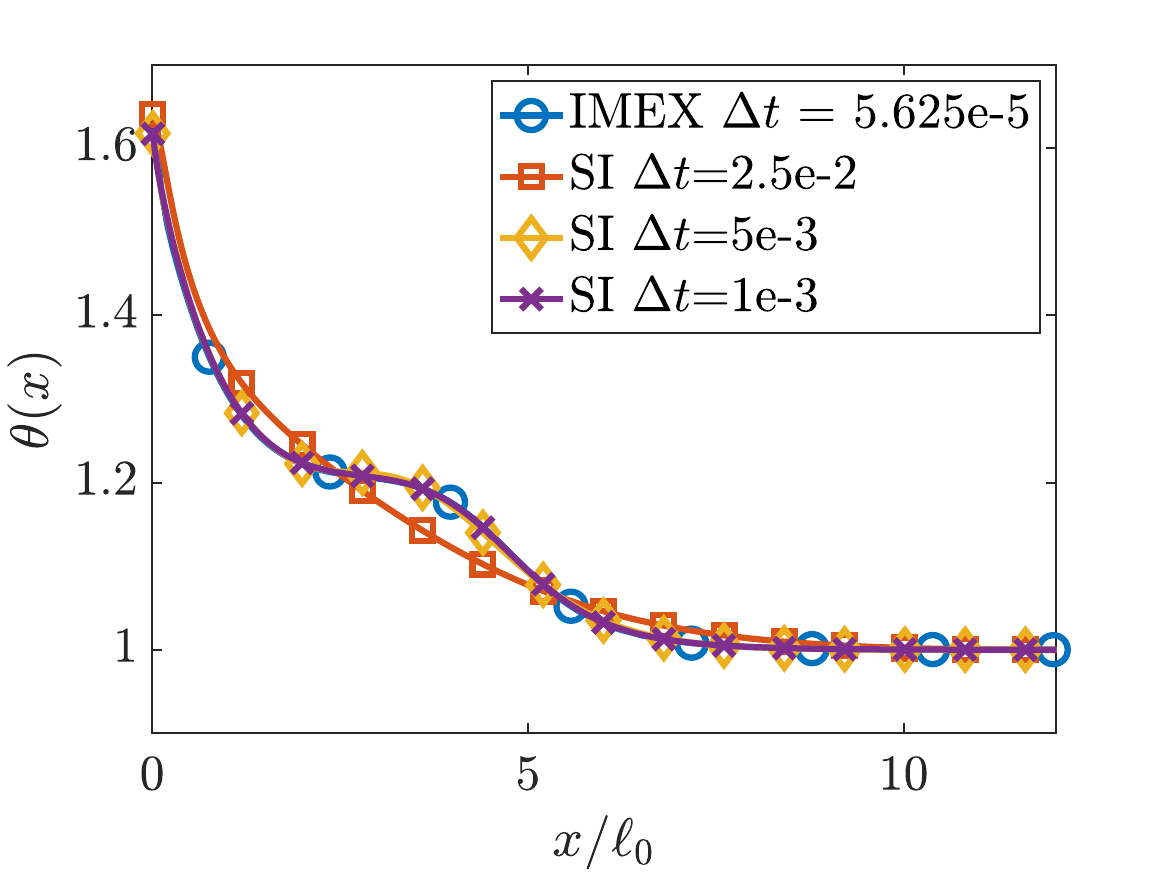}
       \caption{Temperature}
       \label{fig:suddenheat:FV_comp:temp}
    \end{subfigure}
    \caption{
    Sudden wall heating (\Cref{subsec:suddenheat}): Plots of the bulk velocity and temperature at $t=0.025$ of source iteration for various timesteps versus an IMEX finite volume code.
    All runs have a resolution of $h_x=10^{-3}$ for $0<x<0.25\approx 20 \ell_0$.
    Fully implicit methods allow us to choose a timestep based on accuracy rather than stability.
    }
    \label{fig:suddenheat:FV_comp}
\end{figure}

\subsubsection{Comparison of iteration counts}
We apply the SI, HOLO, and MM methods using $N_{x,1}\in\{3,25,250\}$ at the boundary layer.   
The iteration counts at the first timestep for $\dt\in\{2.5\times10^{-2},5.0\times 10^{-3},1.0\times 10^{-3}\}$ are given in \Cref{tab:bndy:SI_vs_HOLO_vs_HOLO-AV_iters}, 
The ratio between these timesteps and the explicit timestep restriction \eqref{eqn:dt_expl} ranges from $0.48$ to $1000$.
The convergence of SI is consistent with the analysis of the linear problem, see \Cref{prop:lin_bgk:SI_error}, in that the convergence rate is independent of $N_{x,1}$ and improves over vanishing $\dt$.
Overall, HOLO and MM-HOLO require fewer iterations to converge; the only exception is when $\dt=2.5\times 10^{-2}$ and $N_{x,1}=250$, where $\dt=1000\dt_\text{expl}$ and both HOLO and MM-HOLO fail to converge.
We attribute this failure to the stiffness from the lagged Euler flux in both methods, which, as shown in \Cref{prop:lin_bgk:HOLO_error}, persists in the linear case.
However; this issue only arises when $\dt/\dt_{\text{expl}}$ is very large.

While only the iterations for the first timestep are given in \Cref{tab:bndy:SI_vs_HOLO_vs_HOLO-AV_iters}, the HOLO and MM-HOLO methods improve once the boundary layer vanishes and the shock moves into the interior.
For example at $t=1.5$, $N_{x,1}=25$, and $\dt=2.5\times 10^{-2}$, the HOLO, MM-HOLO, and SI methods require 9, 11, and 40 iterations respectively.
Therefore, \Cref{tab:bndy:SI_vs_HOLO_vs_HOLO-AV_iters} is a conservative estimation of the benefits of HOLO and MM-HOLO.

\begin{table}[]
    \centering
    {\setlength\tabcolsep{5 pt}
    \begin{tabular}{c c|| c c c | c c c | c c c}
          & & \multicolumn{3}{c|}{SI} & \multicolumn{3}{c|}{HOLO} & \multicolumn{3}{c}{MM-HOLO} \\
          & $N_{x,1}$ & 3 & 25 & 250 & 3 & 25 & 250 & 3 & 25 & 250  \\ \hline
          \multirow{3}{*}{ \parbox{1.5cm}{\centering $\dt=2.5\times 10^{-2}$ }\vspace{1.1em}}
          & $\dt/\dt_{\text{expl}}$ & 12 & 100 & 1000 & 12 & 100 & 1000 & 12 & 100 & 1000 \\
          & Iterations & 46 & 46 & 46 & 16 & 35 & DNC & 16 & 35 & DNC \\ \hline
          \multirow{3}{*}{ \parbox{1.5cm}{\centering $\dt=5.0\times 10^{-3}$ }\vspace{1.1em}} 
          & $\dt/\dt_{\text{expl}}$ & 2.4 & 25 & 250 & 2.4 & 25 & 250 & 2.4 & 25 & 250 \\
          & Iterations & 18 & 18 & 18 & 12 & 13 & 17 & 12 & 13 & 17 \\ \hline
          \multirow{3}{*}{ \parbox{1.5cm}{\centering $\dt=1.0\times 10^{-3}$ }\vspace{1.1em}} 
          & $\dt/\dt_{\text{expl}}$ & 0.48 & 5 & 50 & 0.48 & 5 & 50 & 0.48 & 5 & 50 \\
          & Iterations & 12 & 12 & 12 & 8 & 11 & 11 & 8 & 12 & 12 \\ \hline
    \end{tabular}
    }
    \caption{
    Sudden wall heating (\Cref{subsec:suddenheat}): Total number of iterations of the first DIRK3 timestep with SI \eqref{eqn:BE_SI}, HOLO \eqref{eqn:BE_HOLO}, and MM-HOLO \eqref{eqn:BE_MM}.
    Here $N_{x,1}$ refers to a uniform discretization of the interval $(0,0.25)$, and $\dt_{\text{expl}}$ is defined in \eqref{eqn:dt_expl}.
    DNC denotes ``did not converge''.
    The HOLO and MM-HOLO methods perform better than SI until the timestep is several orders of magnitude over the explicit timestep restriction.
    }
    \label{tab:bndy:SI_vs_HOLO_vs_HOLO-AV_iters} 
\end{table}

\subsubsection{Compression benefits of the MM-HOLO method}  

Set $\dt=2.5\times 10^{-2}$.
We perform the same coarse-velocity compression analysis as in \Cref{subsec:riemann} with $N_{x,1}=25$.  
The reference solution is determined to be an average of the MM-HOLO and HOLO methods with $N_v=120$.
These two methods are then compared with $N_v=\{4,6,8,10,12,20\}$. 
The relative errors in the fluid variables are shown in \Cref{fig:suddenheat:HOLO_vs_MM_compression} for $t=0.2$ and $t=2$. 
Unlike the Sod shock tube, the MM-HOLO method does not show favorable improvement when compared to HOLO;
in fact, for $N_v=10$ and $12$, the HOLO method is slightly more accurate.
We attribute this behavior to the boundary layer, which remains out of equilibrium even if $\nu \gg 1$.
To see this, we plot $g$ for the MM-HOLO method with $N_v=120$ in \Cref{fig:suddenheat:HOLO_vs_MM_compression:g}, which shows that the boundary layer is the primary contribution of $g$.
Therefore, the micro distribution $g$ becomes the primary component to capture in order to reduce to error further, and we conjecture that resolving $g$ is as hard of a problem as capturing the kinetic distribution $f$ and possibly harder since the $g$ is more oscillatory in velocity.

\begin{figure}
    \centering
    \begin{subfigure}{0.40\textwidth}
        \includegraphics[width=\textwidth]{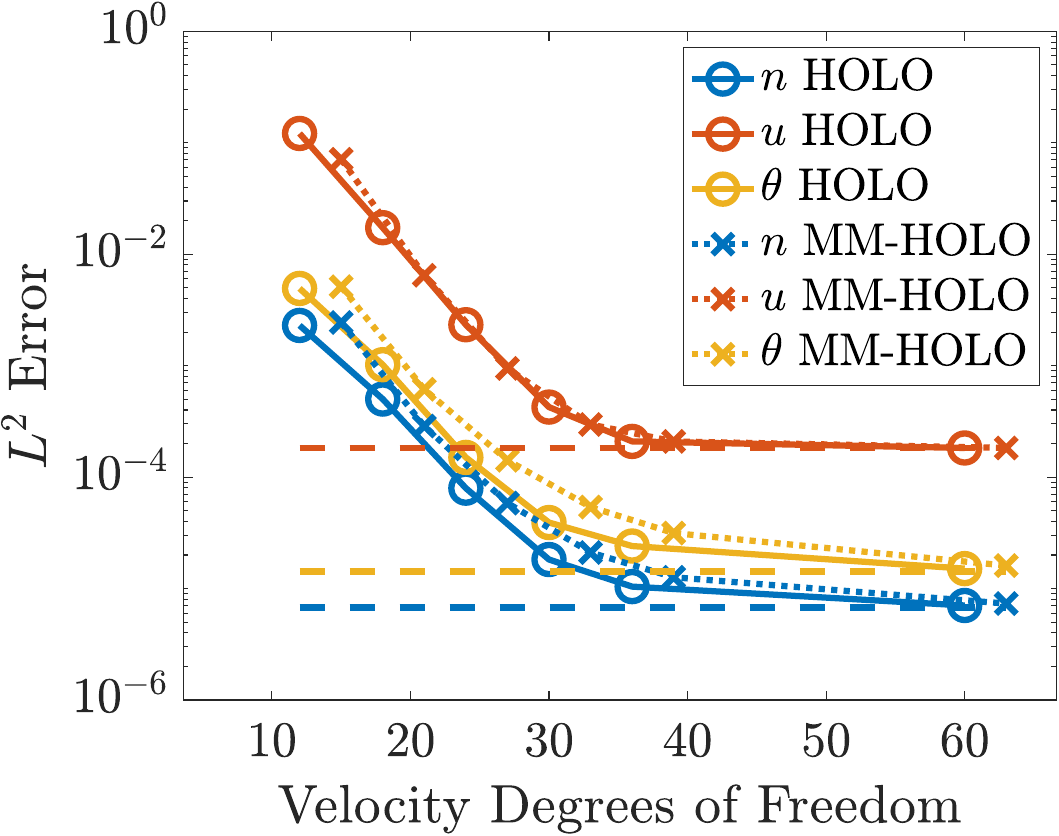}
        \caption{$t=0.2\approx 22t_0$}
        \label{fig:suddenheat:HOLO_vs_MM_compression:t_0p2}
    \end{subfigure}
    \hspace{20pt}
    \begin{subfigure}{0.40\textwidth}
        \includegraphics[width=\textwidth]{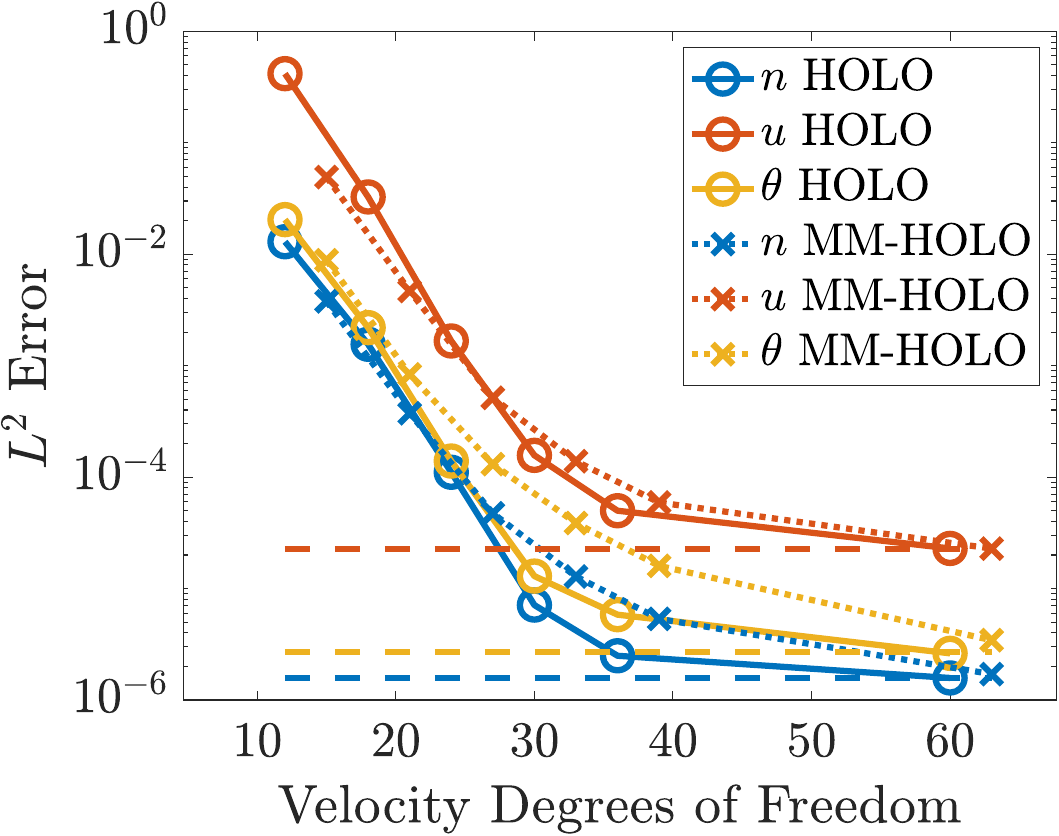}
        \caption{$t=2\approx 220t_0$}
        \label{fig:suddenheat:HOLO_vs_MM_compression:t_2}
    \end{subfigure}

    \caption{
    Sudden wall heating (\Cref{subsec:suddenheat}): Relative $L^2$ error of the fluid variables of the HOLO and MM-HOLO methods plotted against the velocity degrees of freedom.
    We set $N_{x,1}=25$  and $N_v\in\{4,6,8,10,12,20\}$.
    The reference solution is defined to be the average of the MM-HOLO and HOLO solutions with $N_v=120$.
    The dashed lines represent the saturation point, which is defined to be half the relative error between the MM-HOLO and HOLO solutions used to create the reference.
    }
    \label{fig:suddenheat:HOLO_vs_MM_compression}
\end{figure}
\begin{figure}
    \centering
    \begin{subfigure}{0.40\textwidth}        \includegraphics[width=\textwidth]{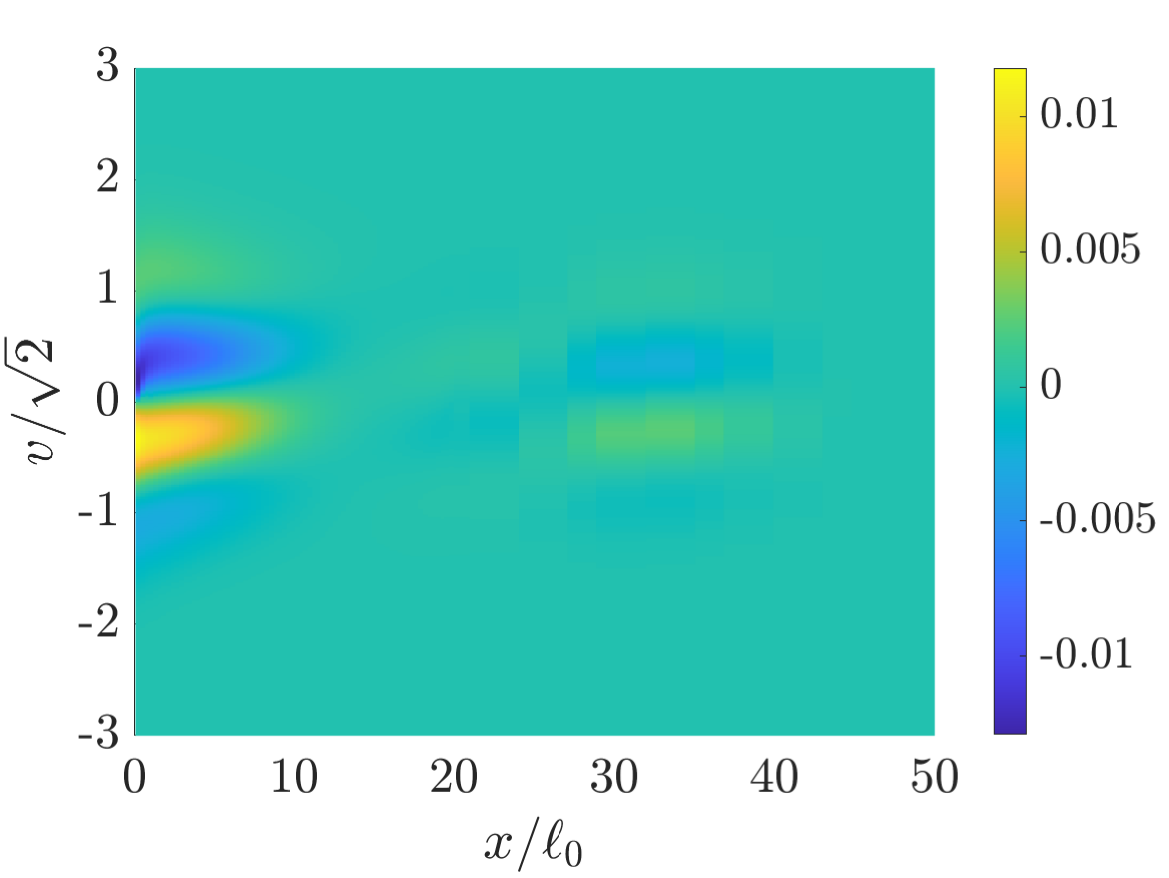}
        \caption{$t=0.2\approx 22t_0$}
        \label{fig:suddenheat:HOLO_vs_MM_compression:g:t_0p2}
    \end{subfigure}
    \hspace{25pt}
    \begin{subfigure}{0.40\textwidth}
        \includegraphics[width=\textwidth]{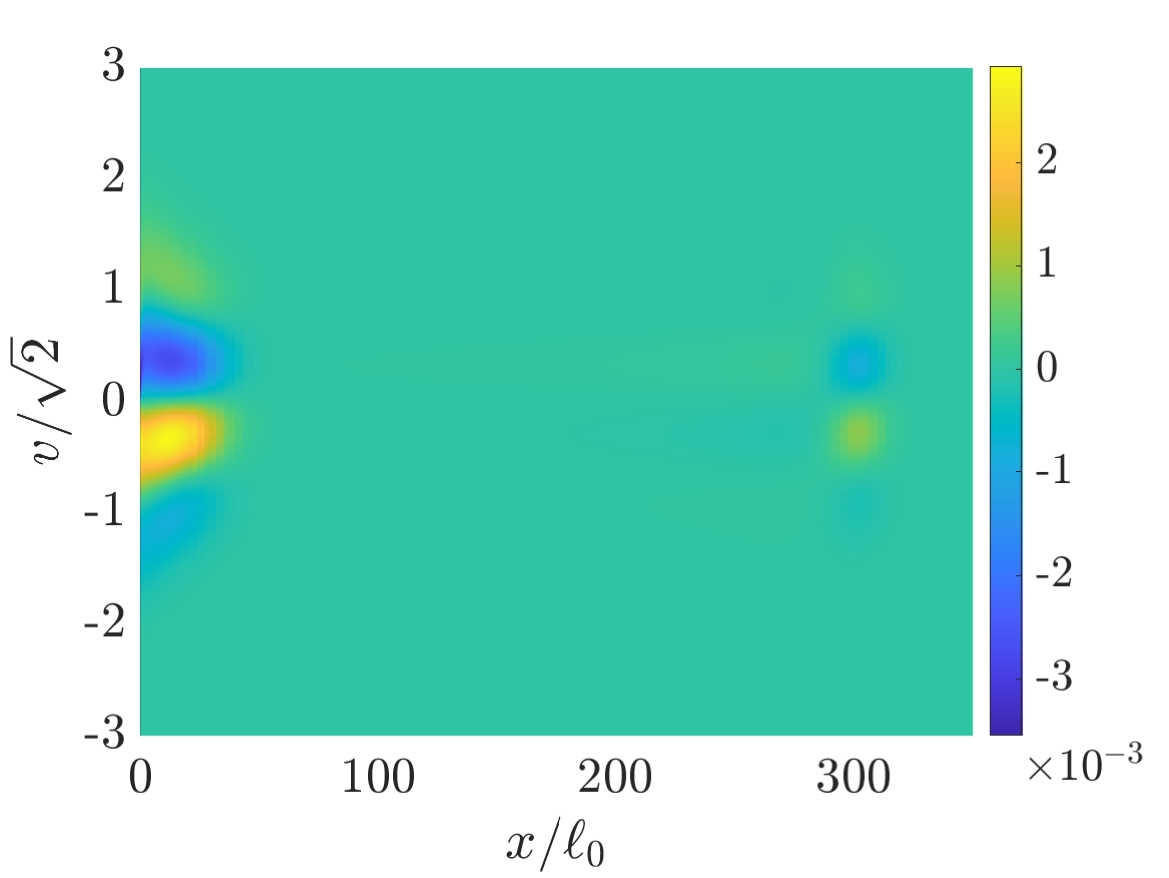}
        \caption{$t=2\approx 220t_0$}
        \label{fig:suddenheat:HOLO_vs_MM_compression:g:t_2}
    \end{subfigure}

    \caption{
    Sudden wall heating (\Cref{subsec:suddenheat}): Plots of the micro distribution $g$ for the MM-HOLO method with $N_{x,1}=25$ and $N_v=120$.
    The largest contributions $g$ arise from the boundary layer and the propagation of the shock into the interior of the domain.  
    }
    \label{fig:suddenheat:HOLO_vs_MM_compression:g}
\end{figure}

Additionally, we perform a low-rank compression test similar to the Sod shock tube on the HOLO and MM-HOLO methods.
We let $N_v=120$ which sets $\dof_x=3(N_{x,1}+N_{x,2}) = 249$ and $\dof_v = 3N_v = 360$.  
In \Cref{fig:suddenheat:HOLO_vs_MM_compression:LR}, we plot the compression factor of the low-rank matrix versus the error of the approximation for $t=0.2$ and $t=2$.  
The plots show that for both times, the MM-HOLO approximation is only marginally more efficient than HOLO, and both methods saturate at the same compression factor.
Thus in this case, the MM-HOLO method does not offer superior compression saving versus the more traditional approach. 
We conjecture the similarity in compression between these two methods is again caused by the boundary layer which drives the dominant portion of non-equilibrium behavior.

\begin{figure}
    \centering
    \begin{subfigure}{0.40\textwidth}
        \includegraphics[width=\textwidth]{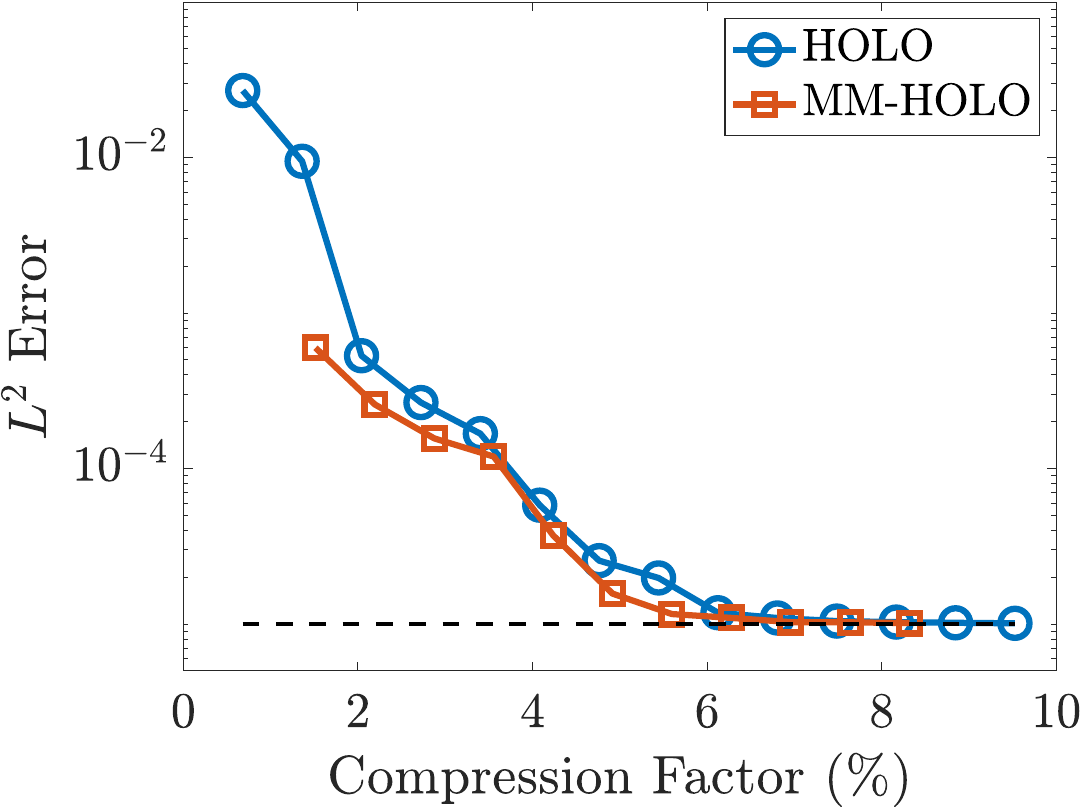}
        \caption{$t=0.2\approx 22t_0$}
        \label{fig:suddenheat:HOLO_vs_MM_compression:LR:t_0p2}
    \end{subfigure}
    \hspace{20pt}
    \begin{subfigure}{0.40\textwidth}
        \includegraphics[width=\textwidth]{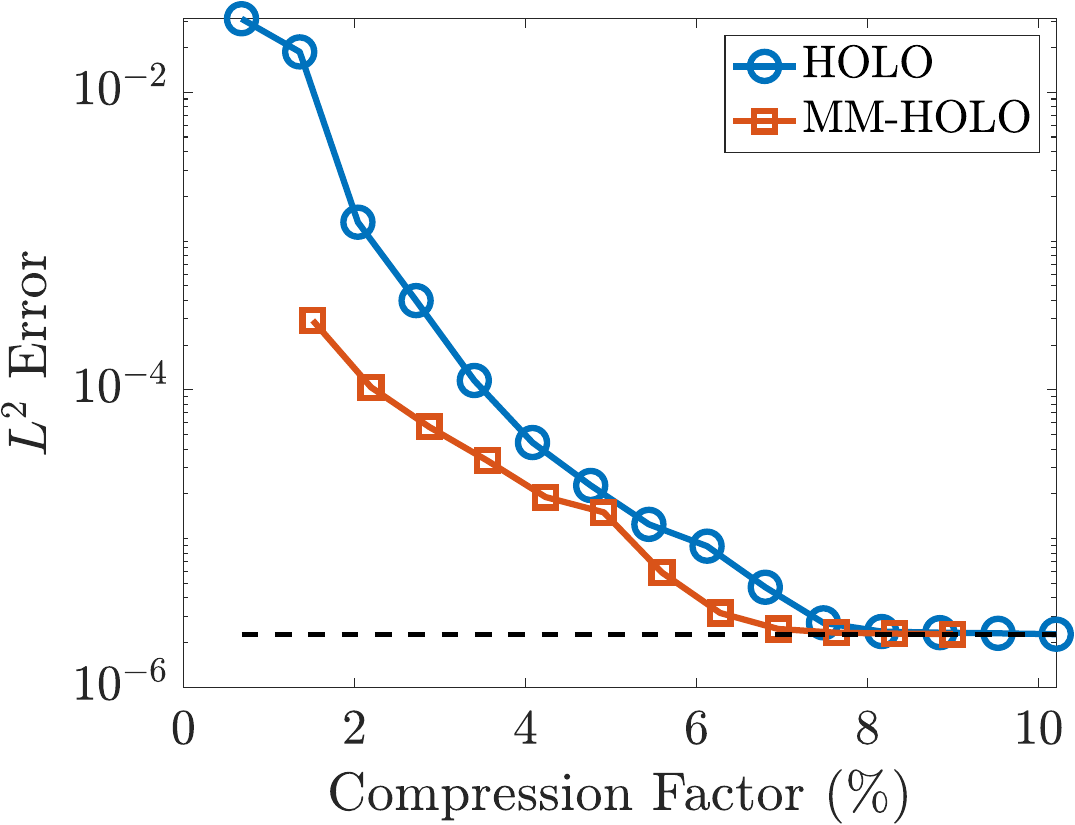}
        \caption{$t=2\approx 220t_0$}
        \label{fig:suddenheat:HOLO_vs_MM_compression:LR:t_2}
    \end{subfigure}

    \caption{
    Sudden wall heating (\Cref{subsec:suddenheat}): Relative $L^2$ error of the low-rank compressed distribution against the reference.
    The reference is defined as the average of the MM-HOLO and HOLO methods with $N_x = 25$ at the boundary layer and $N_v=120$; this produces a coefficient matrix of size $249\times 360$.
    The compression factors are given for the HOLO and MM-HOLO methods in \eqref{eqn:riemann:compression_factor:HOLO} and \eqref{eqn:riemann:compression_factor:MM} respectively.
    Compression of the micro distribution $g$ is slightly better than compression of the kinetic distribution $f$, but both the MM-HOLO and HOLO methods saturate at similar compression factors.
    }
    \label{fig:suddenheat:HOLO_vs_MM_compression:LR}
\end{figure}

\section{Conclusions}\label{sect:conclusions} 
In this work, we have developed a micro-macro decomposition for implicit temporal discretizations of the BGK model.
We have showed through analysis and implementation that the MM-HOLO method retains the acceleration properties of HOLO while allowing compression of the solution when near equilibrium.
Additionally, we have provided theory and examples that show the convergence rates of SI and HOLO and consistency between them.

Resolving the lack of convergence from the HOLO and MM-HOLO methods for large $\dt/h_x$ is an important future topic and could be achieved by utilizing more accurate fluid solvers with dissipation, e.g., Navier-Stokes approximations following the GSIS approach \cite{zeng2023general,su2020can,su2020fast}, or combining the SI and HOLO methods in certain areas of the domain.
Additional research directions include: (1) determining the computational benefits of HOLO and MM-HOLO on higher-dimensional problems with unstructured grids in position space;
(2) analysis of the method on more accurate collision operations that produce the correct Prandtl number, such as the ES-BGK \cite{holway1965kinetic} and Shakhov \cite{shakhov1968generalization} models;
and (3) analysis of the method on multispecies BGK models \cite{haack2017conservative,habbershaw2025asymptotic}.

\section{Acknowledgements}

We would like to thank Evan Habbershaw for using his finite volume code \cite{habbershaw2022progress} in comparison of our method.

\bibliographystyle{abbrv}
\bibliography{references}


\appendix

\section{Analysis of a linearized BGK model that preserves mass, momentum, and energy} \label{sec:appendix_bgk}

In this section we perform the same analysis of the SI and HOLO methods in \Cref{sec:analysis} but using a linearized BGK operator that preserves all three conservation invariants.  
We delay the following analysis to the appendix since the analysis of the linear model in \Cref{sec:analysis} is simpler and easier to follow, but the following model is more physically relevant and therefore justified.
\subsection{The linearized BGK model}
For fixed $u_0\in\mathbb{R}$ and $\theta_0>0$, define $\mpM(v)=\frac{1}{\sqrt{2\pi\theta_0}}\exp(-\frac{(v-u_0)^2}{2\theta_0})$.  From \cite[(2.5)]{xiong2015high}, the linearization of the BGK model around $\mpM$ is 
\begin{equation}\label{eqn:linearized-bgk:linearized-model}
    \partial_t f + v\partial_x f = \nu(\mN(\bmrho_f)-f),
\end{equation}
where
\begin{align}
    \mN(\bmrho) = \mpM\Big(&\big(\tfrac{u_0^2(v-u_0)^2}{2\theta_0^2}-\tfrac{v^2}{2\theta_0}+\tfrac{3}{2}\big)\rho_0 +\big(\tfrac{-u_0(v-u_0)^2}{\theta_0^2}+\tfrac{v}{\theta_0}\big)\rho_1 + \big(\tfrac{(v-u_0)^2}{\theta_0^2}-\tfrac{1}{\theta_0}\big)\rho_2\Big).
\end{align}
The linearized BGK operator has the same conservation invariants as the nonlinear BGK operator; namely, $\vbr{\bme(\mN(\bmrho_f)-f)}=0$.
Moreover, as $\nu\to\infty$, formally $f\to \mN(\bmrho_f)$ where $\bmrho_f$ satisfies
\begin{equation}\label{eqn:linearized-bgk:drift}
    \partial_t\bmrho_f + \partial_x B\bmrho_f = 0,
\end{equation}
with
\begin{align}
    B\bmrho := \vbr{\bme v\mN(\bmrho)} = 
    \begin{bmatrix}
        0 & 1 & 0 \\
        0 & 0 & 2 \\
        \tfrac{1}{2}(u_0^2-3\theta_0)u_0 & -\tfrac{3}{2}(u_0^2-\theta_0) & 3u_0 
    \end{bmatrix}
    \bmrho.
\end{align}
The matrix $B$ is diagonalizable with real eigenvalues
\begin{equation}
    \lambda_0 = u_0, 
    \quad
    \lambda_1 = u_0+\sqrt{3\theta_0},
    \quad \text{and} \quad \lambda_2=u_0-\sqrt{3\theta_0}.
\end{equation}
Hence \eqref{eqn:linearized-bgk:drift} is a hyperbolic system. 
Let $B=P\Lambda P^{-1}$ where $\Lambda = \text{diag}(\lambda_0,\lambda_1,\lambda_2)$ and
\begin{align}
\label{eq:P}
    P = 
    \begin{bmatrix}
        2 & 2 & 2 \\
        2\lambda_0 & 2\lambda_1 & 2\lambda_2 \\
        \lambda_0^2 & \lambda_1^2 & \lambda_2^2 
    \end{bmatrix}.
\end{align}
We make the same boundary condition and discretization assumptions as in \Cref{sec:analysis}; see \Cref{rmk:linear_bgk_disc}.
The backward Euler and spatial discretization  of \eqref{eqn:linearized-bgk:linearized-model} is then given by
\begin{equation}\label{eqn:linearized_bgk:BE}
    f^{\{k+1\}} + \dt vAf^{\{k+1\}} + \nu\dt f^{\{k+1\}} = f^{\{k\}} + \nu\dt \mN(\bmrho_{f^{\{k+1\}}}).
\end{equation}
where $f^{\{k\}}\approx f_h(t^k)$ and $A:V_{x,h}\to V_{x,h}$ is a skew-symmetric operator in $L^2(\W_x)$ with purely imaginary eigenvalues $\gamma i$ such that $|\gamma|\leq h_x^{-1}$.

Let $P_{\mpM}:L^2(\Omega_v)\to\text{span}(\{\mpM,v\mpM,v^2\mpM\})$ be given by $P_{\mpM}f=\mN(\bmrho_f)$.
Then $P_\mpM$ is an orthogonal projection with respect to the $\mpM^{-1}$ inner product $\Mip{w}{z} := \vbr{wz\mpM^{-1}}$ and can be extended to an orthogonal projection in $L^2(\W)$ with respect to the inner product $(w,z)_{\mpM}:=(w,z\mpM^{-1})$.
Define $P_\mpM^\perp = I-P_\mpM$.  

\subsection{Convergence analysis}

Source iteration for \eqref{eqn:linearized_bgk:BE} reads: Given $f^\ell$, find $f^{\ell+1}$ such that 
\begin{equation}\label{eqn:linearized_bgk:SI}
    f^{\ell+1} + \dt vAf^{\ell+1} + \nu\dt f^{\ell+1} = f^{\{k\}} + \nu\dt\mN(\bmrho_{f^{\ell}}).
\end{equation}
The convergence theory for SI remains the same as in \Cref{subsec:lin_bgk:SI_theory}; namely, \Cref{prop:lin_bgk:SI_error} symbolically holds for \eqref{eqn:linearized_bgk:SI}.

The HOLO method reads: Given $f^\ell$, find $f^{\ell+1}$ such that
\begin{subequations}\label{eqn:linearized_bgk:HOLO}
\begin{align}
    f^{\ell+1} + \Delta tvAf^{\ell+1} + \nu\dt f^{\ell+1} &= f^{\{k\}} + \nu\dt\mN(\bmrho^{\ell+1}), \\ 
    \bmrho^{\ell+1} + \dt A(B\bmrho^{\ell+1}) &= \bmrho_{f^{\{k\}}} - \dt A\vbr{\bme v(f^\ell-\mN(\bmrho_{f^{\ell}}))}. \label{eqn:linearized_bgk:lo}
\end{align}
\end{subequations}
Just as in \Cref{sec:analysis}, the linearity and lack of velocity discretization makes the HOLO and MM-HOLO methods equivalent.

\begin{prop}\label{prop:linearized-bgk:HOLO-est}
    Let $e_f^{\ell+1} = f^{\ell+1} - f^{\{k+1\}}$ where $f^{\ell+1}$ and $f^{\{k+1\}}$ are defined respectively in \eqref{eqn:linearized_bgk:HOLO} and \eqref{eqn:linearized_bgk:BE}.
    Then
    \begin{equation}\label{prop:linearized-bgk:HOLO-est:0}
        \|P_\mpM^\perp e_f^{\ell+1}\|_\mpM^2 \leq \tfrac{1}{4}C_{\mathrm{HL}}\tfrac{\nu\dt}{1+\nu\dt}\|P_\mpM^\perp e_f^{\ell}\|_\mpM^2,
    \end{equation}
    where $C_{\mathrm{HL}} = C_0 + C_1 + C_2$, with
    \begin{align}
         C_j = \max_{|\gamma|<\sfrac{\dt}{h_x}}\mathcal{J}_j(\gamma), \label{eqn:linearized-bgk:C_L}
\end{align}
and
\begin{subequations}
   \label{eqn:linearized-bgk:constants}
\begin{align}
        \mathcal{J}_0(\gamma) &= \tfrac{15\theta_0^3\gamma^{6}}{1+3(u_0^2+2\theta_0)\gamma^2+3(u_0^4+3\theta_0^2)\gamma^4+u_0^2(u_0+\sqrt{3\theta_0})^2(u_0-\sqrt{3\theta_0})^2\gamma^6}, \label{eqn:linearized-bgk:C_0} \\
        \mathcal{J}_1(\gamma) &= \tfrac{15\theta_0^2\gamma^{4}(1-u_0\gamma)^2}{1+3(u_0^2+2\theta_0)\gamma^2+3(u_0^4+3\theta_0^2)\gamma^4+u_0^2(u_0+\sqrt{3\theta_0})^2(u_0-\sqrt{3\theta_0})^2\gamma^6}, \\
        \mathcal{J}_2(\gamma) &= \tfrac{7.5\theta_0\gamma^{2}(1-2u_0\gamma+(u_0^2-\theta_0)\gamma^2)^2}{1+3(u_0^2+2\theta_0)\gamma^2+3(u_0^4+3\theta_0^2)\gamma^4+u_0^2(u_0+\sqrt{3\theta_0})^2(u_0-\sqrt{3\theta_0})^2\gamma^6}.
    \end{align}
    \end{subequations}
\end{prop}

\begin{figure}
    \centering
    \begin{subfigure}{0.46\textwidth}
        \rotatebox{90}{\parbox{3cm}{\centering\footnotesize~\hspace{1.5cm}$C_{\mathrm{HL}}/4$}}
        \includegraphics[width=0.925\textwidth]{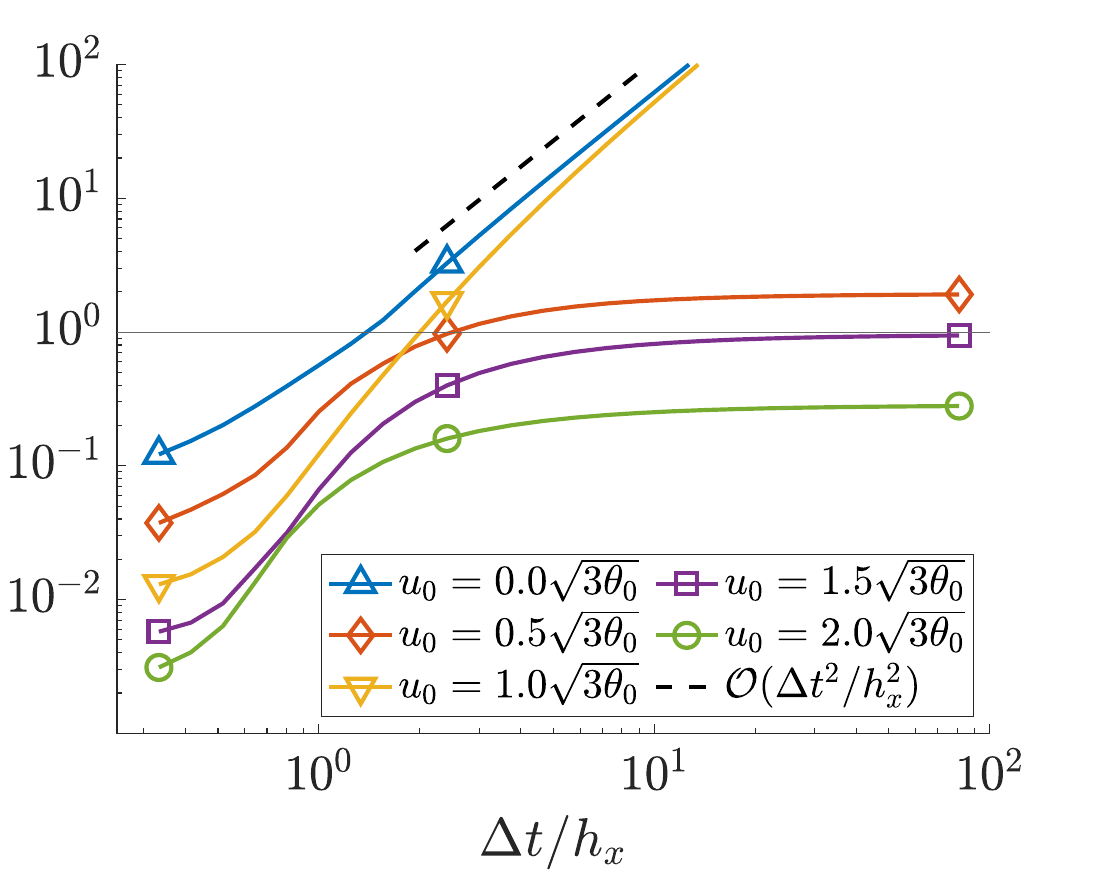}
        \caption{Plot of $C_{\mathrm{HL}}/4$}
        \label{fig:linearized-bgk:constants:C_L}
    \end{subfigure}
    \hspace{20pt}
    \begin{subfigure}{0.46\textwidth}
        \includegraphics[width=\textwidth]{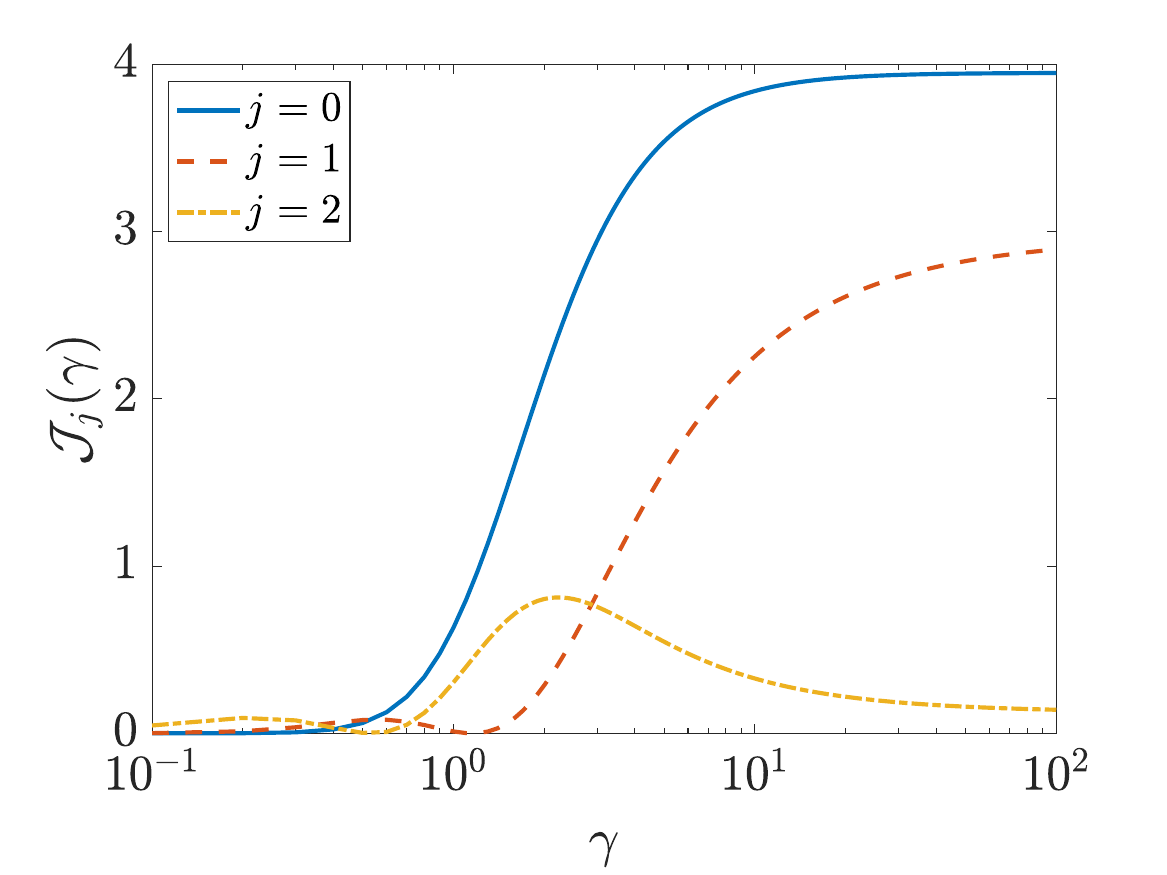}
        \caption{Plot of $\mathcal{J}_i$ where $u_0 = 0.5\sqrt{3\theta_0}$}
        \label{fig:linearized-bgk:constants:Js}
    \end{subfigure}

    \caption{
    Plots of the contraction constant and objective functions in \Cref{prop:linearized-bgk:HOLO-est}.  Here $\theta_0=1$.
    }
    \label{fig:linearized-bgk:constants}
\end{figure}

Before proving \Cref{prop:linearized-bgk:HOLO-est}, we first provide some remarks to give context to the results.
\begin{itemize}
    \item From \eqref{prop:linearized-bgk:HOLO-est:0}, the convergence of the HOLO method is guaranteed for any $\nu$ if $C_{\mathrm{HL}}/4 < 1$.  
    In this sense, the result of \Cref{prop:linearized-bgk:HOLO-est} is similar to the result of  \Cref{prop:lin_bgk:HOLO_error}.
    \item In \Cref{prop:lin_bgk:HOLO_error} if $u_0=0$, then $C_{\tHL}=\mathcal{O}(\sfrac{\dt^2}{h_x^2})$.
    A similar condition holds for \Cref{prop:linearized-bgk:HOLO-est}.
    The worst case scenario for bounding $\mathcal{J}_j$ in \eqref{eqn:linearized-bgk:constants} occurs when
    $u_0=0$ or $u = \pm\sqrt{3\theta_0}$.  In this case, the sixth-order term in the denominators of $\mathcal{J}_j$ vanish, and we expect $C_0$ and consequently $C_{\mathrm{HL}}$ to be $\mathcal{O}(\dt^2/h_x^2)$.
    To verify this claim,
    \Cref{fig:linearized-bgk:constants:C_L} plots $C_{\mathrm{HL}}/4$ as a function of $\dt/h_x$ for $\theta_0=1$ and $u_0$ chosen to be varying multiples of $\sqrt{3\theta_0}$.
    For $u_0=0,\sqrt{3\theta_0}$, $C_{\mathrm{HL}}$ is confirmed to be $\mathcal{O}(\dt^2/h_x^2)$.

    \item For other choices of $u_0$, each $\mathcal{J}_j$ is bounded, but not uniformly.  Unfortunately, we are unable to derive a clean analytic bound $C_{\tHL}$.   Instead, we plot $C_{\tHL}$ as a functions of $\dt/h_x$ for specific values of $u_0$ in \Cref{fig:linearized-bgk:constants:C_L}.
    If $u_0\in\{\tfrac{1}{2}\sqrt{3\theta_0},\tfrac{3}{2}\sqrt{3\theta_0},2\sqrt{3\theta_0}\}$, then $C_{\mathrm{HL}}$ plateaus for large $\dt/h_x$.
    However, since $C_{\mathrm{HL}}/4 > 1$ for $\dt/h_x \geq 3$ in the case of $u_0=\tfrac{1}{2}\sqrt{3\theta_0}$, unconditional convergence of the HOLO method with respect to $\nu$, $\dt$, and $h_x$ is only guaranteed for certain choices of $u_0$ and $\theta_0$.
    
    \item To explore which constant $C_j$ in \eqref{eqn:linearized-bgk:C_L} is the dominant contribution to $C_{\mathrm{HL}}$, we plot $\mathcal{J}_j$ for $j\in\{0,1,2\}$ in \Cref{fig:linearized-bgk:constants:Js} for a specific example of $\theta_0=1$ and $u_0=\tfrac{1}{2}\sqrt{3\theta_0}$.
    For $\gamma\geq 1$, $\mathcal{J}_0$ dominates the contribution to $C_{\mathrm{HL}}$.
    Additionally, while $\mathcal{J}_0$ and $\mathcal{J}_1$ are maximized at their large $\gamma$ plateaus, $\mathcal{J}_2$ obtains its maximum around $\gamma=2$.
\end{itemize}  

\begin{proof}[Proof of \Cref{prop:linearized-bgk:HOLO-est}]
By linearity, error analysis of the HOLO method is equivalent to stability analysis when $f^{\{k\}} = f^{\{k+1\}} = 0$.
Following the proof of \Cref{prop:lin_bgk:HOLO_error}, we have the analog of \eqref{eqn:lin_bgk:HOLO_error:3} with $\delta=2$:
\begin{equation}
        (1+\nu\dt)\|P_\mpM^\perp f^{\ell+1}\|_\mpM^2 \leq \tfrac{\nu\dt}{4}\|\mN(\bmrho^{\ell+1})\|_\mpM^2.
\end{equation}
Direct calculation yields
\begin{align}\label{eqn:linearized_bgk:L_rho_bnd}
\begin{split}
    \|\mN(\bmrho)\|_\mpM^2 &= \|\rho_0\|_{\W_x}^2\vbr{\mpM} + \tfrac{\|\rho_1-u_0\rho_0\|_{\W_x}^2}{\theta_0^2}\vbr{(v-u_0)^2\mpM} \\
    &\quad+ \tfrac{\|2\rho_2-2u_0\rho_1+(u_0^2-\theta_0)\rho_0\|_{\W_x}^2}{\theta_0^2}\vbr{(\tfrac{(v-u_0)^2}{2\theta_0}-\tfrac{1}{2})^2\mpM} \\
    &= \|\rho_0\|_{\W_x}^2 + \tfrac{\|\rho_1-u_0\rho_0\|_{\W_x}^2}{\theta_0} + \tfrac{\|2\rho_2-2u_0\rho_1+(u_0^2-\theta_0)\rho_0\|_{\W_x}^2}{2\theta_0^2} \\
    &:= I_0 + I_1 + I_2.
\end{split}
\end{align}

We seek to bound each $I_i$ using \eqref{eqn:linearized_bgk:lo}.
Note that only the last component, $s_2$, of $\bms:=\vbr{\bme v(f^\ell-\mN(\bmrho_{f^{\ell}}))}$ in \eqref{eqn:linearized_bgk:lo} is non-zero.  Thus we calculate $P^{-1}\bms = s_2\theta_0^{-1}\bmmu$ where $\bmmu=[-1/3,1/6,1/6]^\top$, 
\begin{equation}\label{eqn:linearized_bgk:s_def}
    s_2 = \vbr{\tfrac{1}{2}v^3(f^\ell-\mN(\bmrho_f^\ell))} = \vbr{\tfrac{1}{2}v^3P_\mpM^\perp f^\ell} = \vbr{\tfrac{1}{2}(v-u_0)^3P_\mpM^\perp f^\ell},
\end{equation}
and $P$ is given in \eqref{eq:P}.
The last equality in \eqref{eqn:linearized_bgk:s_def} holds because $v^3$ and $(v-u_0)^3$ differ only by lower-order terms in $v$ that are collision invariants.  We bound $s_2$ by
\begin{equation}\label{eqn:linearized-bgk:s_bound}
    \|s_2\|_{\Omega_x}^2 
    \leq \tfrac{1}{4}\vbr{(v-u_0)^6\mpM}\|P_\mpM^\perp f^\ell\|_\mpM^2 
    = \tfrac{15}{4}\theta_0^3\|P_\mpM^\perp f^\ell\|_\mpM^2.
\end{equation}
Next, setting $\bmeta=P^{-1}\bmrho^{\ell+1}$ and multiplying \eqref{eqn:linearized_bgk:lo} by $P^{-1}$ yields
\begin{equation}
    \eta_j + \lambda_j\dt A\eta_j = -\tfrac{\mu_j}{\theta_0}\dt As_2 
    \implies
    \eta_j = -\tfrac{\mu_j}{\theta_0}(I+\lambda_j\dt A)^{-1}\dt As_2.
\end{equation}
For $j=0,1,2$, define
\begin{equation}
    A^\dagger_j = \textstyle-\sum_{m=0}^2 \frac{\mu_m}{\theta_0}(I+\lambda_m\dt A)^{-1}(\lambda_m)^j\dt A.
\end{equation}
Since $\bmrho^{\ell+1} = P\bmeta$, it follows that
\begin{equation}
\label{eq:rho-to-s2}
\bmrho^{\ell+1} = [2A_0^\dagger s_2,2A_1^\dagger s_2,A_2^\dagger s_2]^\top.
\end{equation}
We use \eqref{eq:rho-to-s2} and \eqref{eqn:linearized-bgk:s_bound} to bound each $I_j$ in \eqref{eqn:linearized_bgk:L_rho_bnd}:
\begin{subequations}\label{eqn:linearized_bgk:I_bnd}
\begin{align}
    I_0 &\leq 4\|A_0^\dagger\|_{\W_x}^2\|s_2\|_{\W_x}^2 \leq \|A_0^\ddagger\|^2\|P_\mpM^\perp f^\ell\|_{\mpM}^2, \\
    I_1 &\leq \frac{4\|A_1^\dagger-u_0A_0^\dagger\|^2\|s_2\|_{\W_x}^2}{\theta_0} \leq \|A_1^\ddagger\|^2\|P_\mpM^\perp f^\ell\|_{\mpM}^2, \\
    I_2 &\leq \frac{4\|A_2^\dagger-2u_0A_1^\dagger+(u_0^2-\theta_0)A_0^\dagger\|^2\|s_2\|_{\W_x}^2}{2\theta_0^2}  \leq \|A_2^\ddagger\|^2\|P_\mpM^\perp f^\ell\|_{\mpM}^2,
\end{align}
\end{subequations}
where
\begin{subequations}
\begin{align}
    A_0^\ddagger &:= \sqrt{15\theta_0^3} A_0^\dagger, \\
    A_1^\ddagger &:= \sqrt{15}\theta_0(A_1^\dagger-u_0A_0^\dagger), \\
    A_2^\ddagger &:= \sqrt{\frac{15}{2}\theta_0}(A_2^\dagger-2u_0A_1^\dagger+(u_0^2-\theta_0)A_0^\dagger).
\end{align}
\end{subequations}

It remains to bound each $A_j^\ddagger$ in \eqref{eqn:linearized_bgk:I_bnd}.  
Note that $A_j^\dagger$, $A_j^\ddagger$ and $A$ are unitarily similar for all $j$.
Denote the spectrum of a matrix $Z$ by $\sigma(Z)$.
Direct calculation of $\sigma(A_j^\dagger)$ yields
\begin{equation} \label{eqn:linearized-bgk:A_dag_eval}
    \sigma(A_j^\dagger) = \Big\{
    \tfrac{(\dt\gamma)^{3-j}}{(1+u_0\dt\gamma i)(1+(u_0+\sqrt{3\theta_0})\dt\gamma i)(1+(u_0-\sqrt{3\theta_0})\dt\gamma i)}:\gamma i\in\sigma(A)\Big\}.
\end{equation}
Moreover, since $A_j^\ddagger$ are linear combinations of $\{A_m^\dagger\}_{m=0}^2$, using \eqref{eqn:linearized-bgk:A_dag_eval} yields
\begin{equation} \label{eqn:linearized-bgk:A_ddag_eval}
    \Big\{ |\gamma|^2:\gamma \in\sigma(A_j^\ddagger)\Big\} = \Big\{ \mathcal{J}_j(\dt\gamma):\gamma i\in\sigma(A)\Big\},
\end{equation}
where $\mathcal{J}_0$, $\mathcal{J}_1$, and $\mathcal{J}_2$ are defined in \eqref{eqn:linearized-bgk:constants}.

Noting that $A_j^\ddagger$ is normal, we use \eqref{eqn:linearized-bgk:A_ddag_eval} and a rescaling to obtain
\begin{equation}
    \|A_j^\ddagger\|^2 = \max_{\gamma\in\sigma(A_j^\ddagger)}|\gamma|^2 = \max_{\gamma i\in\sigma(A)} \mathcal{J}_j(\dt\gamma) \leq \max_{|\gamma|<\sfrac{1\!}{h_x}} \mathcal{J}_j(\dt\gamma) = \max_{|\gamma|<\sfrac{\dt\!}{h_x}} \mathcal{J}_j(\gamma) = C_j,
\end{equation}
where $C_0$, $C_1$ and $C_2$ are defined in \eqref{eqn:linearized-bgk:constants}.
The proof is complete.
\end{proof}

\end{document}